\newif \ifHAL
\HALtrue

\ifHAL
\documentclass[a4paper,11pt]{article}
\usepackage[margin=2.5cm]{geometry}
\usepackage{amsmath,amsthm,amssymb,enumerate, amsbsy}
\usepackage{mathrsfs}  
\usepackage[T1]{fontenc}
\usepackage{gensymb}
\usepackage{longtable}
\usepackage[square,sort&compress,comma,numbers]{natbib}
\usepackage[sc]{mathpazo}
\usepackage{multirow} 
\usepackage{newtxtext,newtxmath}
\usepackage{graphicx}
\usepackage{epstopdf}
\usepackage{cancel}
\usepackage[colorinlistoftodos]{todonotes}
\usepackage{epsfig, setspace}
\usepackage{tikz}
\usepackage[titletoc]{appendix}
\usepackage{caption}
\usetikzlibrary{shapes,calc}
\usepackage{verbatim}
\usepackage{mathrsfs}
\usepackage{accents}
\usepackage[utf8]{inputenc}
\usepackage{float}
\restylefloat{table}
\usetikzlibrary{decorations.pathmorphing}
\usetikzlibrary{decorations.pathreplacing}
\usetikzlibrary{positioning}
\usetikzlibrary{shapes}
\usetikzlibrary{arrows}
\usetikzlibrary{patterns}
\usetikzlibrary{fadings}
\usetikzlibrary{plotmarks}
\usetikzlibrary{calc}
\usetikzlibrary{intersections}
\tikzstyle{every picture}+=[font=\footnotesize]
\usepackage{paralist}
\usepackage{empheq}
\newtheorem{theorem}{Theorem}[section]

\newtheorem{lemma}[theorem]{Lemma}

\theoremstyle{definition}

\theoremstyle{remark}
\newtheorem{remark}{Remark}[section]

\numberwithin{equation}{section}

\usepackage{hyperref}
\hypersetup{allcolors=blue}
\usepackage{anyfontsize}
\usepackage{placeins}

\else
\documentclass[a4paper,fleqn]{cas-sc}
\usepackage[numbers]{natbib}
\def\tsc#1{\csdef{#1}{\textsc{\lowercase{#1}}\xspace}}
\tsc{AE}
\tsc{RM}

\newtheorem{theorem}{Theorem}
\newtheorem{lemma}[theorem]{Lemma}
\newdefinition{remark}{Remark}
\newproof{proof}{Proof}

\fi

\usepackage[linesnumbered,ruled,vlined]{algorithm2e}
\usepackage{xcolor}  
\usepackage{colortbl}
\usepackage{subfig}
\usepackage{wrapfig}

\usepackage{multirow,diagbox,tabularx,blindtext,tabulary,tabularht,longtable,blkarray}

\newcommand{\Qvec}{{\boldsymbol Q}}
\newcommand{\Mvec}{{\boldsymbol M}}
\newcommand{\dx}{\,{\rm dx}}
\newcommand{\abs}[1]{\left\lvert#1\right\rvert}
\newcommand{\norm}[1]{{\lvert\kern-0.25ex \lvert #1 \rvert\kern-0.25ex \rvert}}
\newcommand{\disr}{r_\calT}
\newcommand{\disp}{p_\calT}
\newcommand{\distr}{{\tilde r}_\calT}
\newcommand{\distp}{{\tilde p}_\calT}
\newcommand{\disK}{K_\calT}
\newcommand{\disM}{M_\calT}
\newcommand{\disLamb}{\lambda_\calT}

\newcommand{\Real}{\mathbb{R}}
\newcommand{\polP}{\mathbb{P}}
\newcommand{\bbI}{\mathbb{I}}

\newcommand{\ba}{{\boldsymbol a}}
\newcommand{\bb}{{\boldsymbol b}}
\newcommand{\be}{{\boldsymbol e}}
\newcommand{\bn}{{\boldsymbol n}}
\newcommand{\bt}{{\boldsymbol t}}
\newcommand{\bv}{{\boldsymbol v}}
\newcommand{\bw}{{\boldsymbol w}}
\newcommand{\bx}{{\boldsymbol x}}
\newcommand{\by}{{\boldsymbol y}}
\newcommand{\bB}{{\boldsymbol B}}
\newcommand{\bK}{{\boldsymbol K}}
\newcommand{\bR}{{\boldsymbol R}}
\newcommand{\bS}{{\boldsymbol S}}

\newcommand{\bzero}{{\boldsymbol 0}}
\newcommand{\bnu}{{\boldsymbol \nu}}
\newcommand{\btau}{{\boldsymbol \tau}}
\newcommand{\calE}{\mathcal{E}}
\newcommand{\calG}{\mathcal{G}}
\newcommand{\calN}{\mathcal{N}}
\newcommand{\calT}{\mathcal{T}}
\newcommand{\calV}{\mathcal{V}}
\newcommand{\tr}{^{\mathrm{T}}}
\DeclareMathOperator*{\argmin}{arg\,min}
\newcommand{\SCAL}{{\cdot}}
\newcommand{\vertex}{\calV_\calT}
\newcommand{\vertexi}{\calV_\calT^\circ}
\newcommand{\vertexb}{\calV_\calT^\partial}
\newcommand{\sfi}{\mathsf{i}}

\newcommand{\pho}{\phantom{1}}

\ifHAL \newcommand{\rev}[1]{#1}
\else \newcommand{\rev}[1]{\textcolor{blue}{#1}} \fi

\ifHAL
\title{An energy-decreasing algorithm for the finite element approximation of ferronematic equilibrium states \rev{and its numerical realization}}
\fi

\begin{document}

\ifHAL
\author{ \stepcounter{footnote} \stepcounter{footnote}  Alexandre Ern\footnote{CERMICS, CNRS, ENPC, Institut Polytechnique de Paris, 6 \& 8 avenue B. Pascal, 77455 Marne-la-Vallée, France, and Inria, 48 Rue Barrault, 75647 Paris, France. Email: alexandre.ern@enpc.fr.}  $\;\;$    
Ruma R. Maity\footnote{Department of Mathematics, Indian Institute of Science, Bangalore, 560012, India and Engineering Mathematics, University of Innsbruck, 6020 Innsbruck, Austria. Email: rumaranim@iisc.ac.in.} 
}

\maketitle
\fi

\ifHAL
\begin{abstract}
We develop an energy-decreasing algorithm for the finite element approximation of two-dimensional ferronematic equilibrium states. The problem is formulated as the minimization of the harmonic energy of two two-dimensional vector fields, both with prescribed length, together with an additional nonlinear relation on the orientation of the two vectors. The finite element setting is based on piecewise continuous finite elements on a weakly acute triangulation. The numerical realization of the energy-decreasing algorithm employs a decomposition-coordination framework and a Uzawa-like iteration. Numerical experiments are presented to illustrate the computational performances of the algorithm.
\end{abstract}


\medskip


\fi

\section{Introduction}

Nematic liquid crystals are a physical state of a material where the constituent molecules exhibit long-range orientational ordering but no long-range positional ordering  \cite{deGPr:95}.
The study of composite systems consisting of \rev{dilute} magnetic nanoparticles in a nematic liquid crystal is commonly referred to as ferronematics in the literature \cite{Brochard_DE_Gennes_1970,Burylov2013,Mertelj_Lisjak_Drofenik_Copic_2013}.  
The nemato-magnetic coupling enhances magnetic susceptibility, enabling low-field control of material behavior with potential applications in optics, sensors, and telecommunications  \cite{Hess2015,Katona2014}. 
Ferronematic equilibrium states are sought as minimizers of a suitable free energy density of the system, which generally consists of the Landau--de Gennes free energy density for the nematic liquid crystal, plus the Ginzburg--Landau free energy density for the magnetization, and the nemato-magnetic coupling energy density (see, e.g.,  \cite{Ferronematics_2D}    and the references cited therein). 

We focus here on a two-dimensional setting, where the ferronematic system \rev{is confined in a bounded domain} $\Omega \subset \mathbb{R}^2$. In this case, the two variables describing the state of the system \rev{at the macroscopic level} are: (i) The nematic liquid crystal tensor, $\Qvec$, a mapping from $\Omega$ to the space of matrices in $\Real^{2\times 2}$; (ii) The magnetization vector, $\Mvec$, a mapping from $\Omega$ to $\Real^2$. \rev{The tensor $\Qvec$ describes the macroscopic orientational ordering of the molecules and is required to be symmetric and traceless (see, e.g., \cite{Mottram2014IntroductionTQ}). The vector $\Mvec$ describes the spatially averaged magnetic moment of the suspended ferromagnetic nanoparticles. We note that the vector $\Mvec$ conceptually differs from the local (pointwise) magnetization vector considered in Maxwell's equations. The symmetry and traceless requirement on the tensor $\Qvec$ imply that $Q_{21}=Q_{12}$ and $Q_{22}=-Q_{11}$, so that the tensor $\Qvec$ is fully described by its two components $Q_{11}$ and $Q_{12}$. Thus, there are four independent macroscopic variables, which can be conveniently collected in the $\Real^4$-valued vector 
\begin{equation}
\Psi:=(Q_{11},Q_{12},M_1,M_2)\tr.
\end{equation} 
We note that another (equivalent) writing of the macroscopic orientational tensor is 
\begin{equation} \label{eq:Qvec}
\Qvec=s\left(2\bn\otimes\bn-\mathbb{I}\right),
\end{equation} 
where $s>0$ is a real number, $\mathbb{I}$ is the identity matrix, and $\bn$ is a unit vector in $\Real^2$. Here, $s$ measures the degree of orientational ordering of the constituent molecules, and $\bn$ is the so-called nematic director. This formula shows that the two eigenvectors of $\Qvec$ are $\bn$ (with eigenvalue $s$) and the vector orthogonal to $\bn$ (with eigenvalue $-s$).}

\rev{
Following \cite{Ferromagnetic,Han2021_ferro2D,Dalby2022} and the references therein, the free energy of the dilute suspension of magnetic nanoparticles we consider in this work is given by
\begin{equation}
\calE_\ell(\Psi):= E(\Psi) + \frac{1}{\ell} \int_\Omega f_{\mathrm{B}}(\Psi)\dx,
\end{equation}
with the elastic (harmonic) energy 
\begin{equation}
E(\Psi):= \frac12 \int_\Omega \big\{ |\nabla Q_{11}|^2 + |\nabla Q_{12}|^2 + |\nabla M_1|^2 + |\nabla M_2|^2 \big\} \dx,
\end{equation}
and the bulk coupling energy density
\begin{equation} \label{eq:bulk}
f_{\mathrm{B}}(\Psi) := \frac14\big( Q_{11}^2+Q_{12}^2-1\big)^2 + \frac14\big( M_{1}^2+M_{2}^2-1\big)^2 - \frac{c}{2} \big( Q_{11}(M_1^2-M_2^2)+2Q_{12}M_1M_2\big).
\end{equation}
Here, for simplicity we assumed, as in \cite{FerroRMAMNN2021}, a unit magnetic coherence parameter weighting the relative importance of nematic and magnetic elastic energies, and we introduced a single rescaled elastic constant $\ell$. The bulk coupling energy density involves a single (nondimensional) parameter $c$ called nemato-magnetic coupling parameter. In this work, we assume that $c>0$ so that the last term in~\eqref{eq:bulk} (which rewrites $\Mvec\tr \Qvec \Mvec$) favors the alignment $\Mvec{\cdot}\bn=\pm1$ when minimizing the free energy, recalling from~\eqref{eq:Qvec} that $\bn$ is the eigenvector of $\Qvec$ with eigenvalue $s$ (whereas $c<0$ favors the alignment of $\Mvec$ with the perpendicular eigenvector with eigenvalue $-s$). Owing to the presence of the elastic energy, sign changes in the alignment property $\Mvec{\cdot}\bn=\pm1$ are not favored. We note in passing that the macroscopic variables collected in $\Psi$ have finite elastic energy so that they are essentially continuous in $\Omega$.
} 

\rev{The study of equilibrium states of ferronematic crystals as minimizers of the free energy $\calE_\ell$ (under suitable boundary conditions on $\partial\Omega$) has already been explored in the literature \cite{Ferronematics_2D,Dalby2022,FerroRMAMNN2021,Han2021_ferro2D}.} Typically, one looks for a critical point of the free energy by solving the corresponding Euler conditions which leads to a system of four coupled nonlinear PDEs to be solved in $\Omega$ (again with suitable boundary conditions). In this work, we are specifically interested in the practically important case where the rescaled elastic constant $\ell$ takes (very) small values. \rev{The limit process as $\ell\to0$ has been studied in \cite{Ferronematics_2D,FerroRMAMNN2021}. Calling $\Psi_\ell$ a minimizer of $\calE_\ell$, one can show that the sequence $(\Psi_\ell)_{\ell>0}$ admits a subsequence that tends to a minimizer $\Psi_*$ of the elastic (harmonic) energy $E$ alone, under the additional constraint that $\Psi_*$ belongs to a one-dimensional manifold $\calN \subset \Real^4$ which is identified in \cite[Lemma 2.1]{FerroRMAMNN2021} by minimizing the bulk coupling energy density. Specifically, one finds that a minimizer of $f_{\mathrm{B}}$ is of the form 
\begin{equation} \label{eq:Psi_B}
\Psi_{\mathrm{B}}=(Q_c\cos(2\varphi),Q_c\sin(2\varphi),M_c\cos(\varphi),M_c\sin(\varphi))\tr,
\end{equation} 
where $\varphi\in\Real$ and the positive real numbers $Q_c$ and $M_c$ can be explicitly evaluated from the nemato-magnetic coupling parameter $c$ (as detailed below in Section~\ref{sec:cont}). The constraint manifold $\calN$ introduced above is then composed of all the fields $\Psi$ of the form~\eqref{eq:Psi_B}.}

In this work, we propose an algorithm to compute iteratively a constrained minimizer of the elastic (harmonic) energy in the setting of a piecewise affine, continuous finite element approximation. \rev{The proposed algorithm generalizes the known approach for the Oseen--Frank model where a single unit vector is considered (see below for further discussion) to the setting of two unit vectors coupled by a nonlinear constraint. Although the present ferronematic problem is amenable to a simplification at the theoretical level (see Lemma~\ref{lem:norm_dG} below), we use this problem as a groundbase to develop our algorithm. The paper contains two main contributions.} \rev{First, at the theoretical level,} we propose an iterative method for which we guarantee energy decrease at each step under the assumption of a weakly acute triangulation \cite{Brandts2009,Qiang:1998:GL,Xu1999}. \rev{Second, we study one possible computational realization of our algorithm relying on an additional inner iteration. Specifically,} we adopt a decomposition-coordination approach in the spirit of Fortin and Glowinski's framework for constrained optimization with separable objective functions and coupled constraints  \cite{Fortin:Glowinski:1983}. The constraint is incorporated using a Lagrange multiplier within an augmented Lagrangian approach \cite{Bertsekas:1982,Bertsekas:1999}.  A \rev{critical point} of the Lagrangian is sought using an Uzawa-type iterative algorithm \cite{Arrow:Hurwicz:Uzawa:1958,Fortin:Glowinski:1983}. \rev{We do not provide here a convergence guarantee on the inner iteration, but extensive numerical tests on basic to more challenging test cases indicate a robust behavior of the inner iteration, and that the overall methodology is computationally effective.}

\rev{Devising an energy-decreasing iterative algorithm for the minimization of the harmonic energy of two unit vectors coupled by a nonlinear constraint represents in itself a conceptual challenge, and to the best of the authors' knowledge, such an algorithm is not yet available in the literature. We hope that the present work provides a groundbase for future research in this direction. In contrast,} the minimization of the elastic (harmonic) energy under prescribed length constraints on a single vector field  is already well explored in the literature. The simplest setting is the Oseen--Frank model which minimizes the harmonic energy of a unit direction vector \cite{Alouges:1997:Harmonic_map,Cohen:Lin:Luskin,Lin:Luskin:1989}. We also mention \cite{Pierre2004}  for further computational enhancements,  \cite{Bartels:2005:FEscheme_harmonic_maps} for the stability and convergence analysis, \cite{Bartels:2007:harmonic-map-flow} for the implicit time discretization of harmonic map flows into spheres and \cite{Bartels:2014:Q-tensor,Borthagaray2020,Nochetto:2017:FEM} for extensions to more complex models including tensor fields, such as the Eriksen model. 

\rev{The present energy-decreasing iterative algorithm hinges on the following key ideas, which loosely follow and extend those devised by Alouges in \cite{Alouges:1997:Harmonic_map} for the Oseen--Frank model. First, at the continuous level,} we leverage a representation of every element in the constraint manifold $\calN$ by the nematic director $\bn$ which lives in the unit circle $\bS$ in $\Real^2$. The first challenge is that Lemma~\ref{lem:norm_dG} is not available at the discrete level. The first key 
idea forward is to replace it by the identity~\eqref{eq:engy_eval}, which evaluates the harmonic energy of a discrete field $\Psi_\calT$ by a finite sum of positive contributions resulting from local increments of unit vectors associated with pairs of mesh nodes sharing a mesh edge. This is where the weakly acute property of the triangulation is invoked.  However, two unit vectors per mesh node need to be considered (and not just one), the two being related by a nonlinear mapping from $\bS$ onto itself. The second key idea forward is achieved by Lemma~\ref{lem:phi} which uncovers a subtle geometric property of the constraint manifold $\calN$. The result of Lemma~\ref{lem:phi} brings the central design condition that guarantees the energy-decreasing property of the present iterative algorithm. 
As a final remark, we observe that the constraint of taking values in $\calN$ is enforced at the mesh nodes. This means that we are considering a nonconforming approximation as the discrete fields do not take values in $\calN$ everywhere in $\Omega$ (but only at the mesh nodes). \rev{Our extensive numerical experiments confirm that this nonconformity is under control with mesh refinement.}

The remainder of this paper is organized as follows. Section~\ref{sec:cont} \rev{discusses the main properties of} the admissible constraint manifold $\calN$ and formulates the energy minimization problem. For reference,  this section also outlines the energy-decreasing algorithm at the continuous level. 
Section~\ref{sec:FEM} develops the finite element approximation, leading to the discrete energy-decreasing algorithm. The main  theoretical results of this section are Lemma~\ref{lem:phi} and Theorem~\ref{thm:engy-decr-disc}.
Section~\ref{sec:algo} describes \rev{a numerical realization of the energy-decreasing iterative algorithm, which hinges on} a decomposition-coordination approach and Uzawa-like iterations. Finally, Section~\ref{sec:results} reports numerical experiments on test cases illustrating the computational performance of the proposed overall methodology. \rev{We consider first a baseline problem to study computational aspects, and then test the methodology on more challenging problems including variations on the nemato-magnetic coupling parameter, higher winding numbers, domains with holes, and multiple solutions. A comparison with an established method from the literature (projected-gradient iteration) is also included.}

\section{Continuous setting}
\label{sec:cont}

For positive integer numbers $p,q$, $\|\bullet\|_{\Real^q}$ denotes the Euclidean norm in $\Real^q$ and $\|\bullet\|_{\Real^{p\times q}}$ denotes the Frobenius norm in $\Real^{p\times q}$.
The Cartesian components of a vector $\Psi\in\Real^q$ are denoted $(\Psi_i)_{1\le i\le q}$.
\rev{With the above notation, the elastic (harmonic) energy $E : H^1(\Omega;\Real^4) \rightarrow \Real$ can be rewritten as follows: For all $\Psi:=(\Qvec,\Mvec) \in H^1(\Omega;\Real^4)$,}
\begin{equation} \label{eq:energy}
E(\Psi) := \frac12 \int_\Omega \norm{\nabla \Psi}^2_{\Real^{4\times2}} \dx = \frac12 \int_\Omega \norm{\nabla \Qvec}^2_{\Real^{2\times2}} \dx + \frac12 \int_\Omega \norm{\nabla \Mvec}^2_{\Real^{2\times2}} \dx.
\end{equation}

\subsection{Admissible manifold}

Let $\bS$ be the unit sphere (circle) in $\Real^2,$  $\be_1:=(1,0)\tr$ the first vector of the Cartesian basis of $\Real^2$, and $\bbI$ the $2\times 2$ identity matrix. We consider the (nonlinear) map
\begin{equation}
\bR : \bS \rightarrow \bS, \quad \bn \mapsto \bR(\bn) := (2\bn\otimes\bn-\bbI)\be_1.
\end{equation}
An elementary computation shows that, if $\bn\in\bS$ is parametrized as $\bn:=(\cos(\varphi),\sin(\varphi))\tr$, then $\bR(\bn)=(\cos(2\varphi),\sin(2\varphi))\tr$. 
\rev{Let $c>0$ be the value of the nemato-magnetic coupling parameter, and set
$Q_c:=\big(\frac{c}{4}+\chi\big)^{\frac13}+\big(\frac{c}{4}-\chi\big)^{\frac13}$ and $M_c:=(1+cQ_c)^{\frac12}$
with $\chi:=\big(\frac{c^2}{16}-\frac{1}{27}(1+\frac{c^2}{2})^3\big)^{\frac12}$. W}e 
define the (nonlinear) map
\begin{equation} \label{eq:calG}
\calG : \bS \rightarrow \Real^4, \quad \bn \mapsto \calG(\bn):= (Q_c\bR(\bn),M_c\bn)\tr.
\end{equation}
It is readily seen that $\calG : \bS \rightarrow \calN$ is a smooth diffeomorphism, where
\begin{equation}\label{equn:manifold}
\calN := \{ \calG(\bn)  \;|\; \bn \in \bS \} \subset \Real^4,
\end{equation}
\rev{and $\calN$ corresponds to the constraint minimization set identified in the introduction.}
This set is a connected, compact, one-dimensional smooth submanifold of $\Real^4$. 

For all $\bn\in\bS$, let $T_{\bn}\bS$ denote the tangent space of $\bS$ at $\bn$, i.e., $T_{\bn}\bS=\{\bv\in\Real^2\;|\;\bv\SCAL\bn=0\}$. Similarly, let $T_{\calG(\bn)}\calN$ denote the tangent space of $\calN$ at $\calG(\bn)$. The differential of $\calG$ at $\bn\in\bS$, $d\calG(\bn): T_{\bn}\bS \rightarrow T_{\calG(\bn)}\calN$ is given by, for all $\bn=(n_1,n_2)\tr \in\bS$ and all $\bv=(v_1,v_2)\tr \in T_{\bn}\bS$,
\begin{equation} \label{eq:diff_calG}
\langle d\calG(\bn),\bv \rangle = 
\begin{pmatrix}
2Q_c(v_1 \bn+n_1 \bv) \\ M_c\bv 
\end{pmatrix}
= 
\begin{pmatrix}
4Q_cn_1v_1 \\
2Q_c(n_1v_2+n_2v_1) \\
M_cv_1 \\
M_cv_2
\end{pmatrix}.
\end{equation}

\begin{lemma}[Norm] \label{lem:norm_dG}
We have $\|\langle d\calG(\bn),\bv \rangle\|_{\Real^4}^2=(4Q_c^2+M_c^2)\|\bv\|_{\Real^2}^2$.
\end{lemma}

\begin{proof}
Since $v_1n_1+v_2n_2=0$, the first component of $\langle d\calG(\bn),\bv \rangle$
rewrites $2Q_c(n_1v_1-n_2v_2)$. Since $(n_1v_1-n_2v_2)^2+(n_1v_2+n_2v_1)^2=(n_1^2+n_2^2)(v_1^2+v_2^2) = \|\bn\|_{\Real^2}^2 \|\bv\|_{\Real^2}^2=\|\bv\|_{\Real^2}^2$, the conclusion is straightforward.
\end{proof}

\begin{lemma}[Composition using $\calG$] \label{lem:compose}
\textup{(i)} Let $\bn\in H^1(\Omega;\Real^2)$ take values in $\bS$ a.e.~in $\Omega$; then $\Psi:=\calG \circ \bn \in H^1(\Omega;\Real^4)$ and takes values in $\calN$ a.e.~in $\Omega$.
\textup{(ii)} Let $\Psi \in H^1(\Omega;\Real^4)$ take values in $\calN$ a.e.~in $\Omega$; 
then, $\bn := \calG^{-1} \circ \Psi \in H^1(\Omega;\Real^2)$ and takes values in $\bS$ a.e.~in $\Omega$.
\end{lemma}

\begin{proof}
(1) Proof of (i). By construction, $\Psi$ takes values in $\calN$ a.e.~in $\Omega$. Since $\calG$ is bounded and $\Omega$ is a bounded set, $\Psi \in L^2(\Omega;\Real^4)$. Letting $\partial_k$ denote the partial derivative in the $k$-th Cartesian direction, for all $k\in\{1,2\}$, we observe that 
\[
\partial_k \Psi = \langle d\calG(\bn),\partial_k\bn\rangle.
\] 
Owing to Lemma~\ref{lem:norm_dG}, we infer that 
\begin{equation} \label{eq:norm_psi}
\|\partial_k \Psi\|_{\Real^4}^2 = (4Q_c^2+M_c^2)\|\partial_k\bn\|_{\Real^2}^2 \quad \forall k\in\{1,2\}.
\end{equation} 
Since $\bn\in H^1(\Omega;\Real^2)$ by assumption, this shows that $\Psi \in H^1(\Omega;\Real^4)$.

(2) Proof of (ii). The proof is straightforward since $\bn := \calG^{-1} \circ \Psi = M_c^{-1}(\Psi_3,\Psi_4)\tr$. 
\end{proof}

\subsection{Minimization problem}

Let $\Omega$ be an open, bounded, connected, Lipschitz set (domain) in $\Real^2$. Given some Dirichlet boundary data $\Psi^b \in H^{\frac12}(\partial \Omega;\Real^4)$ taking values in $\calN$ a.e.~in $\partial\Omega$, we consider the functional space
\begin{equation}
W := \{\Phi \in H^1(\Omega;\Real^4)\;|\; \Phi(\bx)\in\calN \;\text{a.e.~$\bx\in\Omega$} \;\wedge \; \Phi|_{\partial\Omega}=\Psi^b\},
\end{equation}
and the elastic (harmonic) energy functional $E : H^1(\Omega;\Real^4) \rightarrow \Real$ defined in~\eqref{eq:energy}. 
We assume that there exists $\Psi^* \in W$.  Sufficient conditions for the Ginzburg--Landau problem are discussed in\ifHAL \ \cite{Bethuel} \fi based on the notion of winding numbers for $\Psi^b.$

Our goal is to solve the following constrained minimization problem:
\begin{equation} \label{eq:inf}
\min_{\Phi \in W} E(\Phi).
\end{equation}
As $E$ is strongly convex and $W$ is a nonempty, closed (but nonconvex) set, the model problem~\eqref{eq:inf} admits at least one solution.

\subsection{Energy-decreasing algorithm}

Before presenting the algorithm, we recall some elementary properties of the 
projection $P_{\bS}$ from $\Real^2_*:= \Real^2\setminus\{\bzero\}$ onto $\bS$, i.e.,
$P_{\bS}(\bx):= \frac{1}{\|\bx\|_{\Real^2}}\bx$ for all $\bx\in \Real^2_*$.
Set 
\begin{equation} \label{eq:outside_ball}
\Real^2_{\ge1}:=\{\bx\in\Real^2 \;|\; \|\bx\|_{\Real^2}\ge1\}.
\end{equation}

\begin{lemma}[Projection onto sphere] \label{lem:PS1}
\textup{(i)} $\|P_{\bS}(\bx)-P_{\bS}(\by)\|_{\Real^2} \le \|\bx-\by\|_{\Real^2}$ for all $\bx,\by \in \Real^2_{\ge1}$.
\textup{(ii)} For all $\bv \in H^1(\Omega;\Real^2)$ taking values in $\Real^2_{\ge1}$, $P_{\bS}(\bv) \in H^1(\Omega;\Real^2)$ and takes values in $\bS$; moreover, $\|\nabla P_{\bS}(\bv)\|_{\Real^{2\times2}} \le \|\nabla\bv\|_{\Real^{2\times2}}$. 
\end{lemma}

\begin{proof}
The statement (i) is just a consequence of the fact that, when acting on $\Real^2_{\ge1}$, $P_{\bS}$ is the projection onto $\bB^1$, the unit ball of $\Real^2$ which is a closed convex set. The statement (ii) follows from the same arguments, see also 
\cite[Prop.~1]{Alouges:1997:Harmonic_map}. 
\end{proof}

Given an initial guess $\Psi^0\in W$, our goal is to generate a sequence $(\Psi^j)_{j=1}^\infty$ in $W$ that is energy-decreasing, i.e.,
\begin{equation} \label{eq:decay}
E(\Psi^{j+1}) \le E(\Psi^j), \quad \forall j\ge0.
\end{equation}
At iteration $j\ge0$, given $\Psi^j\in W$, 
the algorithm proceeds in four steps:
\begin{enumerate}
\item Set $\bn^j:= \calG^{-1}(\Psi^j)$. Owing to Lemma~\ref{lem:compose}(ii), $\bn^j  \in H^1(\Omega;\Real^2)$ and takes values in $\bS$ a.e.~in $\Omega$.
\item Define the linear space $\bK^j := \{\bw \in H^1_0(\Omega;\Real^2)\;|\;
\bw(\bx) \in T_{\bn^j(\bx)}\bS \; \text{a.e. $\bx\in\Omega$}\}$ and solve the constrained minimization problem
\begin{equation} \label{eq:min_step2}
\bw^j := \argmin_{\bw \in \bK^j} \frac12 \int_\Omega \|\nabla (\bn^j-\bw)\|_{\Real^{2\times 2}}^2 \dx.
\end{equation}
This is a well-posed problem since the functional is strongly convex in $H^1_0(\Omega;\Real^2)$ and the minimizing set is a closed convex subset of $H^1_0(\Omega;\Real^2)$. Moreover, since $\bw^j(\bx)\SCAL \bn^j(\bx)=0$ and $\|\bn^j(\bx)\|_{\Real^2}=1$ a.e.~in $\Omega$, we infer that $\bn^j-\bw^j$ takes values in $\Real^2_{\ge1}$ a.e.~in $\Omega$ (observe that $\|\bn^j(\bx)-\bw^j(\bx)\|_{\Real^2}^2 = 1 + \|\bw^j(\bx)\|_{\Real^2}^2 \ge 1$).
\item Set $\bn^{j+1}:= {P_{\bS}(\bn^j-\bw^j)}$. Owing to Lemma~\ref{lem:PS1}, $\bn^{j+1}\in H^1(\Omega;\Real^2)$ and takes values in $\bS$.
\item Set $\Psi^{j+1}:=\calG{(\bn^{j+1})}$. Owing to Lemma~\ref{lem:compose}(i), $\Psi^{j+1}  \in H^1(\Omega;\Real^4)$ and takes values in $\calN$ a.e.~in $\Omega$. Moreover, for a.e.~$\bx\in\partial\Omega$, we have $\bw^j(\bx)=\bzero$ so that $\bn^{j+1}(\bx)=P_{\bS}(\bn^j(\bx))=\bn^j(\bx)$; hence, $\Psi^{j+1}(\bx) = \calG(\calG^{-1}(\Psi^j(\bx))) = \Psi^j(\bx) = \Psi^b(\bx)$. In conclusion, we have $\Psi^{j+1}\in W$ as claimed.
\end{enumerate}

\begin{theorem}[Energy decay] \label{th:decay_cont}
The sequence $(\Psi^j)_{j=1}^\infty$ satisfies \eqref{eq:decay}.
\end{theorem}

\begin{proof}
Owing to~\eqref{eq:norm_psi}, we infer that 
\[
\|\nabla\Psi^{m}\|_{\Real^{4\times 2}}^2
= (4Q_c^2+M_c^2) \|\nabla \bn^m\|_{\Real^{2\times 2}}^2 \quad \forall m\in\{j,j+1\}.
\] 
Moreover, owing to Lemma~\ref{lem:PS1}(ii), we infer that
\[
\int_\Omega \|\nabla \bn^{j+1}\|_{\Real^{2\times2}}^2 \dx \le
\int_\Omega \|\nabla (\bn^{j}-\bw^j)\|_{\Real^{2\times2}}^2 \dx \le
\int_\Omega \|\nabla \bn^{j}\|_{\Real^{2\times2}}^2 \dx,
\]
where the second bound follows from \eqref{eq:min_step2} and the observation that $\bzero \in \bK^j$. Putting everything together gives
\begin{align*}
\int_\Omega \|\nabla\Psi^{j+1}\|_{\Real^{4\times 2}}^2 \dx
&= (4Q_c^2+M_c^2)\int_\Omega \|\nabla \bn^{j+1}\|_{\Real^{2\times 2}}^2 \\
&\le (4Q_c^2+M_c^2)  \int_\Omega \|\nabla \bn^{j}\|_{\Real^{2\times2}}^2 \dx 
=\int_\Omega \|\nabla\Psi^{j}\|_{\Real^{4\times 2}}^2 \dx.
\end{align*}
This completes the proof.
\end{proof}

\section{Finite element approximation}
\label{sec:FEM}

\subsection{Discrete setting}

We assume that $\Omega$ is a polygonal domain in $\Real^2$ and consider a triangulation $\calT$ that covers $\Omega$ exactly. The collection of mesh vertices is partitioned as $\vertex:=\vertexi \cup \vertexb$, where $\vertexi$ (resp., $\vertexb$) denotes the collection of interior (resp., boundary) mesh vertices. The standard hat (Courant) basis functions are denoted as $\{\rho_{\ba}\}_{\ba \in\vertex}$ and satisfy the partition-of-unity property $\sum_{\ba\in \vertex} \rho_{\ba} =1$ over $\Omega$. We set
\begin{equation} \label{eq:def_kab}
k_{\ba\bb} := \int_\Omega \nabla \rho_{\ba} \SCAL \nabla \rho_{\bb} \dx
\quad \forall \ba,\bb \in \vertex,
\end{equation}
and observe that $k_{\ba\bb}=k_{\bb\ba}$ as well as 
$\sum_{\ba\in \vertex} k_{\ba\bb}= \sum_{\bb\in \vertex} k_{\ba\bb} = 0$
owing to the partition-of-unity property. 
In what follows, we assume that the triangulation $\calT$ is weakly acute \cite{Brandts2009,Qiang:1998:GL,Xu1999}  so that
\begin{equation} \label{eq:DMP}
k_{\ba\bb} \le 0 \quad \forall \ba,\bb \in \vertex, \; \ba\ne\bb.
\end{equation}
Recall that this condition is satisfied if and only if the sum of the two angles opposite any edge connecting a pair of mesh vertices is less than or equal to $\pi$. This is a reasonable condition in 2D.

To approximate the minimization problem~\eqref{eq:inf}, we consider the discrete minimizing set
\begin{align}
W_\calT := \{ \Phi_\calT \in H^1(\Omega;\Real^4) \;|\; & \Phi_{\calT}|_T \in \polP^1(T;\Real^4), \, \forall T\in\calT \;;\; \nonumber \\
& \Phi_\calT(\ba)\in \calN,\, \forall \ba\in\vertexi;\;; \Phi_\calT(\ba)=\Psi^b(\ba),\,\forall \ba\in\vertexb\},
\end{align}
where $\polP^1(T;\Real^4)$ is composed of divariate, $\Real^4$-valued, affine polynomials restricted to $T\in\calT$. 
 Every $\Phi_\calT \in W_\calT$ can be expanded as follows:
\begin{equation} \label{eq:expansion}
\Phi_\calT = \sum_{\ba \in\vertex} \Phi_\calT(\ba) \rho_\ba.
\end{equation}
An important observation is the nonconforming nature of the approximation provided by the discrete minimizing set, i.e., we have $W_\calT \not \subset W$. Indeed, we have $\Phi_\calT(\bx) \in \calN$ if $\bx\in\vertex$, but this is not guaranteed if $\bx\in \Omega \setminus \vertex$.
Altogether, the discrete minimization problem reads as follows:
\begin{equation} \label{eq:disc_inf}
\min_{\Phi_\calT\in W_\calT} E(\Phi_\calT),
\end{equation}
with the elastic (harmonic) energy functional still defined in~\eqref{eq:energy}.

\subsection{Preliminary notation and results}

Recalling \eqref{eq:calG}, given any $\Psi=(\Psi_1,\Psi_2,\Psi_3,\Psi_4)\tr \in \calN$,
we write $\bn:=M_c^{-1}(\Psi_3,\Psi_4)\tr$ and $\bnu:=Q_c^{-1}(\Psi_1,\Psi_2)\tr$,
so that $\bn \in \bS$ and $\bnu=\bR(\bn)\in \bS$ as well. Consistently with this notation,
for all $\Phi_\calT \in W_\calT$ and all $\ba\in\vertex$, we define the two vectors
$\bn(\ba)$ and $\bnu(\ba)$, both in $\bS$ and such that $\bnu(\ba)=\bR(\bn(\ba))$.

\begin{lemma}[Energy evaluation] \label{lem:engy_eval}
The following holds for all $\Phi_\calT \in W_\calT$:
\begin{equation} \label{eq:engy_eval}
\int_\Omega \|\nabla \Phi_\calT\|_{\Real^{4\times2}}^2 \dx = - \frac12
\sum_{\ba,\bb\in\vertex} k_{\ba\bb}\Big\{ \, Q_c^2\|\bnu(\ba)-\bnu(\bb)\|_{\Real^2}^2
+ M_c^2\|\bn(\ba)-\bn(\bb)\|_{\Real^2}^2 \, \Big\}.
\end{equation}
\end{lemma}

\begin{proof}
Using the expansion~\eqref{eq:expansion}, the definition~\eqref{eq:def_kab} of $k_{\ba\bb}$ and the properties $\sum_{\ba\in \vertex} k_{\ba\bb}= \sum_{\bb\in \vertex} k_{\ba\bb} = 0$,
we infer that
\begin{align*}
&\int_\Omega \|\nabla \Phi_\calT\|_{\Real^{4\times2}}^2 \dx = \sum_{\ba,\bb\in\vertex}
k_{\ba\bb} \Phi_\calT(\ba) \SCAL \Phi_\calT(\bb) \\
&=\frac12 \sum_{\ba,\bb\in\vertex}
k_{\ba\bb} \Phi_\calT(\ba) \SCAL (\Phi_\calT(\bb)-\Phi_\calT(\ba)) + 
\frac12 \sum_{\ba,\bb\in\vertex}
k_{\ba\bb} \Phi_\calT(\bb) \SCAL (\Phi_\calT(\ba)-\Phi_\calT(\bb))  \\
&=-\frac12 \sum_{\ba,\bb\in\vertex}
k_{\ba\bb} \|\Phi_\calT(\ba)-\Phi_\calT(\bb)\|_{\Real^4}^2,
\end{align*}
and the claim follows from 
$\|\Phi_\calT(\ba)-\Phi_\calT(\bb)\|_{\Real^4}^2 = Q_c^2\|\bnu(\ba)-\bnu(\bb)\|_{\Real^2}^2 + M_c^2\|\bn(\ba)-\bn(\bb)\|_{\Real^2}^2$.
\end{proof}

Our second preliminary result is related to the geometry of the admissible set $\calN$.
We define the function $\varphi:(-1,1)\rightarrow\Real$ such that
\begin{equation} \label{eq:def_phi}
\varphi(r) := \frac{2r}{1-r^2} \quad \forall r\in(-1,1).
\end{equation}

\begin{lemma}[Geometry of $\calN$] \label{lem:phi}
Let $\bn\in\bS$ and let $\bt\in T_{\bn}\bS$ be the unit vector 
obtained from $\bn$ by a rotation
of $\frac{\pi}{2}$ in $\Real^2$. Set $\bnu:=\bR(\bn)\in\bS$ and let 
$\btau\in T_{\bnu}\bS$ be the unit vector obtained from $\bnu$ by a rotation
of $\frac{\pi}{2}$ in $\Real^2$. The following holds for all $r\in (-1,1)$:
\begin{equation} \label{eq:RPS}
(\bR \circ P_{\bS})(\bn + r \bt) = P_{\bS}(\bnu + \varphi(r)\btau),
\end{equation}
with the function $\varphi:(-1,1)\rightarrow\Real$ defined in~\eqref{eq:def_phi}. 
In other words, we have
\begin{equation} \label{eq:RPinN}
\big(Q_c P_{\bS}(\bnu + \varphi(r)\btau),M_cP_{\bS}(\bn + r \bt)\big)\tr \in \calN.
\end{equation}
\end{lemma}

\begin{proof}
We first observe that $\bn + r \bt \in \Real^2_{\ge1}$ and $\bnu + \varphi(r)\btau \in
\Real^2_{\ge1}$, so that both projections $P_{\bS}(\bn + r \bt)$ and
$P_{\bS}(\bnu + \varphi(r)\btau)$ are well-defined. There are many ways of 
verifying~\eqref{eq:RPS}. We use here a representation by complex numbers. 
We represent $\bn$ as $e^{\sfi\alpha}$ with $\sfi^2=-1$, $\bn+r\bt$ as $(1+\sfi r)e^{\sfi\alpha}$, and $P_{\bS}(\bn+r\bt)$ as $z_1:=(1+r^2)^{-\frac12}(1+\sfi r)e^{\sfi\alpha}$. Similarly,
we represent $\bnu$ by $e^{2\sfi\alpha}$ and $P_{\bS}(\bnu + \varphi(r)\btau)$ by
$z_2:=(1+\varphi(r)^2)^{-\frac12}(1+\sfi \varphi(r))e^{2\sfi\alpha}$. 
Since $|z_1|=|z_2|=1$, the identity~\eqref{eq:RPS} amounts to $z_1^2 = z_2$. We infer that
\begin{align*}
z_1^2 = z_2 & \; \Longleftrightarrow \; \frac{(1+\sfi r)^2}{1+r^2} = \frac{1+\sfi \varphi(r)}{\sqrt{1+\varphi(r)^2}} \\
& \; \Longleftrightarrow \; \frac{1-r^2}{1+r^2} = \frac{1}{\sqrt{1+\varphi(r)^2}} \; \wedge \;  \frac{2r}{1+r^2} = \frac{\varphi(r)}{\sqrt{1+\varphi(r)^2}} \\
& \; \Longleftrightarrow \; \varphi(r) := \frac{2r}{1-r^2}.
\end{align*}
This completes the proof.
\end{proof}

\subsection{Discrete energy-decreasing algorithm}

Given an initial guess $\Psi_\calT^0\in W_\calT$, our goal is to generate a sequence $(\Psi_\calT^j)_{j=1}^\infty$ in $W_\calT$ that is energy-decreasing, i.e.,
\begin{equation} \label{eq:enrg_decr_calT}
E(\Psi_\calT^{j+1}) \le E(\Psi_\calT^{j}).
\end{equation}
At iteration $j\ge0$, given $\Psi_\calT^j\in W_\calT$, 
we write $\Psi_\calT^j(\ba)=(Q_c\bnu^j(\ba),M_c\bn^j(\ba))\tr \in \calN$ with 
$\bnu^j(\ba)=\bR(\bn^j(\ba))$, for all $\ba\in\vertex$, and we let $\bt^j(\ba)$ 
(resp., $\btau^j(\ba)$) be the unit vector obtained from $\bn^j(\ba)$ (resp., 
$\bnu^j(\ba)$) by a rotation of $\frac{\pi}{2}$ in $\Real^2$. Set $\disM:=\Real^{\#\vertex}$
where $\#\vertex$ denotes the cardinality of the set $\vertex$.
We define the closed convex subset 
\begin{equation}\label{equn:aux-space}
\disK := \big\{\, \disr:=(r_{\ba})_{\ba\in\vertex} \in \disM \;|\; r_{\ba} \in [-\gamma,\gamma] \;\forall \ba\in \vertexi; \; r_{\ba}=0 \; \forall \ba \in\vertexb \,\big\}, 
\end{equation}
where $\gamma\in (0,1)$ is a user-defined parameter. Then, we define the (nonconvex) 
functional $F^j : \disK \rightarrow \Real$ such that, for all $\disr:=(r_\ba)_{\ba\in\vertex}$, 
\begin{align}
F^j(\disr) := -\frac12 \sum_{\ba,\bb\in\vertex} k_{\ba\bb} \Big\{ \,
&Q_c^2\|(\bnu^j(\ba)+\varphi(r_{\ba})\btau^j(\ba))-(\bnu^j(\bb)+\varphi(r_{\bb})\btau^j(\bb))\|_{\Real^2}^2 \nonumber \\
+ {} &M_c^2\|(\bn^j(\ba)+r_{\ba}\bt^j(\ba))-(\bn^j(\bb)+r_{\bb}\bt^j(\bb))\|_{\Real^2}^2
\, \Big\}. \label{eq:def_Fj}
\end{align} 
This functional is well-defined since $\gamma\in(0,1)$. 
The algorithm proceeds in two steps:
\begin{enumerate}
\item We solve the minimization problem
\begin{equation}\label{equn: aux-minimization-problem}
\disr^* :=(r^*_{\ba})_{\ba\in\vertex} \in \argmin_{\disr\in \disK} F^j(\disr).
\end{equation}
We refer the reader to Section~\ref{sec:algo} for algorithmic aspects of solving this problem.
\item We set, for all $\ba \in \vertex$,
\begin{equation}\label{equn:coupled-proj}
\bn^{j+1}(\ba):=P_{\bS}(\bn^j(\ba)+r^*_{\ba}\bt^j(\ba)),
\quad
\Psi_\calT^{j+1}(\ba) := \calG(\bn^{j+1}(\ba)) = \big(Q_c \bR(\bn^{j+1}(\ba)),M_c\bn^{j+1}(\ba)\big)\tr.
\end{equation}
(Observe that $\bn^j(\ba)+r^*_{\ba}\bt^j(\ba)\in \Real^2_{\ge1}$.) 
By construction, $\Psi_\calT^{j+1}(\ba) \in \calN$ for all $\ba \in \vertex$. Moreover, since $r^*_{\ba}=0$ for all $\ba\in\vertexb$, we infer that $\bn^{j+1}(\ba)=\bn^j(\ba)$, so that $\Psi_\calT^{j+1}(\ba) = \Psi_\calT^{j}(\ba) = \Psi^b(\ba)$ for all $\ba \in \vertexb$. In conclusion, we have $\Psi_\calT^{j+1}:=\sum_{\ba\in\vertex} \Psi_\calT^{j+1}(\ba) \rho_{\ba} \in W_\calT$ as claimed.
\end{enumerate}

\begin{theorem}[Energy decay] \label{thm:engy-decr-disc}
Assume that the triangulation is \rev{weakly} acute so that \eqref{eq:DMP} holds true.
Then the sequence $(\Psi_\calT^j)_{j=1}^\infty$ satisfies~\eqref{eq:enrg_decr_calT}. 
\end{theorem}

\begin{proof}
Owing to Lemma~\ref{lem:phi}, we infer that
\[
\bnu^{j+1}(\ba):= \bR(\bn^{j+1}(\ba)) = (\bR \circ P_{\bS})(\bn^j(\ba)+r^*_{\ba}\bt^j(\ba))
=P_{\bS}(\bnu^j(\ba)+\varphi(r^*_{\ba})\btau^j(\ba)).
\]
Invoking Lemma~\ref{lem:engy_eval} gives
\begin{align*}
\int_\Omega &\|\nabla \Psi_\calT^{j+1}\|_{\Real^{4\times2}}^2 \dx \\= {}& - \frac12
\sum_{\ba,\bb\in\vertex} k_{\ba\bb}\Big\{ \, Q_c^2\|\bnu^{j+1}(\ba)-\bnu^{j+1}(\bb)\|_{\Real^2}^2
+ M_c^2\|\bn^{j+1}(\ba)-\bn^{j+1}(\bb)\|_{\Real^2}^2 \, \Big\} \\
={}&  - \frac12
\sum_{\ba,\bb\in\vertex} k_{\ba\bb}\Big\{ \, Q_c^2\|P_{\bS}(\bnu^j(\ba)+\varphi(r^*_{\ba})\btau^j(\ba)) - P_{\bS}(\bnu^j(\bb)+\varphi(r^*_{\bb})\btau^j(\bb))\|_{\Real^2}^2 \\
&+ M_c^2 \|P_{\bS}(\bn^j(\ba)+r^*_{\ba}\bt^j(\ba))-P_{\bS}(\bn^j(\bb)+r^*_{\bb}\bt^j(\bb))\|_{\Real^2}^2 \, \Big\}.
\end{align*}
Invoking Lemma~\ref{lem:PS1} and crucially that $k_{\ba\bb}\le 0$ (see~\eqref{eq:DMP}),
we infer that
\begin{align*}
\int_\Omega &\|\nabla \Psi_\calT^{j+1}\|_{\Real^{4\times2}}^2 \dx \\
\le{}&  - \frac12
\sum_{\ba,\bb\in\vertex} k_{\ba\bb}\Big\{ \, Q_c^2\|(\bnu^j(\ba)+\varphi(r^*_{\ba})\btau^j(\ba)) - (\bnu^j(\bb)+\varphi(r^*_{\bb})\btau^j(\bb))\|_{\Real^2}^2 \\
&+ M_c^2 \|(\bn^j(\ba)+r^*_{\ba}\bt^j(\ba))-(\bn^j(\bb)+r^*_{\bb}\bt^j(\bb))\|_{\Real^2}^2 \, \Big\}.
\end{align*}
Recalling~\eqref{eq:def_Fj}, this proves that
\[
\int_\Omega \|\nabla \Psi_\calT^{j+1}\|_{\Real^{4\times2}}^2 \dx \le F^j(\disr^*).
\]
Observing that $\int_\Omega \|\nabla \Psi_\calT^{j}\|_{\Real^{4\times2}}^2 \dx = F^j(0)$  and since $0\in K_\calT$, we conclude that
\[
\int_\Omega \|\nabla \Psi_\calT^{j+1}\|_{\Real^{4\times2}}^2 \dx \le F^j(\disr^*)
\le F^j(0) = \int_\Omega \|\nabla \Psi_\calT^{j}\|_{\Real^{4\times2}}^2.
\]
This proves the claim.
\end{proof}

\section{\rev{Numerical realization}}
\label{sec:algo}

In this section, we provide some details on the computational algorithm for solving
the minimization problem~\eqref{equn: aux-minimization-problem}, followed by an
overall description of the energy-decreasing iterative procedure. 

\subsection{Decomposition-coordination principle using augmented Lagrangian}

To tackle the minimization problem~\eqref{equn: aux-minimization-problem}, we 
employ a decomposition-coordination strategy based on an augmented Lagrangian framework.
The ideas are: (i) to introduce the auxiliary variable $p_{\ba}=\varphi(r_\ba)$ 
for all $\ba\in\vertex$, which we rewrite in more compact form as 
$\disp=\varphi(\disr)\in \disM$; (ii) to consider a minimization problem
in the two variables $\disr:=(r_\ba)_{\ba\in\vertex}\in \disK $ and 
$\disp:=(p_\ba)_{\ba\in\vertex}\in \disM$ (decomposition step); (iii)  
to enforce that $\disp=\varphi(\disr)$
by means of a Lagrange multiplier $\disLamb=(\lambda_\ba)_{\ba\in\vertex}\in \disM$
and an augmented Lagrangian (coordination step). 

In the first step, recalling the cost functional $F^j$ in \eqref{eq:def_Fj}, we set
\begin{subequations} \begin{align}
F^j_1(\disr) &:=-\frac12 \sum_{\ba,\bb\in\vertex} k_{\ba\bb} M_c^2\|(\bn^j(\ba)+r_{\ba}\bt^j(\ba))-(\bn^j(\bb)+r_{\bb}\bt^j(\bb))\|_{\Real^2}^2,\\
F^j_2(\disp) &:= -\frac12 \sum_{\ba,\bb\in\vertex} k_{\ba \bb}  
Q_c^2\|(\bnu^j(\ba)+p_{\ba}\btau^j(\ba))-(\bnu^j(\bb)+p_{\bb}\btau^j(\bb))\|_{\Real^2}^2,
\end{align} \end{subequations} 
so that $F^j_1(\disr)+F^j_2(\disp)=F^j(\disr)$ whenever $\disp=\varphi(\disr)$. 
For later use, we record the expressions for the  partial derivatives of the above two functionals: For all $\ba \in\vertex$,
\begin{subequations}
\begin{align}
{\partial}_{r_{\ba}} F^j_1(\distr) &= -2 \sum_{\bb\in\vertex} k_{\ba\bb} M_c^2 \big({\tilde r}_{\ba} -{\tilde r}_{\bb} \bt^j(\ba)\cdot \bt^j(\bb)-\bt^j(\ba)\cdot \bn^j(\bb) \big), \label{eq:grad_F1} \\
{\partial}_{p_{\ba}} F^j_2(\distp) &= -2 \sum_{\bb\in\vertex} k_{\ba\bb} Q_c^2 \big({\tilde p}_{\ba} -{\tilde p}_{\bb} \btau^j(\ba)\cdot \btau^j(\bb)-\btau^j(\ba)\cdot \bnu^j(\bb) \big). \label{eq:grad_F2}
\end{align}
\end{subequations}
Then the decomposed version of the minimization problem~\eqref{equn: aux-minimization-problem} reads as
\begin{align}\label{equn: decomposed-prob}
& \min_{\substack{(\disr,\disp)\in {\disK} \times {\disM} \\ \disp=\varphi(\disr)}}
\Big\{ F^j_1(\disr) + F^j_2(\disp) \Big\}.
\end{align}
The constraint in \eqref{equn: decomposed-prob} is enforced using a Lagrange multiplier $\disLamb\in \disM$. We define the Lagrangian on $\disK\times \disM \times \disM$ such that
\begin{equation}\label{saddle-formulation}
{L}^j(\disr,\disp,\disLamb)=F^j_1(\disr) + F^j_2(\disp) + \disLamb\tr D (\disp-\varphi(\disr)).
\end{equation}
Here, $D$ is a diagonal matrix with entries $(d_\ba)_{\ba\in\vertex}$ that is expected to scale as the entries $k_{\ba\bb}$ of the stiffness matrix so as to balance dimensionally all the terms on the above right-hand side. In the present two-dimensional setting, a reasonable choice is $d_\ba:=h^{d-2}=1$ for all $\ba\in\vertex$, where $h$ denotes the mesh size (the largest element diameter in $\calT$). 
We work with an augmented Lagrangian formulation. Given  an augmentation parameter $\zeta>0$, we set

\begin{align}\label{equn:cost-nonlinear-ccoupling}
{L}^j_\zeta(\disr,\disp,\disLamb)=&F^j_1(\disr) + F^j_2(\disp) +  \disLamb\tr  D (\disp-\varphi(\disr)) +\frac{\zeta}{2} (\disp-\varphi(\disr))\tr D (\disp-\varphi(\disr)).
\end{align}
Our goal is now to find a \rev{critical point of $L^j$}. We do so by using an alternating-direction Uzawa-like iterative method that we describe next.

\subsection{Alternating-direction Uzawa-like iterations}

We find a \rev{critical point} of $L^j_\zeta$ by means of the Uzawa/Alternating Direction Method of Multipliers (ADMM)   \cite{Arrow:Hurwicz:Uzawa:1958,Fortin:Glowinski:1983}. 
The iterative method is initialized by specifying $\disp^0\in\disM$ and $\disLamb^0\in\disM$
(see Remark~\ref{rem:init} below). Then, at iteration $n\ge0$, assuming $\disp^n,\disLamb^n$ known from the initialization or the previous iteration, we calculate $\disr^{n+1}\in\disK,$ then $\disp^{n+1}\in\disM$ and finally $\disLamb^{n+1}\in\disM$ in the following three steps:
\begin{subequations} \label{eq:Uzawa} \begin{align}
\disr^{n+1} &\leftarrow \text{arg}\min_{\disr\in \disK}{L}^j_\zeta (\disr,\disp^n, \disLamb^n), \label{equn:opt-r}\\
\disp^{n+1} &\leftarrow \text{arg}\min_{\disp\in \disM}{L}^j_\zeta (\disr^{n+1},\disp, \disLamb^n), \label{equn:opt-p}\\
\disLamb^{n+1} &\leftarrow \disLamb^{n}+\rho  \nabla_\lambda{L}^j_\zeta(\disr^{n+1},\disp^{n+1}, \disLamb), \label{equn:opt-LM}
\end{align}\end{subequations}
where $\rho>0$ is a user-dependent parameter. 

Let us briefly describe the three steps in~\eqref{eq:Uzawa}, starting with the second step~\eqref{equn:opt-p} which simply amounts to a linear solve: Find $\disp^{n+1}\in\disM$ so that
\begin{equation} \label{eq:linsys-p}
{\partial}_{p_{\ba}}  {L}^j_\zeta(\disr^{n+1}, \disp^{n+1}, \disLamb^{n})= {\partial}_{p_{\ba}}  F_2^j(\disp^{n+1})   + d_{\ba} \lambda_{\ba}^{n} +\zeta d_{\ba} (p_{\ba}^{n+1} - \varphi(r_{\ba}^{n+1}))=0, \quad \forall\ba\in \vertex,
\end{equation}       
where the partial derivatives of $F_2^j$ are given by~\eqref{eq:grad_F2}.
The third step, \eqref{equn:opt-LM}, in turn amounts to a simple componentwise 
update on the Lagrange multiplier by setting
\begin{equation} \label{eq:update-LM}
\lambda_\ba^{n+1} = \lambda_\ba^{n} + \rho d_\ba (p_\ba^{n+1}- \varphi(r_\ba^{n+1})), \quad \forall \ba\in \vertex.
\end{equation}
Finally, the first step, \eqref{equn:opt-r}, is more intricate owing to the nonlinearity produced by the function $\varphi$. The idea to enhance computational efficiency is to replace, only in the first step~\eqref{equn:opt-r}, the nonlinear function $\varphi$ by its linearization at $\disr^{n}$ given by $\varphi_{\textsc{l}}^n(\disr):=\varphi(\disr^{n})+\varphi'(\disr^{n})(\disr-\disr^{n})$. With the same goal in mind, we also relax the constraint of seeking a minimizer in $\disK$, and just seek a minimizer in $\disM$ (this is possible since the linearized function $\varphi_{\textsc{l}}$ is defined over the whole real line). Our numerical experiments below indicate that the obtained minimizer anyway lies in $\disK$ (see Remark~\ref{rem:domain} below for further discussion). Altogether, the first step~\eqref{equn:opt-r} is replaced by 
\begin{equation} \label{eq:step1_linLag}
\disr^{n+1} \leftarrow \text{arg}\min_{\disr\in \disM}{L}^{j,n}_{\textsc{l},\zeta} (\disr,\disp^n, \disLamb^n),
\end{equation}
with the linearized augmented Lagrangian 
\begin{equation}
{L}^{j,n}_{\textsc{l},\zeta}(\disr,\disp,\disLamb)=F^j_1(\disr) + F^j_2(\disp) +  \disLamb\tr  D (\disp-\varphi_{\textsc{l}}^n(\disr)) +\frac{\zeta}{2} (\disp-\varphi_{\textsc{l}}^n(\disr))\tr D (\disp-\varphi_{\textsc{l}}^n(\disr)).
\end{equation} 
A straightforward calculation shows that~\eqref{eq:step1_linLag} amounts to the following linear solve: Find $\disr^{n+1}\in\disM$ so that, for all $\ba\in \vertex$,
\begin{align}\label{gradL:r}
{\partial}_{r_{\ba}}{L}^{j,n}_{\textsc{l},\zeta}(\disr^{n+1}, \disp^n, \disLamb^n)\notag = {}& {\partial}_{r_{\ba}}F^j_1(\disr^{n+1})  - d_{\ba}\lambda_{\ba}^n \varphi'(r_{\ba}^n) \\&-\zeta d_{\ba} \big(p_{\ba}^n -\varphi(r_{\ba}^n) -\varphi'(r_{\ba}^n)(r^{n+1}_{\ba}-r_{\ba}^{n})\big)\varphi'(r_{\ba}^n) = 0, 
\end{align}
where the partial derivatives of $F^j_1$ are given by~\eqref{eq:grad_F1}. The linearization in \eqref{eq:step1_linLag} is consistent since  ${L}^{j,n}_{\textsc{l},\zeta} \rightarrow {L}^{j}_{\zeta}$ in the limit $n \rightarrow \infty$ (provided convergence holds).
We emphasize that the nonlinear function $\varphi$ is still used in~\eqref{equn:opt-p} and in~\eqref{equn:opt-LM}. In particular, if the Uzawa-like iteration converges, \eqref{eq:update-LM} guarantees that $p_\ba^{n+1} = \varphi(r_\ba^{n+1})$, which is the critical condition required to ensure an energy decrease. One consequence of the above linearization is the need to provide a value for the initial primal variable, $\disr^0$.

\begin{remark}[Initialization] \label{rem:init}
We initialize the Uzawa/ADMM iterations in a nested fashion by exploiting the outer energy-decreasing iteration (index $j$). At iteration $j$, we set $\disr^0:=\disr^{{j-1}}$, $\disp^0:=\disp^{{j-1}}$ and $\disLamb^0:=\disLamb^{{j-1}}$, where the right-hand sides are the values obtained at convergence of the Uzawa/ADMM algorithm at the previous iteration $(j-1)$. For $j=0$, we simply set $\disr^0:=0$, $\disp^0:=0$ and $\disLamb^0:=0$.
\end{remark}

\subsection{Overall computational algorithm}
\label{sec:overall}

The pseudocode in Algorithm \ref{alg:energy_min_uzawa} describes the computational workflow. 
The method employs a nested iterative structure, where the outer iteration (index $j$) performs energy-descent updates, whereas the inner iteration (index $n$) solves the associated constrained minimization problem using the Uzawa/ADMM iteration described in the previous section. \rev{The stopping criterion for the outer energy-decreasing iteration is
\begin{equation} \label{eq:stop-descent}
{0\leq} E(\Psi^{j}_\calT)-E(\Psi^{j+1}_\calT) \leq  \varepsilon_{\text{outer}}
\end{equation} 
with tolerance $\varepsilon_{\text{outer}}:= 10^{-6}$.}
The stopping criterion for the inner Uzawa/ADMM iteration is based on the  satisfaction of the coupling condition $\disp=\varphi(\disr)$ and reads as follows: 
\begin{align}\label{equn: uzawa-tol}
\Bigg(\frac{1}{\#\vertex}\sum_{\ba\in  \vertex} |\varphi(r_\ba^{n+1}) - p_\ba^{n+1}|^2 \Bigg)^{\frac12}\leq \varepsilon_{\text{inner}}.
\end{align}
\rev{The default value for the tolerance $\varepsilon_{\text{inner}}$ is $10^{-6}$, but this tolerance is tightened adaptively as the outer iterations progress. Thus, $\varepsilon_{\text{inner}}$ actually depends on $j$. The criterion to adapt $\varepsilon_{\text{inner}}$ hinges on the energy decrease. If this fails, i.e., the lower bound in~\eqref{eq:stop-descent} is not satisfied, $\varepsilon_{\text{inner}}$ is reduced by a factor of $10$. Indeed, a spurious energy increase at some iteration $j$ indicates insufficient accuracy in finding a critical point of the Lagrangian, so that Theorem~\ref{thm:engy-decr-disc} cannot be applied. A tighter tolerance $\varepsilon_{\text{inner}}$ eventually restores the energy decrease. In our numerical experiments, the tolerance $\varepsilon_{\text{inner}}$ is typically decreased by one to three orders of magnitude.
Finally, the default choice for solving the linear systems arising in the Uzawa/ADMM iteration is the (unpreconditioned) MINRES iterative solver with a relative residual tolerance of $10^{-9}$ and at most $5000$ iterations.}

\begin{algorithm}[H]
\caption{Energy-descent with inner Uzawa/ADMM iterations}
\label{alg:energy_min_uzawa}
\KwIn{Initial approximation \(\Psi_\calT^{0} \in W_\calT \) on triangulation \(\calT\), \rev{initial value for $\varepsilon_{\text{inner}}$}}

\For{\(j=0,1,2,\ldots \)}{
{\bf Step 1:} {Given \(\Psi_\calT^{j} \in W_\calT\),
solve \eqref{equn: aux-minimization-problem} for
\(\disr^* \in \disK \) using Uzawa/ADMM};

{\bf Initialize Uzawa variables:} \( \disr^0:=\disr^{j-1}\), \( \disp^0:=\disp^{j-1}\), and \(\disLamb^0:=\disLamb^{j-1}\) if \(j\ge1\) or \(\disr^0:=0\), \(\disp^0:=0\), \(\disLamb^0:=0\) otherwise  \;

\For{\(n =0,1, \ldots \)}{
Step 1.1 (Primal Variable Update). Compute $\disr^{n+1}$ from \eqref{eq:step1_linLag};

Step 1.2 (Auxiliary Variable Update). Compute $\disp^{n+1}$ from \eqref{eq:linsys-p};

Step 1.3 (Lagrange multiplier Update). Compute $\disLamb^{n+1}$ from \eqref{eq:update-LM};

\If{\eqref{equn: uzawa-tol} holds}{
\textbf{break} \tcp*{Terminate inner iterations}
}
}

{\bf Step 2:} For all $\ba \in \vertex$, compute
$\Psi_\calT^{j+1}(\ba)$ as in  \eqref{equn:coupled-proj}.

\If{\eqref{eq:stop-descent} holds}{\textbf{break} \tcp*{Terminate outer iterations}
}
\Else{\rev{\textbf{if} energy decrease fails, divide $\varepsilon_{\text{inner}}$ by 10} 
\tcp*{Tighten tolerance for Uzawa/ADMM}}
}
\KwOut{Discrete approximation \(\Psi_\calT^{\text{cv}}\) and discrete energy \({E}(\Psi_\calT^{\text{cv}})\)}\end{algorithm}

\begin{remark}[\rev{Guarantees}] \label{rem:domain}
\rev{We do not provide here convergence guarantees on the inner iterations. In the next section, we perform a thorough numerical study showing that convergence is achieved for all reasonable choices of the numerical parameters $\zeta$ and $\rho$.  Furthermore, we do not enforce a barrier to guarantee that the discrete variable $\disr$ lies within the admissible interval $(-1,1)$. It turns out to be the case in all our numerical experiments. For completeness, we illustrate in Section~\ref{sec:barrier} the use of a barrier enforcing the above constraint on $\disr$.}
\end{remark}

\section{Numerical results}
\label{sec:results}

This section presents computational experiments to illustrate the performances of the present methodology in terms of efficiency and accuracy. 
We start with a baseline test case to study how the parameters $\zeta$ and $\rho$ influence the convergence of the inner Uzawa/ADMM iterations. We also consider using a barrier for the primal variable $\disr$, and we present a comparison with the gradient-projection method. Then, we present test cases with more challenging situations, including variations on the nemato-magnetic coupling parameter $c$, higher winding numbers, domains with holes and multiple solutions.
 
To define analytical solutions, we build suitable fields in the unit sphere $\bS$. To this purpose, we use below the shorthand notation $\bn \equiv \alpha$ with $\bn\in\bS$ and $\alpha\in \Real$ whenever $\bn=(\cos(\alpha),\sin(\alpha))\tr$. For all $z\in\mathbb{C}$ with $z\ne0$, the argument of $z$ is the angle $\alpha(z)\in[0,2\pi)$ so that $z=|z|e^{i\alpha(z)}$. Then, identifying points in $\Real^2$ with complex numbers, for every complex function $f:\mathbb{C}\to\mathbb{C}$ which is analytical and nonzero in $\Omega$, we can define the field $\bn_{f}:\Omega \to \bS$ so that $\bn_f(\bx) \equiv \alpha\big( f(\bx) \big) - \frac{\pi}{2}$ and set $\Psi_f(\bx):=\calG(\bn_f)$. The field $\Psi_f$ is defined on $\Omega$ and takes values in $\calN$ by construction. Moreover, one can verify that $\Psi_f$ is a minimizer of the elastic (harmonic) energy in $\Omega$ provided it satisfies the prescribed boundary conditions. In what follows, we use the field $\Psi_f$ only to compute the Dirichlet boundary prescription, and our goal is to recover numerically the entire field $\Psi_f$ inside $\Omega$ by using the above energy-decreasing iterative algorithm, starting from some other initial guess (satisfying the correct boundary conditions).

All the computations are implemented in MATLAB (R2022a) and run  on a desktop PC with an Intel Core i7-12700 processor (12 cores, 20 threads, 2.1--4.9\,GHz) and 64\,GB of RAM, running Windows~10 Enterprise. 

\subsection{Baseline test case}\label{baseline}

We consider the square domain $\Omega:=(0,1)^2$ and discretize it using three successively refined, unstructured triangulations of size $h\in \{0.1270,0.0635,0.0318\}$ (the diameter of the largest triangle in each triangulation). These triangulations are, respectively, called $\calT_1$, $\calT_2$, $\calT_3$. They all are weakly acute so that~\eqref{eq:DMP} always holds true. We denote by $\Psi^{\textrm{cv}}_{\calT_i}$ the converged solution on the mesh $\calT_i$.
The default choice for the nemato-magnetic coupling parameter is $c:=0.005$, giving $Q_c\approx 1.0013$ and $M_c\approx 1.0025$. 

In the baseline test case, the exact solution is obtained from  the complex function $f(z):=z-z_0$, where $z_0$ has real and imaginary parts corresponding to the point $\bx_0:=(2,0.2)\tr$. Since $\bx_0\not\in\overline{\Omega}$, the function $f$ is indeed nonzero in $\overline{\Omega}$ as required. It is straightforward to verify that the corresponding field $\bn_f$ taking values in $\bS$ is given by $\bn_f(\bx)=-\frac{(\bx-\bx_0)^\perp}{\|\bx-\bx_0 \|_{\Real^2}}$. As stated above, the field $\Psi_f:=\calG(\bn_f)$ is only used to prescribe the Dirichlet boundary condition $\Psi^b:=\Psi_f|_{\partial\Omega}$. For convenience, we adopt the more explicit notation $\Psi_{\bx_0}$ for $\Psi_f$. Figure \ref{figlin:soln-nM-coupling-deg1} displays the exact solution $\Psi_{\bx_0}$. For reference, we have $E(\Psi_{\bx_0}) \approx 1.15113.$ 

\begin{figure}[htb]\centering
\includegraphics[width=0.35\linewidth]
{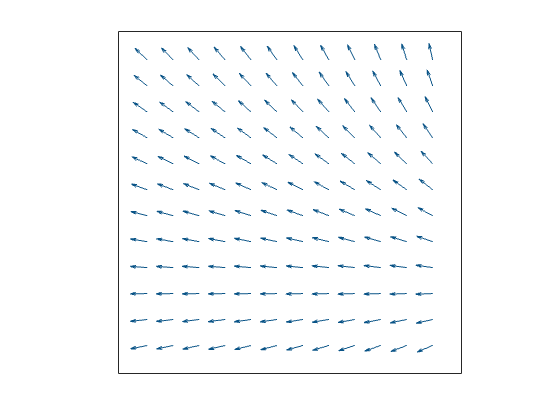}
\includegraphics[width=0.35\linewidth]
{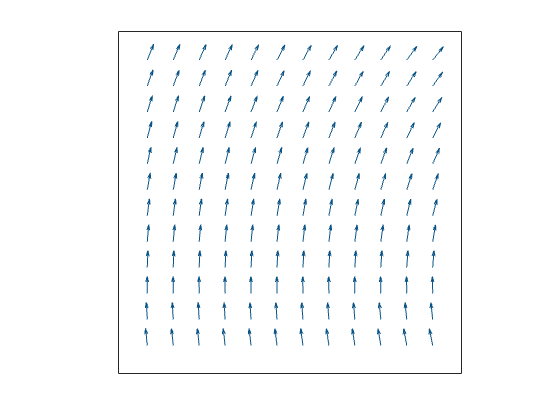}
\caption{Baseline test case. Nematic profile $(Q_{11},Q_{12})\tr$ (left) and magnetic profile $\Mvec$ (right) for the exact solution $\Psi_{\bx_0}$ with $\bx_0:=(2,0.2)\tr$; the domain is $\Omega:=(0,1)^2$.}
\label{figlin:soln-nM-coupling-deg1}
\end{figure}

The energy-decreasing iterative algorithm is initialized by using the nodal values of the field $\Psi^0:=\Psi_{f^0}:=\calG(\bn_{f^0})$ with $f^0(z)=z-y(z)$ with $y(z):=(1-\delta(z))z_0 + \delta(z)z_1$ for some point $z_1 \not\in \overline{\Omega}$ and the blending function $\delta(z):=\rho_\Omega^{-1}\text{dist}(z,\partial \Omega)$ where $\rho_\Omega:=\max_{z\in \Omega}\text{dist}(z,\partial \Omega)$. Note that $\delta$ takes values in $[0,1]$ and that $y(z)=z_0$ for all $z\in\partial\Omega$ so that indeed $\Psi^0|_{\partial\Omega} = \Psi^b$. Moreover, by construction, $\Psi^0$ takes values in $\calN$. Thus, $\Psi^0$ is a suitable choice to initialize the energy-decreasing iterative algorithm. In what follows, $z_1$ has real and imaginary parts corresponding to the point $\bx_1:=(1.5,1.5)\tr$.

\subsubsection{Choice of the Uzawa parameters $\zeta$ and $\rho$}
\label{sec:basic_tests}

We first study how the two parameters driving the Uzawa/ADMM iterations, namely the augmentation parameter $\zeta$ and the update parameter $\rho$, influence the convergence behavior of the overall algorithm. We consider the values 
$\zeta\in\{1,4,16\}$ and $\rho\in\{1,0.5,0.25\}$. Table~\ref{Tab:tol-Ex1} reports the final inner tolerances obtained using the adaptive approach described in Section~\ref{sec:overall}. We emphasize that a monotonic decrease of the energy is obtained in all cases. We observe that the required inner tolerance becomes smaller under mesh refinement, by roughly one order of magnitude per refinement. Moreover, on a fixed mesh, larger penalty parameters $\zeta$ require smaller inner tolerances, whereas the dependence on $\rho$ is less pronounced.

\ifHAL \begin{table}[htb] \else \begin{table}[width=.6\linewidth,cols=10,pos=htb] \fi
\ifHAL \centering \fi
\caption{Baseline test case. Final adaptive inner tolerance on the first three triangulations for various choices of the Uzawa/ADMM parameters ($\zeta\in\{1,4,16\}$, $\rho\in\{1,0.5,0.25\}$).}
\label{Tab:tol-Ex1}
\ifHAL
\begin{tabular}{|c |ccc| ccc| ccc|}
\hline
& \multicolumn{3}{c|}{$\calT_1$}
& \multicolumn{3}{c|}{$\calT_2$}
& \multicolumn{3}{c|}{$\calT_3$} \\
\hline
\else
\begin{tabular*}{\tblwidth}{@{}c ccc ccc ccc@{}}
\toprule
& \multicolumn{3}{c}{$\calT_1$}
& \multicolumn{3}{c}{$\calT_2$}
& \multicolumn{3}{c}{$\calT_3$} \\
\midrule
\fi
$\zeta\, \backslash\,\rho $ & $1$ &  $0.5$ &  $0.25$ & $1$ &  $0.5$ &  $0.25$ & $1$ &  $0.5$ &  $0.25$\\
\ifHAL \hline \fi
$1$ & $10^{-6}$ & $10^{-6}$ & $10^{-6}$& $10^{-7}$ & $10^{-7}$ & $10^{-6}$ &$10^{-8}$ &$10^{-8}$ & $10^{-8}$\\
$4$ & $10^{-6}$ & $10^{-6}$ & $10^{-6}$& $10^{-6}$ & $10^{-7}$ & $10^{-7}$ &$10^{-9}$ &$10^{-8}$ & $10^{-8}$\\
$16$& $10^{-7}$ & $10^{-7}$ & $10^{-7}$& $10^{-8}$ & $10^{-8}$ & $10^{-7}$ &$10^{-9}$ &$10^{-9}$ &$10^{-9}$\\
\ifHAL \hline \fi
\ifHAL \end{tabular} \else \end{tabular*} \fi
\end{table}

Table \ref{Tab:en-err-Ex1-adptol} reports the relative energy errors at convergence of the overall algorithm, $\Delta E_{i} := E(\Psi^{\rm cv}_{\calT_i})/E(\Psi_{\bx_0})-1$. 
 We observe that the relative energy error is reduced as the mesh is refined, confirming that discretization errors are under control. Moreover, as expected, the two Uzawa/ADMM parameters have almost no impact (or a very tiny one) on the relative energy error. Thus, the value of $\zeta$ and $\rho$  only affects the cost of the computation, but not the accuracy.  Finally, we note that in this example, $\Delta E_i$ always takes positive values.  

\ifHAL \begin{table}[htb] \else \begin{table}[width=.6\linewidth,cols=10,pos=htb] \fi
\ifHAL \centering \fi
\caption{Baseline test case. Relative energy error $\Delta E_{i} := {E(\Psi^{\rm cv}_{\calT_i})/E(\Psi_{\bx_0})-1}$ on the first three triangulations for various choices of the Uzawa/ADMM parameters ($\zeta\in\{1,4,16\}$, $\rho\in\{1,0.5,0.25\}$).}
\label{Tab:en-err-Ex1-adptol}
\ifHAL
\begin{tabular}{|c |ccc| ccc| ccc|}
\hline
& \multicolumn{3}{c|}{$\calT_1 \,\, (\times 10^{-3})$}
& \multicolumn{3}{c|}{$\calT_2\,\, (\times 10^{-4})$}
& \multicolumn{3}{c|}{$\calT_3\,\, (\times 10^{-4})$} \\
\hline
\else
\begin{tabular*}{\tblwidth}{@{}c ccc ccc ccc@{}}
\toprule
& \multicolumn{3}{c}{$\calT_1 \,\, (\times 10^{-3})$}
& \multicolumn{3}{c}{$\calT_2\,\, (\times 10^{-4})$}
& \multicolumn{3}{c}{$\calT_3\,\, (\times 10^{-4})$} \\
\midrule
\fi
$\zeta\, \backslash\,\rho $ & $1$ &  $0.5$ &  $0.25$ & $1$ &  $0.5$ &  $0.25$ & $1$ &  $0.5$ &  $0.25$\\
\ifHAL \hline \fi
$1$ &0.32 & 0.32 & 0.32 & 0.81  &  0.81 &  0.81 & 0.20 &  0.20 &  0.20\\
$4$ &0.32 & 0.32 & 0.32 & 0.82  &  0.81 &  0.81 & 0.20 &  0.20 &  0.20\\
$16$&0.32 & 0.32 & 0.32 & 0.81  &  0.81 &  0.81 & 0.20 &  0.20 &  0.20\\
\ifHAL \hline \fi
\ifHAL \end{tabular} \else \end{tabular*} \fi
\end{table}
Table~\ref{Tab:Uzawa-Ex1} reports the total number of inner (Uzawa/ADMM) iterations (in units of $10^3$) and, in parentheses, the number of outer iterations. The latter is $6$ in all cases except in one case where it is $5$. 
An increase in the total number of inner iterations is observed (as expected) with mesh refinement. 
Another observation is that the augmentation parameter $\zeta$ has a strong influence on the inner iteration count: larger values (for example $\zeta=16$) lead to (much) more inner  iterations across all meshes. The effect of $\rho$
is much milder.
The cells in Table~\ref{Tab:Uzawa-Ex1} corresponding to the lowest value of the inner iteration count for each mesh are marked in green. For the three meshes, these cells correspond to $(\zeta, \rho)=(1,1)$.

\ifHAL \begin{table}[htb] \else \begin{table}[width=.88\linewidth,cols=10,pos=htb] \fi
\ifHAL \centering \fi
\caption{Baseline test case. Total number of inner (Uzawa/ADMM) iterations (in units of $10^3$) and, in parentheses, number of outer iterations on the first three triangulations for various choices of the Uzawa/ADMM parameters ($\zeta\in\{1,4,16\}$, $\rho\in\{1,0.5,0.25\}$). The shaded cells in the table indicate the most effective choice of parameters for each triangulation.}
\label{Tab:Uzawa-Ex1}
\ifHAL
\setlength{\tabcolsep}{3pt}
\begin{tabular}{|c|ccc|ccc|ccc|}
\hline
\multirow{2}{*}{$\zeta\,\backslash\,\rho$} & \multicolumn{3}{c|}{$\calT_1$} & \multicolumn{3}{c|}{$\calT_2$} & \multicolumn{3}{c|}{$\calT_3$} \\
\cline{2-10}
& \multicolumn{1}{c}{$1$} & \multicolumn{1}{c}{$0.5$} & \multicolumn{1}{c|}{$0.25$}
& \multicolumn{1}{c}{$1$} & \multicolumn{1}{c}{$0.5$} & \multicolumn{1}{c|}{$0.25$}
& \multicolumn{1}{c}{$1$} & \multicolumn{1}{c}{$0.5$} & \multicolumn{1}{c|}{$0.25$} \\
\hline
\else
\begin{tabular*}{\tblwidth}{@{}c ccc ccc ccc@{}}
\toprule
& \multicolumn{3}{c}{$\calT_1$} & \multicolumn{3}{c}{$\calT_2$} & \multicolumn{3}{c}{$\calT_3$} \\
\midrule
$\zeta\,\backslash\,\rho$ & $1$ & $0.5$ & $0.25$ & $1$ & $0.5$ & $0.25$ & $1$ & $0.5$ & $0.25$ \\
\midrule
\fi
1  & \cellcolor{green!20}0.13\,(6) & 0.19\,(6) & 0.36\,(6) & \cellcolor{green!20}0.40\,(6) & 0.44\,(6) & 0.40\,(6) & \cellcolor{green!20}\pho1.43\,(6) & \pho1.56\,(6) & \pho1.70\,(6) \\
4  & 0.39\,(6) & 0.50\,(6) & 0.82\,(6) & 0.86\,(5) & 1.44\,(6) & 1.65\,(6) & \pho6.30\,(6) & \pho5.20\,(6) & \pho5.74\,(6) \\
16 & 1.67\,(6) & 2.16\,(6) & 3.54\,(6) & 5.75\,(6) & 6.36\,(6) & 5.56\,(6) & 21.25\,(6) & 23.29\,(6) & 25.43\,(6) \\
\ifHAL \hline \fi
\ifHAL \end{tabular} \else \end{tabular*} \fi
\end{table}
Table \ref{table:three-ener-Ex1} reports the exact energy $E(\Psi_{\bx_0})$ (computed using a sixth-order quadrature on each mesh cell), the energy $E(\Psi_{\bx_0}^I)$ for the nodal interpolation $\Psi_{\bx_0}^I$ of the exact solution $\Psi_{\bx_0}$, and the discrete energy $E(\Psi^{\rm cv}_{\calT_i})$ at convergence on the three triangulations  ${\calT_i}$, $i\in\{1,2,3\}.$ We observe that the energy $E(\Psi_{\bx_0}^I)$ using nodal interpolation  provides an upper bound on the exact energy, $E(\Psi_{\bx_0})$, for the present test case. Table~\ref{table:three-ener-Ex1} also reports the relative energy errors $\Delta E_{i}$, yielding second-order convergence,  together with the errors on $\Psi^{\rm cv}_{\calT_i}$ evaluated in the $H^1$-seminorm and in the $L^2$-norm, yielding first- and second-order convergence, respectively. 
Table~\ref{table:energy-iter-Ex1} displays the computed energies $E(\Psi^j_{\calT_i})$ as a function of the outer iteration index $j\in\{0,\ldots,6\}$ and the mesh index $i\in\{1,2,3\}$, confirming the energy decrease in all cases.

\ifHAL \begin{table}[htb] \else \begin{table}[width=0.9\linewidth,cols=11,pos=htb] \fi
\ifHAL \centering \fi
\caption{Baseline test case. Exact energy $E(\Psi_{\bx_0})$, energy for the nodal interpolation of the exact solution, $E(\Psi_{\bx_0}^I)$, computed energy $E(\Psi^{\rm cv}_{\calT_i})$ at convergence, relative energy error $\Delta E_{i}$, and relative errors in $H^1$-seminorm and $L^2$-norm on $\Psi^{\rm cv}_{\calT_i}$, together with convergence rates, for the three triangulations. The second column (kDofs) reports the total number of Dofs (4 times the number of mesh vertices) in units of $10^3$.}
\label{table:three-ener-Ex1}\label{Table:error-estimates}
\small
\ifHAL
\begin{tabular}{|c|c|ccc|cc|cc|cc|}
\hline
\else
\begin{tabular*}{\tblwidth}{@{}cc ccc cc cc cc@{}}
\toprule
\fi
$i$ & kDofs & $E(\Psi_{\bx_0})$ & $E(\Psi_{\bx_0}^I)$ & $E(\Psi^{\rm cv}_{\calT_i})$ & $\Delta E_{i}$ & rate & $H^1$-seminorm & rate & $L^2$-norm & rate \\
\ifHAL \hline \fi
1 & 0.56  & 1.151 & 1.152 & 1.152 & 3.19$\times 10^{-4}$ & --   & 4.47$\times 10^{-2}$ & --    & 1.05$\times 10^{-3}$ & --    \\
2 & 2.08  & 1.151 & 1.151 & 1.151 & 8.11$\times 10^{-5}$       & 1.98 & 2.24$\times 10^{-2}$ & 1.00 & 2.63$\times 10^{-4}$ & 2.00 \\
3 & 8.00 & 1.151 & 1.151 & 1.151 & 2.04$\times 10^{-5}$       & 1.99 & 1.12$\times 10^{-2}$ & 1.00 & 6.62$\times 10^{-5}$ & 1.99 \\
\ifHAL \hline \fi
\ifHAL \end{tabular} \else \end{tabular*} \fi
\end{table}

\ifHAL \begin{table}[htb] \else \begin{table}[width=.3\linewidth,cols=4,pos=htb] \fi
\ifHAL \centering \fi
\caption{Baseline test case. Discrete energy $E(\Psi^j_{\calT_i})$ for outer iteration index $j\in\{0,\ldots,6\}$ and mesh index $i\in\{1,2,3\}$.}
\label{table:energy-iter-Ex1}
\ifHAL
\begin{tabular}{|c|rrr|}
\hline
$j\,\backslash\,i$ & \multicolumn{1}{c}{$1$} & \multicolumn{1}{c}{$2$} & \multicolumn{1}{c|}{$3$} \\
\hline
\else
\begin{tabular*}{\tblwidth}{@{}c rrr@{}}
\toprule
$j\,\backslash\,i$ & \multicolumn{1}{c}{$1$} & \multicolumn{1}{c}{$2$} & \multicolumn{1}{c}{$3$} \\
\midrule
\fi
0 & 10.023 & 10.442 & 10.639 \\
1 &  3.160 &  3.340 &  3.426 \\
2 &  1.295 &  1.325 &  1.339 \\
3 &  1.153 &  1.153 &  1.153 \\
4 &  1.152 &  1.151 &  1.151 \\
5 &  1.152 &  1.151 &  1.151 \\
6 &  1.152 &  1.151 &  1.151 \\
\ifHAL \hline \fi
\ifHAL \end{tabular} \else \end{tabular*} \fi
\end{table}

Figure~\ref{figl:energy-errorEx1} presents the histograms of the absolute pointwise errors between $\disp^*$ and $\varphi(\disr^*)$ (in logarithmic scale) sampled at each interior mesh vertex for the first three triangulations $\calT_i$, $i\in\{1,2,3\}$, and $\zeta=\rho:=1$. The histograms are normalized by the total number of interior mesh vertices. The distributions are supported in the interval $[-7.07,-5.75]$ for $i=1$, $[-8.56,-6.68]$ for $i=2$, and $[-10.16,-7.68]$ for $i=3$. A progressive leftward shift of the mean value ($\mu$), from $7.93\times10^{-7}$ on $\calT_1$, to $8.78\times10^{-8}$ on $\calT_2$, and to $8.48\times10^{-9}$ on $\calT_3$, is observed, that is, about one order of magnitude improvement per mesh  refinement. 
   The distributions are skewed, with a sharp upper cut-off at about $2\,\varepsilon_{\text{inner}}$ and a tail towards smaller errors.  The smallest errors occur at the vertices adjacent to $\partial\Omega$, where $\disr^*$ and $\disp^*$ are small because of the homogeneous Dirichlet condition. The tail only contains relatively few vertices: at most $0.25\%$ of the interior vertices have an error below $\varepsilon_{\text{inner}}/100$.

\begin{figure}[htb]
\centering
{\includegraphics[width=8cm,height=6cm]{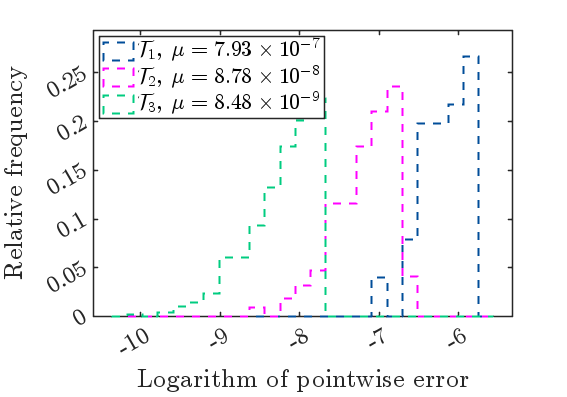}}
\caption{Baseline test case. Histograms of absolute pointwise error (logarithmic  scale) between $\disp^*$ and $\varphi(\disr^*)$ at convergence on the three triangulations. The mean value for each distribution  ($\mu$) is reported in the legend.}
\label{figl:energy-errorEx1}
\end{figure}

\subsubsection{Comparison with other solver strategies}
\label{sec:barrier}

In this section, we compare the present methodology with two variants: the use of a logarithmic barrier for the primal variable $\disr$ to guarantee that is stays in the admissible domain, and the use of a gradient-projection method instead of the present Lagrangian framework with Uzawa/ADMM iterations. We refer to these two variants as LB and PG, respectively.

Considering first LB, we consider the Lagrangian (we omit the augmentation term and we use the same notation as above for simplicity)
\begin{equation}
{L}^j(\disr,\disp,\disLamb)=F^j_1(\disr) + F^j_2(\disp) + \disLamb\tr  D (\disp-\varphi(\disr))  +  \sum_{\ba\in {\vertexi}} s_\ba \log\abs{1-r_\ba^2},
\end{equation}
where $(s_\ba)_{\ba\in \vertexi}$ is a collection of user-dependent weights. The energy barrier in the Lagrangian introduces a nonlinearity in the first step, \eqref{equn:opt-r}, of the Uzawa/ADMM algorithm. Hence, the linearization of the coupling function $\varphi$ is no longer relevant. The nonlinear minimization problem \eqref{equn:opt-r}  to evaluate $\disr^{n+1}$ is  implemented using a gradient-descent iterative scheme combined  with Armijo's rule for line-search at each step. Details are omitted for brevity as this algorithm is classical. 
For the computations, we use as above the Uzawa parameters $\zeta=\rho:=1$ and we set $s_\ba:=0.01$ for all $\ba\in \vertexi$. 
The outer iteration tolerance is $\varepsilon_{\text{outer}}=10^{-6}$ (as above), whereas the final, adaptive Uzawa/ADMM tolerances are $\varepsilon_{\text{inner}}=10^{-6}$ on $\calT_1,\calT_2,$ and $\varepsilon_{\text{inner}}=10^{-7}$ on $\calT_3.$    
The outer iteration terminates after $6$, $6$ and $7$ iterations, respectively. The relative errors on the solution in the $H^1$-seminorm and the $L^2$-norm and the relative energy error are essentially the same as those reported in the previous section, confirming that the expected accuracy has been reached.

Considering now PG, we solve \eqref{equn: aux-minimization-problem} with the gradient-projection method, using an Armijo backtracking rule along the projection arc
\cite{Bertsekas2016}.
Starting from $r_{\calT}^0:=0$, we compute, for all $n\ge0$,
\begin{equation} \label{eq:pg}
r_{\calT}^{n+1}:=P_{K_{\calT}}\big(r_{\calT}^n-\eta_n\partial_{{r_\calT}} F^j(r_{\calT}^n)\big),
\end{equation}
where the step size $\eta_n>0$ is chosen by Armijo's rule and where $P_{K_{\calT}}:M_{\calT}\to K_{\calT}$ implements the projection onto the interval $[-\tilde{\gamma},\tilde{\gamma}]$, $\tilde{\gamma}\in (0,1)$, of the values of its argument at all the interior mesh vertices. 
The last iterate in~\eqref{eq:pg} is taken as $r_{\calT}^*$ in \eqref{equn: aux-minimization-problem}. Since the coupling $p_{\calT}=\varphi(r_{\calT})$ is built into $F^j$, this approach does not require the auxiliary variable $p_{\calT}$ nor the Lagrange multiplier $\lambda_{\calT}$. Indeed, every iterate lies in $K_{\calT}$ and we have $F^j(r_{\calT}^*)\le F^j(0)$, which is the only property of $r_{\calT}^*$ needed to apply  Theorem~\ref{thm:engy-decr-disc}.
The projected-gradient iteration~\eqref{eq:pg} is stopped when the projected-gradient residual satisfies 
\begin{equation}
\left( \frac{1}{\#\vertex} \sum_{\ba\in \vertexi} |r_{\ba}^n-\big({P}_{K_{\calT}}\big(r_{\calT}^n-\partial_{r_{\calT}} F^j(r_{\calT}^n)\big) \big)_{\ba}|^2 \right)^{\frac12}\le\varepsilon_{\text{pg}}.
\end{equation}
For the computations, we set $\tilde{\gamma}:=0.8$, the outer iteration tolerance is $\varepsilon_{\text{outer}}:=10^{-6}$ (as above), and the projected-gradient tolerance is  $\varepsilon_{\text{pg}}:=10^{-6}$ (no adaptive tightening of the tolerance is needed). The outer iteration terminates after $6$ iterations on the three triangulations. The relative errors on the solution in the $H^1$-seminorm and the $L^2$-norm and the relative energy error are essentially the same as above, confirming again that the expected accuracy has been reached.

Table~\ref{table:CPU-times-Ex1} reports the relative CPU times of LB and PG normalized by the CPU time of the baseline methodology, on each triangulation. We only draw qualitative conclusions as these times depend on the implementation and our three codes leave room for further optimization, which is not the focus of this work. Our main conclusions are as follows. The three approaches (baseline, LB and PG) are suitable realizations of the energy-decreasing algorithm: in each case, the discrete energy decreases monotonically, and the same errors are obtained on each mesh. Moreover, within the present implementation, the baseline  method is the fastest, whereas LB is the most expensive, mainly because of the additional cost incurred by the gradient-descent solver within each Uzawa/ADMM iteration. To provide further insight, we also report in Table~\ref{table:CPU-times-Ex1} the total number of gradient-descent iterations and the total number of Uzawa/ADMM iterations (in parentheses) for LB, and the total number of projected-gradient iterations for PG. Interestingly, we observe that the total number of Uzawa/ADMM iterations is not much impacted by the additional enforcement of a logarithmic barrier, but that the main overhead in the cost stems from the gradient-descent iterations.

\ifHAL \begin{table}[htb] \else \begin{table}[htb] \fi
\centering
\caption{Baseline test case. Relative CPU times (normalized with respect to the baseline simulation on each triangulation $\calT_i$, $i\in\{1,2,3\}$) measured using a logarithmic barrier (LB) and the projected-gradient (PG) method for computing $\Psi^{\rm cv}_{\calT_i}$. For LB, we also report the total number of gradient-descent iterations and the total number of Uzawa/ADMM iterations (in parentheses), and for PG, we also report the total number of projected-gradient iterations. All the iteration counts are in units of $10^3$.}
\label{table:CPU-times-Ex1}
\ifHAL
\begin{tabular}{|c|cc|cc|}
\else
\begin{tabular}{@{}ccccc@{}}
\fi
\ifHAL \hline \fi
& \multicolumn{2}{c|}{LB} & \multicolumn{2}{c|}{PG} \\
\ifHAL \hline \fi
$i$ & CPU & its & CPU & its \\
\ifHAL \hline \fi
1 &  \pho1.75 & \pho1.51\,(0.13) &  1.25 & \pho0.98 \\
2 & \pho6.30  & \pho3.65\,(0.30) & 3.30 & \pho3.72 \\
3 & 11.51 & 14.09\,(1.06)&  2.22 & 10.10 \\
\ifHAL \hline \fi
\end{tabular}
\end{table}

\subsubsection{Variation on nemato-magnetic coupling parameter $c$}
\label{sec:var_c}

The amplitudes $Q_c$ and $M_c$ depend on the nemato-magnetic coupling  parameter $c$. These amplitudes modify the relative weights of the two contributions to the energy and may affect the behavior of the proposed methodology. The theoretical energy-decreasing property and the underlying mathematical framework of the algorithm, however, remain unchanged. The numerical studies reported in \cite{Ferronematics_2D,Han2021_ferro2D} typically consider coupling strengths satisfying $|c|\leq 1$ when $\ell$ is small. In particular, in the positive-coupling regime $0<c\leq 1$, the functions $c\mapsto Q_c$ and $c\mapsto M_c$ are increasing and satisfy $Q_c\leq 1.37$ and $M_c\leq 1.54.$ 
To investigate the robustness of the energy-decreasing algorithm beyond this commonly considered parameter regime, we repeat the numerical experiments from Section~\ref{sec:basic_tests} for larger values of $c$. Specifically, we consider the values $c\in\{0.25,1,10,20\}$.  
Table~\ref{table:various-c-Ex1} reports the results
on the finest triangulation $\calT_3$ for $\zeta=\rho:=1.$  
The exact energy, the energy of the nodal
interpolant of the exact solution, 
and the computed energy agree to the displayed precision
in all four cases for $c$. Despite the substantial increase in the energy
value with $c$, the relative energy error $\Delta E_3$ remains
between $0.21\times10^{-4}$ and $0.24\times10^{-4}$.
The adaptively refined inner tolerances are
$10^{-8}$, $10^{-8}$, $10^{-6}$, and $10^{-6}$, respectively, showing that the larger values of $c$ turn out to require less stringent tolerances. Moreover, the total number of inner iterations are $1273$, $865$, $2298$,
and $8886$, with $6$, $6$, $7$, and $7$ outer iterations,
respectively, for the four values of $c$.
We observe that the two larger values of $c$ lead to larger iteration counts, especially for Uzawa/ADMM iterations.

\begin{table}[htb]
\centering
\caption{Variations on nemato-magnetic coupling parameter. Exact energy $E(\Psi_{\bx_0})$, energy for the nodal interpolation of the exact solution, $E(\Psi_{\bx_0}^I)$,  computed energy $E(\Psi^{\rm cv}_{\calT_3})$ at convergence,  relative energy error $\Delta E_{i}$, total Uzawa/ADMM  and outer iterations (in units of $10^3$), and adaptively refined inner tolerance on the finest triangulation $\calT_3$ and for $c\in\{0.25, 1, 10, 20\}$. }
\label{table:various-c-Ex1}
\ifHAL
\begin{tabular}{|r|rrr rrr |}
\hline
\else
\begin{tabular}{@{}rrrr rrr @{}}
\toprule
\fi
\multicolumn{1}{|c|}{$c$}&$E(\Psi_{\bx_0})$ & $E(\Psi_{\bx_0}^I)$ & $E(\Psi^{\rm cv}_{\calT_3})$ &  \multicolumn{1}{c}{$\Delta E_{3}$}& iterations &   $\varepsilon_{\text{inner}}$\\
\ifHAL \hline \fi
0.25 & 1.35 & 1.35  & 1.35 & 0.21$\times 10^{-4}$  & 1.27\,(6) &   $10^{-8}$ \\
1 & 2.26 & 2.26 & 2.26 &  0.23$\times 10^{-4}$  & 0.87\,(6)  &$10^{-8}$ \\
10 & 64.20 & 64.20 &64.20 & 0.24$\times 10^{-4}$    & 2.30\,(7) &$10^{-6}$ \\
20 & 250.62 & 250.62 & 250.62& 0.24$\times 10^{-4}$    & 8.89\,(7) &$10^{-6}$\\
\ifHAL \hline \fi
\end{tabular}
  \end{table}
  
\subsection{Test case with higher winding number}\label{example2}

In this section, we consider another analytical solution, still on the unit square domain $\Omega:=(0,1)^2$, but we now use the complex function $\tilde{f}(z):=-(z-w_0)^3$, where $w_0$ has real and imaginary parts corresponding to the point $\by_0:=(1.2,0.2)\tr$ . Since $\by_0\not\in\overline{\Omega}$, the function $\tilde{f}$ is analytical and nonzero in $\overline{\Omega}$ as required. Writing $(n_1,n_2)\tr$ for the two coordinates of the unit vector $-\frac{(\bx-\by_0)^\perp}{\|\bx-\by_0\|_{\Real^2}}$,  
we have $\bn_{\tilde{f}}=\big(4n_1^3-3n_1,3n_2-4n_2^3\big)\tr$. The solution presented in Figure~\ref{fig:nMvector-Ex2} is $\Psi_{\by_0}:=\Psi_{\tilde{f}}=\calG(\bn_{\tilde{f}})$, and its four components $\Psi_{\by_0}=(Q_{11},Q_{12},M_1,M_2)\tr$ are displayed in Figure~\ref{fig:surface-Ex2}. This analytical solution is more challenging to approximate than the baseline solution considered in Section~\ref{baseline}. This motivates the use of one additional level of mesh refinement, leading to the (weakly acute) mesh $\calT_4$ with mesh size $0.0159$. The reference exact energy, computed using a sixth-order quadrature on that mesh, is $E(\Psi_{\by_0})\approx 57.26846.$

\begin{figure}[htb]\centering
\includegraphics[width=0.35\linewidth]
{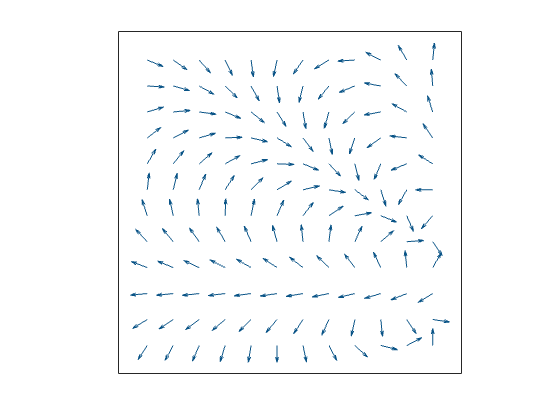}
\includegraphics[width=0.35\linewidth]
{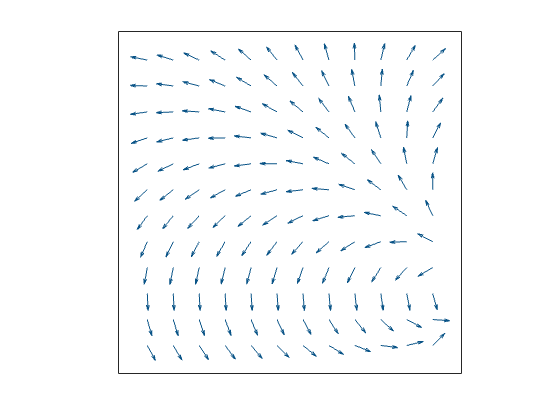}
\caption{Test case with higher winding number.  Nematic profile $(Q_{11},Q_{12})\tr$ (left) and magnetic profile $\Mvec$  (right) for the exact solution $\Psi_{\by_0}$ with $\by_0:=(1.2,0.2)\tr$; the domain is $\Omega:=(0,1)^2$.}
\label{fig:nMvector-Ex2}
\end{figure}

\begin{figure}[htb]\centering
\subfloat[$Q_{11}$]{\includegraphics[width=0.2\linewidth]
{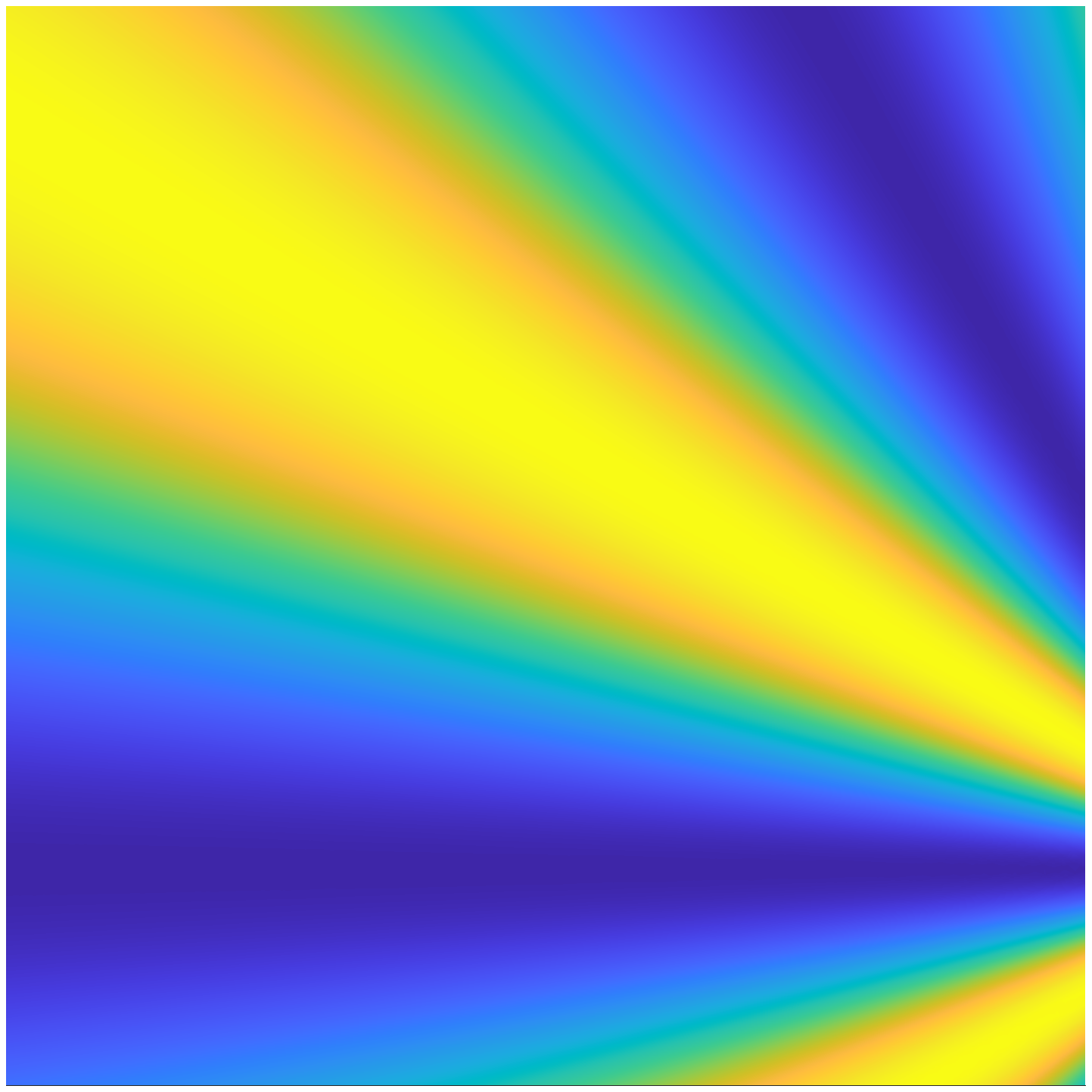}}
\hspace{0.3 cm}
\subfloat[$Q_{12}$]{\includegraphics[width=0.2\linewidth]
{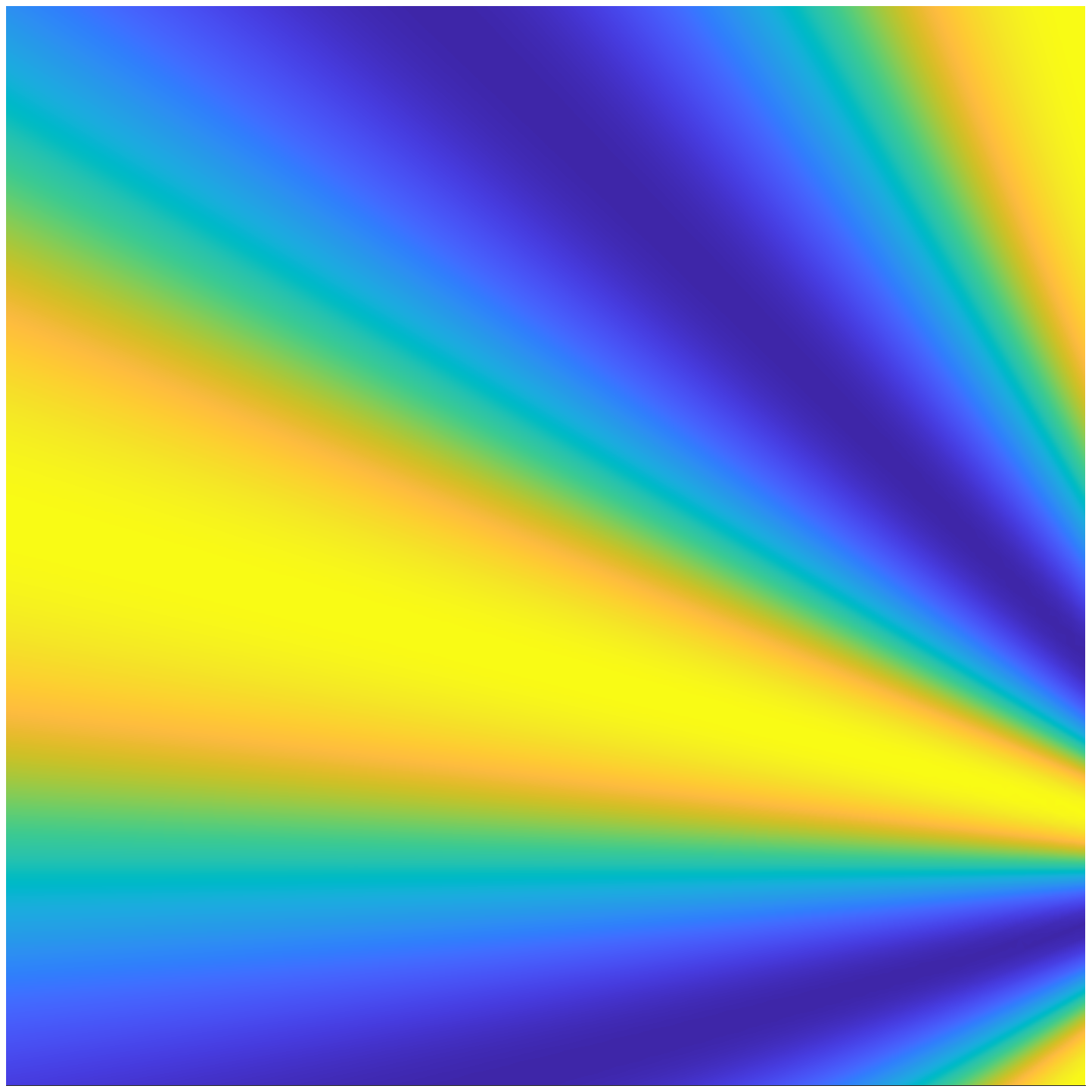}}
\hspace{0.3 cm}
\subfloat[$M_1$]{\includegraphics[width=0.2\linewidth]
{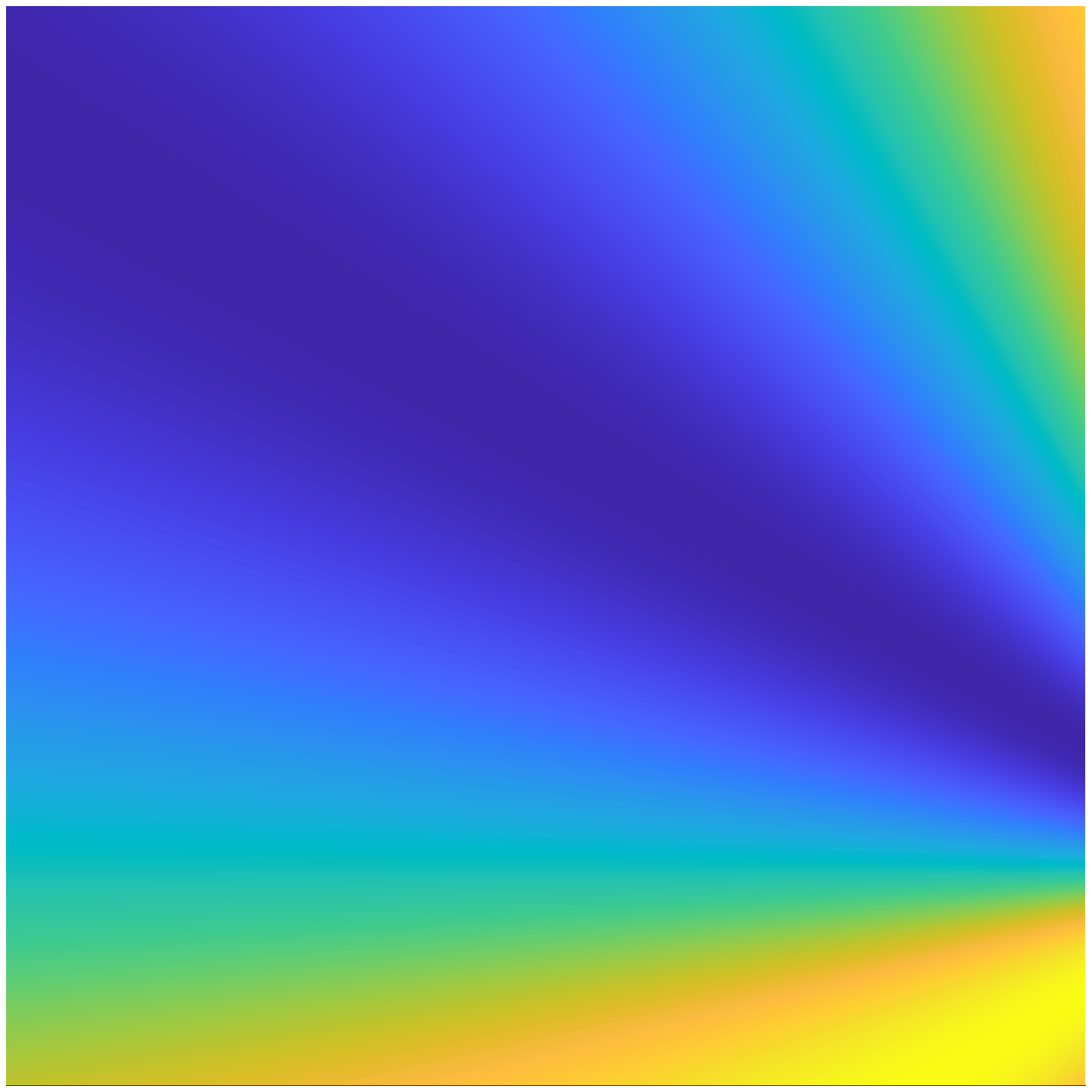}}
\hspace{0.3 cm}
\subfloat[$M_2$]{\includegraphics[width=0.2\linewidth]
{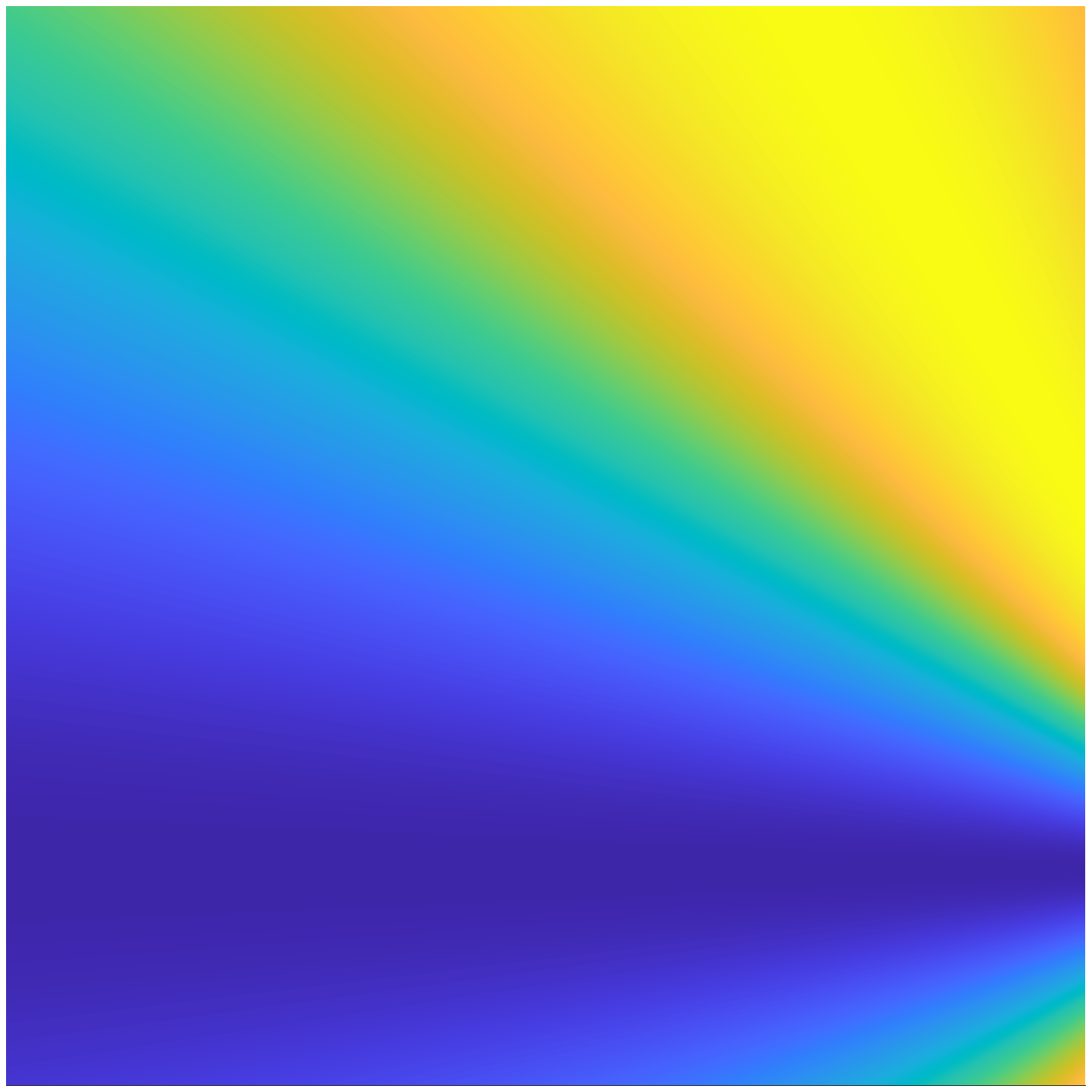}}
\caption{Test case with higher winding number. Exact solution $\Psi_{\by_0}=(Q_{11},Q_{12},M_1,M_2)\tr.$ }
\label{fig:surface-Ex2}
\end{figure}

Table~\ref{Tab:tol-Ex2} reports the final, adaptively refined, inner tolerances on the four triangulations. Compared with the baseline case, the constraint $p_{\calT}=\varphi(r_\calT)$ needs to be enforced less tightly: on $\calT_1$ and $\calT_2$,  the initial tolerance $\varepsilon_{\text{inner}}=10^{-6}$ suffices for all nine parameter pairs $(\zeta,\rho)$, and on $\calT_3$, one reduction to $10^{-7}$ is triggered only for the largest penalty parameter $\zeta=16$ with $\rho\in\{1,0.5\}$. A tighter tolerance becomes necessary only on the finest triangulation $\calT_4$, where $\varepsilon_{\text{inner}}\in\{10^{-7},10^{-8}\}$, with the smaller value again selected for the larger values of $\zeta$.

\ifHAL \begin{table}[H] \else \begin{table}[width=.8\linewidth,cols=13,pos=h] \fi
\ifHAL \centering \fi
\caption{Test case with higher winding number. Final adaptive inner tolerance on the four triangulations for various choices of the Uzawa/ADMM parameters ($\zeta\in\{1,4,16\}$, $\rho\in\{1,0.5,0.25\}$).}
\label{Tab:tol-Ex2}
\ifHAL
\begin{tabular}{|c |ccc| ccc| ccc| ccc|}
\hline
& \multicolumn{3}{c|}{$\calT_1$}
& \multicolumn{3}{c|}{$\calT_2$}
& \multicolumn{3}{c|}{$\calT_3$}
& \multicolumn{3}{c|}{$\calT_4$}\\
\hline
\else
\begin{tabular*}{\tblwidth}{@{}c ccc ccc ccc ccc@{}}
\toprule
& \multicolumn{3}{c}{$\calT_1$}
& \multicolumn{3}{c}{$\calT_2$}
& \multicolumn{3}{c}{$\calT_3$}
& \multicolumn{3}{c}{$\calT_4$}\\
\midrule
\fi
$\zeta\, \backslash\,\rho $ & $1$ &  $0.5$ &  $0.25$ & $1$ &  $0.5$ &  $0.25$ & $1$ &  $0.5$ &  $0.25$ & $1$ &  $0.5$ &  $0.25$\\
\ifHAL \hline \fi
1 & $10^{-6}$ & $10^{-6}$ & $10^{-6}$ & $10^{-6}$ & $10^{-6}$ & $10^{-6}$ & $10^{-6}$ & $10^{-6}$ & $10^{-6}$ & $10^{-7}$ & $10^{-7}$ & $10^{-7}$ \\
4 & $10^{-6}$ & $10^{-6}$ & $10^{-6}$ & $10^{-6}$ & $10^{-6}$ & $10^{-6}$ & $10^{-6}$ & $10^{-6}$ & $10^{-6}$ & $10^{-8}$ & $10^{-7}$ & $10^{-7}$ \\
16 & $10^{-6}$ & $10^{-6}$ & $10^{-6}$ & $10^{-6}$ & $10^{-6}$ & $10^{-6}$ & $10^{-7}$ & $10^{-7}$ & $10^{-6}$ & $10^{-8}$ & $10^{-8}$ & $10^{-8}$ \\
\ifHAL \hline \fi
\ifHAL \end{tabular} \else \end{tabular*} \fi
\end{table}

Table~\ref{Tab:en-err-Ex2} presents the relative energy error at convergence of the overall algorithm, $\Delta E_i:=1-E(\Psi^{\text{cv}}_{\calT_i})/E(\Psi_{\by_0})$. The relative energy errors coincide in all four displayed digits for the nine choices of the parameter pair $(\zeta,\rho)$ on every mesh. This confirms again that the Uzawa/ADMM parameters only affect the computational costs, but not the accuracy. We notice that we have reversed here the sign in the definition of $\Delta E_i$, as the discrete energy at convergence, $E(\Psi^{\text{cv}}_{\calT_i})$, turns out to be lower than the exact energy, $E(\Psi_{\by_0})$. The key observation in Table~\ref{Tab:en-err-Ex2} is the improvement of the discretization error under mesh refinement. Indeed, the error $\Delta E_i$ is rather substantial on the coarsest mesh and improves significantly under mesh refinement.

\ifHAL \begin{table}[htb] \else \begin{table}[width=.8\linewidth,cols=13,pos=htb] \fi
\ifHAL \centering \fi
\caption{Test case with higher winding number.  Relative energy error $\Delta E_i:=1-E(\Psi^{\rm cv}_{\calT_i})/E(\Psi_{\by_0})$ on the four triangulations  for various choices of the Uzawa/ADMM parameters ($\zeta\in\{1,4,16\}$, $\rho\in\{1,0.5,0.25\}$).}
\label{Tab:en-err-Ex2}
\ifHAL
\begin{tabular}{|c|ccc|ccc|ccc|ccc|}
\hline
& \multicolumn{3}{c|}{$\calT_1\,(\times 10^{-1})$} & \multicolumn{3}{c|}{$\calT_2\,(\times 10^{-2})$} & \multicolumn{3}{c|}{$\calT_3\,(\times 10^{-3})$} & \multicolumn{3}{c|}{$\calT_4\,(\times 10^{-3})$}\\
\hline
\else
\begin{tabular*}{\tblwidth}{@{}c ccc ccc ccc ccc@{}}
\toprule
& \multicolumn{3}{c}{$\calT_1\,(\times 10^{-1})$} & \multicolumn{3}{c}{$\calT_2\,(\times 10^{-2})$} & \multicolumn{3}{c}{$\calT_3\,(\times 10^{-3})$} & \multicolumn{3}{c}{$\calT_4\,(\times 10^{-3})$}\\
\midrule
\fi
$\zeta\, \backslash\,\rho $ & $1$ &  $0.5$ &  $0.25$ & $1$ &  $0.5$ &  $0.25$ & $1$ &  $0.5$ &  $0.25$ & $1$ &  $0.5$ &  $0.25$\\
\ifHAL \hline \fi
1  & 1.02 & 1.02 & 1.02 & 2.66 & 2.66 & 2.66 & 6.72 & 6.72 & 6.72 & 1.69 & 1.69 & 1.69 \\
4  & 1.02 & 1.02 & 1.02 & 2.66 & 2.66 & 2.66 & 6.72 & 6.72 & 6.72 & 1.69 & 1.69 & 1.69 \\
16 & 1.02 & 1.02 & 1.02 & 2.66 & 2.66 & 2.66 & 6.72 & 6.72 & 6.72 & 1.69 & 1.69 & 1.69 \\
\ifHAL \hline \fi
\ifHAL \end{tabular} \else \end{tabular*} \fi
\end{table} 

Next we examine the total number of Uzawa/ADMM iterations for $\zeta\in\{1,4,16\}$ and $\rho\in\{1,0.5,0.25\}$ in Table \ref{Tab:Cu-Uzawa-Ex2}. Larger values of $\zeta$ lead to higher iteration counts, whereas the influence of $\rho$ is comparatively milder, yielding similar iteration counts in most cases except on $\calT_1$. On this coarse mesh, the number of iterations increases rapidly with increasing $\zeta$ and decreasing $\rho$. The minimal iteration counts are observed for $(\zeta,\rho)=(1,1)$ on the four meshes. The lowest iteration count is actually observed on $\calT_2$ rather than on $\calT_1$.

\ifHAL \begin{table}[htb] \else \begin{table}[width=0.7\linewidth,cols=7,pos=htb] \fi
\ifHAL \centering \fi
\caption{Test case with higher winding number. Total number of Uzawa/ADMM iterations (in units of $10^3$) and, in parentheses, number of outer iterations on the first four triangulations for various choices of the Uzawa/ADMM parameters ($\zeta\in\{1,4,16\}$, $\rho\in\{1,0.5,0.25\}$). The shaded cells in the table indicate the most effective choice of parameters for each triangulation.}
\label{Tab:Cu-Uzawa-Ex2}
\ifHAL
\begin{tabular}{|c|ccc|ccc|}
\hline
\multirow{2}{*}{$\zeta\,\backslash\,\rho$} & \multicolumn{3}{c|}{$\calT_1$} & \multicolumn{3}{c|}{$\calT_2$} \\
\cline{2-7}
& \multicolumn{1}{c}{$1$} & \multicolumn{1}{c}{$0.5$} & \multicolumn{1}{c|}{$0.25$} & \multicolumn{1}{c}{$1$} & \multicolumn{1}{c}{$0.5$} & \multicolumn{1}{c|}{$0.25$} \\
\hline
\else
\begin{tabular*}{\tblwidth}{@{\extracolsep{\fill}}c ccc ccc@{}}
\toprule
&\multicolumn{3}{c}{$\calT_1$}&\multicolumn{3}{c}{$\calT_2$}\\
\cmidrule(lr){2-4}\cmidrule(l){5-7}
$\zeta\,\backslash\,\rho$ & $1$ & $0.5$ & $0.25$ & $1$ & $0.5$ & $0.25$ \\
\midrule
\fi
1  & \cellcolor{green!20}0.54\,(42) & 1.11\,(42) & \pho2.26\,(42)  & \cellcolor{green!20}0.37\,(28) & 0.53\,(28) & 0.98\,(28) \\
4  & 1.11\,(42) & 2.18\,(42) & \pho4.34\,(42)  & 1.01\,(28) & 1.33\,(28) & 2.18\,(28) \\
16 & 3.02\,(42) & 5.69\,(42) & 11.12\,(42) & 2.80\,(26) & 3.77\,(27) & 5.97\,(27) \\
\ifHAL
\hline
\multirow{2}{*}{$\zeta\,\backslash\,\rho$} & \multicolumn{3}{c|}{$\calT_3$} & \multicolumn{3}{c|}{$\calT_4$} \\
\cline{2-7}
& \multicolumn{1}{c}{$1$} & \multicolumn{1}{c}{$0.5$} & \multicolumn{1}{c|}{$0.25$} & \multicolumn{1}{c}{$1$} & \multicolumn{1}{c}{$0.5$} & \multicolumn{1}{c|}{$0.25$} \\
\hline
\fi
1  & \cellcolor{green!20}0.66\,(26) & \pho0.79\,(26) & 0.96\,(26)  & \cellcolor{green!20}\pho2.18\,(25) & \pho2.61\,(25) & \pho3.11\,(26) \\
4  & 1.72\,(23) & \pho2.13\,(24) & 2.71\,(25)  & 10.86\,(26) & \pho7.15\,(23) & \pho8.74\,(24) \\
16 & 8.78\,(25) & 10.80\,(25) & 7.46\,(22) & 30.27\,(24) & 36.59\,(25) & 44.20\,(26) \\
\ifHAL \hline \fi
\ifHAL \end{tabular} \else \end{tabular*} \fi
\end{table}

Table \ref{table:three-ener-Ex2} reports the energies for the exact solution, $\Psi_{\by_0}$, its nodal interpolant, $\Psi_{\by_0}^I$, and the discrete solution $\Psi^{\rm cv}_{\calT_i}$ at convergence on the four triangulations. The energy of the nodal interpolant, $E(\Psi_{\by_0}^I)$, turns out to be smaller (even on the finest mesh) than the exact energy, and, as expected, the discrete energy $E(\Psi^{\rm cv}_{\calT_i})$ is (slightly) smaller than  $E(\Psi_{\by_0}^I)$. Table \ref{table:three-ener-Ex2} also reports the relative energy errors, yielding second-order convergence, together with 
the errors on $\Psi^{\rm cv}_{\calT_i}$ evaluated in the $H^1$-seminorm and in the $L^2$-norm, yielding first- and second-order convergence, respectively. These results match the observations made for the baseline case.  
The computed energies $E(\Psi^j_{\calT_i})$ as a function of the outer iteration index $j\in \{0,\ldots,5\}$ and the mesh index $i\in\{1,2,3,4\}$ are reported in Table \ref{table:energy-Ex2}, confirming the energy decrease in all cases. We also report the discrete energy at convergence, the total number of Uzawa/ADMM iterations (in units of $10^3$) together with the number of outer iterations (between parentheses), and the final adaptive inner tolerance.

\ifHAL \begin{table}[htb] \else \begin{table}[width=0.9\linewidth,cols=11,pos=htb] \fi
\ifHAL \centering \fi
\caption{Test case with higher winding number. Exact energy $E(\Psi_{\by_0})$, energy for the nodal interpolation of the exact solution, $E(\Psi_{\by_0}^I)$, computed energy $E(\Psi^{\rm cv}_{\calT_i})$ at convergence, relative energy error $\Delta E_{i}$, and relative errors in $H^1$-seminorm and $L^2$-norm on $\Psi^{\rm cv}_{\calT_i}$, together with convergence rates, for the four triangulations. The second column (kDofs) reports the total number of Dofs (4 times the number of mesh vertices) in units of $10^3$.}
\label{table:three-ener-Ex2}\label{Table:ex2-error-estimates}
\small
\ifHAL
\begin{tabular}{|c|r|ccc|cc|cc|cc|}
\hline
\else
\begin{tabular*}{\tblwidth}{@{}cr ccc cc cc cc@{}}
\toprule
\fi
$i$ & kDofs & $E(\Psi_{\by_0})$ & $E(\Psi_{\by_0}^I)$ & $E(\Psi^{\rm cv}_{\calT_i})$ & $\Delta E_{i}$ & rate & $H^1$-seminorm & rate & $L^2$-norm & rate \\
\ifHAL \hline \fi
1 & 0.56  & 57.268 & 51.594 & 51.449 & $1.02 \times 10^{-1}$ & --    & $3.42 \times 10^{-1}$ & --   & $7.44 \times 10^{-2}$ & --   \\
2 & 2.08  & 57.268 & 55.753 & 55.747 & $2.66 \times 10^{-2}$ & 1.94 & $1.75 \times 10^{-1}$ & 0.97 & $1.79 \times 10^{-2}$ & 2.05 \\
3 & 8.00 & 57.268 & 56.884 & 56.884 & $6.72 \times 10^{-3}$ & 1.98 & $8.84 \times 10^{-2}$ & 0.99 & $4.47 \times 10^{-3}$ & 2.00 \\
4 & 31.36 & 57.268 & 57.172 & 57.172 & $1.69 \times 10^{-3}$ & 2.00 & $4.43 \times 10^{-2}$ & 1.00 & $1.12 \times 10^{-3}$ & 2.00 \\
\ifHAL \hline \fi
\ifHAL \end{tabular} \else \end{tabular*} \fi
\end{table}
\ifHAL \begin{table}[htb] \else \begin{table}[width=.5\linewidth,cols=5,pos=htb] \fi
\ifHAL \centering \fi
\caption{Test case with higher winding number. Discrete energy $E(\Psi^j_{\calT_i})$ for outer iteration index $j\in\{0,\ldots,5\}$ and mesh index $i\in\{1,\ldots,5\}$. The energy at convergence is also reported. The two bottom lines report the total number of Uzawa/ADMM iterations (in units of $10^3$) together with the number of outer iterations (in parentheses) and the final adaptive inner tolerance.}
\label{table:energy-Ex2}
\ifHAL
\begin{tabular}{|c|rrrr|}
\hline
$j\,\backslash\,i$ & \multicolumn{1}{c}{$1$} & \multicolumn{1}{c}{$2$} & \multicolumn{1}{c}{$3$} & \multicolumn{1}{c|}{$4$} \\
\hline
\else
\begin{tabular*}{\tblwidth}{@{}c rrrr@{}}
\toprule
$j\,\backslash\,i$ & \multicolumn{1}{c}{$1$} & \multicolumn{1}{c}{$2$} & \multicolumn{1}{c}{$3$} & \multicolumn{1}{c}{$4$} \\
\midrule
\fi
0  & 72.166 & 93.364 & 103.653 & 107.052 \\
1  & 63.180 & 73.965 &  76.655 &  77.306 \\
2  & 58.892 & 63.651 &  64.135 &  64.298 \\
3  & 56.506 & 58.698 &  59.259 &  59.445 \\
4  & 54.942 & 56.786 &  57.672 &  57.915 \\
5  & 53.750 & 56.136 &  57.170 &  57.440 \\
\ifHAL \hline \fi
cv & 51.449 & 55.747 & 56.884 & 57.172 \\
\ifHAL \hline \fi
Iterations & 0.54\,(42) & 0.37\,(28) & 0.66\,(26) & 2.18\,(25) \\
\ifHAL \hline \fi
Tolerance & $10^{-6}$ & $10^{-6}$ & $10^{-6}$ & $10^{-7}$ \\
\ifHAL \hline \fi
\ifHAL \end{tabular} \else \end{tabular*} \fi
\end{table}

Finally, the histograms shown in Figure \ref{figl:energy-errorEx2} are evaluated from the pointwise logarithmic errors between $\disp^*$ and $\varphi(\disr^*)$ at each interior mesh vertex for the four triangulations. The error distributions are supported in the interval $[-8.38,-5.18]$ for $i=1,$ $[-8.63,-5.60]$ for $i=2,$ $[-9.47,-5.47]$ for $i=3$, and $[-10.31,-6.47]$ for $i=4.$ 
A progressive leftward shift of the histograms is observed, although the behavior is not strictly monotone as smaller pointwise errors are observed on $\calT_1$ than on $\calT_2$ and $\calT_3$, whereas the error on $\calT_4$ reflect a clear improvement. Altogether, the histograms reflect errors that remain compatible with the prescribed tolerance at each refinement level.

\begin{figure}[htb]
\centering
{\includegraphics[width=8cm,height=6cm]{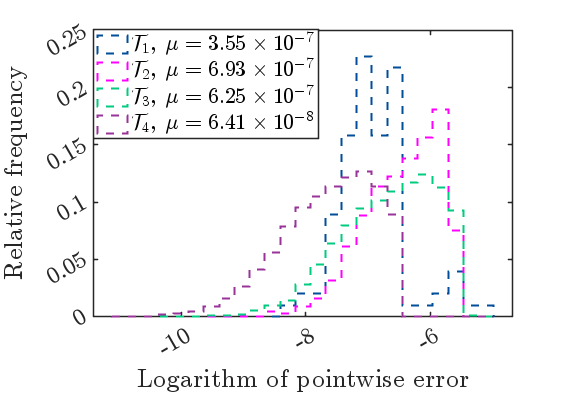}}
\caption{Test case with higher winding number. Histograms of absolute pointwise error (logarithmic  scale) between $\disp^*$ and $\varphi(\disr^*)$ at convergence on the four triangulations. The mean value for each distribution ($\mu$) is reported in the legend.}
\label{figl:energy-errorEx2}
\end{figure}

\subsection{Test case on a geometry with holes}\label{subsec:four-hole}

We consider the domain $\Omega:=(-1,1)^2\setminus \bigcup_{j=1}^{4} A_j,$
where \(A_1,A_2,A_3,A_4\) are four mutually disjoint closed squares with centers located at the points $(-0.5, 0)\tr$, $(0.5, 0)\tr$, $(0,-0.5)\tr$, $(0,0.5)\tr$, respectively, and with edge length equal to $\frac13.$ 
This study of ferronematic medium in a domain $\Omega$ with  four excluded squares is motivated from the experiments in
\cite{Phillips2011}, where the authors study square and
polygonal particle arrangements, including four-particle configurations within the Landau--de Gennes framework for liquid-crystal colloid dispersion. 
For our numerical experiments, we define the rational function 
$F(z):=\frac{(z-a_1)(z-a_2)}
{(z-b_1)(z-b_2)},$ where \(a_1,a_2\) are the zeros and \(b_1,b_2\)  the poles of $F$. In particular, we choose the zeros and poles to be the centers of the four squares $A_j$, so that the two horizontal holes enclose the zeros of $F$, whereas the two vertical holes enclose its poles. 
Figures~\ref{fig:vec-field-Ex3} and~\ref{fig:level-surface-Ex3} display the exact solution.
The magnetic orientation rotates in opposite
senses around the horizontal and vertical pairs of inclusions. The corresponding
nematic profile undergoes twice the angular variation. Since all singular points
of $F$ lie inside the excluded regions, both fields remain smooth and have their
prescribed nonzero lengths throughout the physical domain.  
As above, the exact solution is used only to define
the boundary data and to compute the numerical errors. Moreover,
to construct a nontrivial initial state satisfying the correct boundary data, 
we utilize the blending technique described above, where the blending is performed 
with respect to a complex function with a single zero located at the point $(1.5,0.2)\tr$. 

\begin{figure}[htb]\centering
{\includegraphics[width=0.25\linewidth]{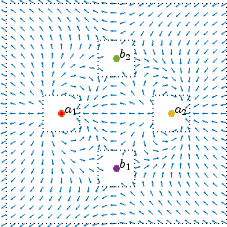}}
\hspace{1 cm}
{\includegraphics[width=0.245\linewidth]{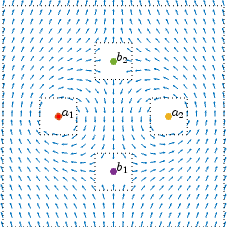}}
\caption{Test case on a geometry with holes. Nematic profile $(Q_{11},Q_{12})\tr$ (left) and magnetic profile $\Mvec$ (right) for the exact solution $\Psi$ on the domain with four holes.}
\label{fig:vec-field-Ex3}
\end{figure}
\begin{figure}[htb]\centering
\subfloat[$Q_{11}$]{\includegraphics[width=0.22\linewidth]
{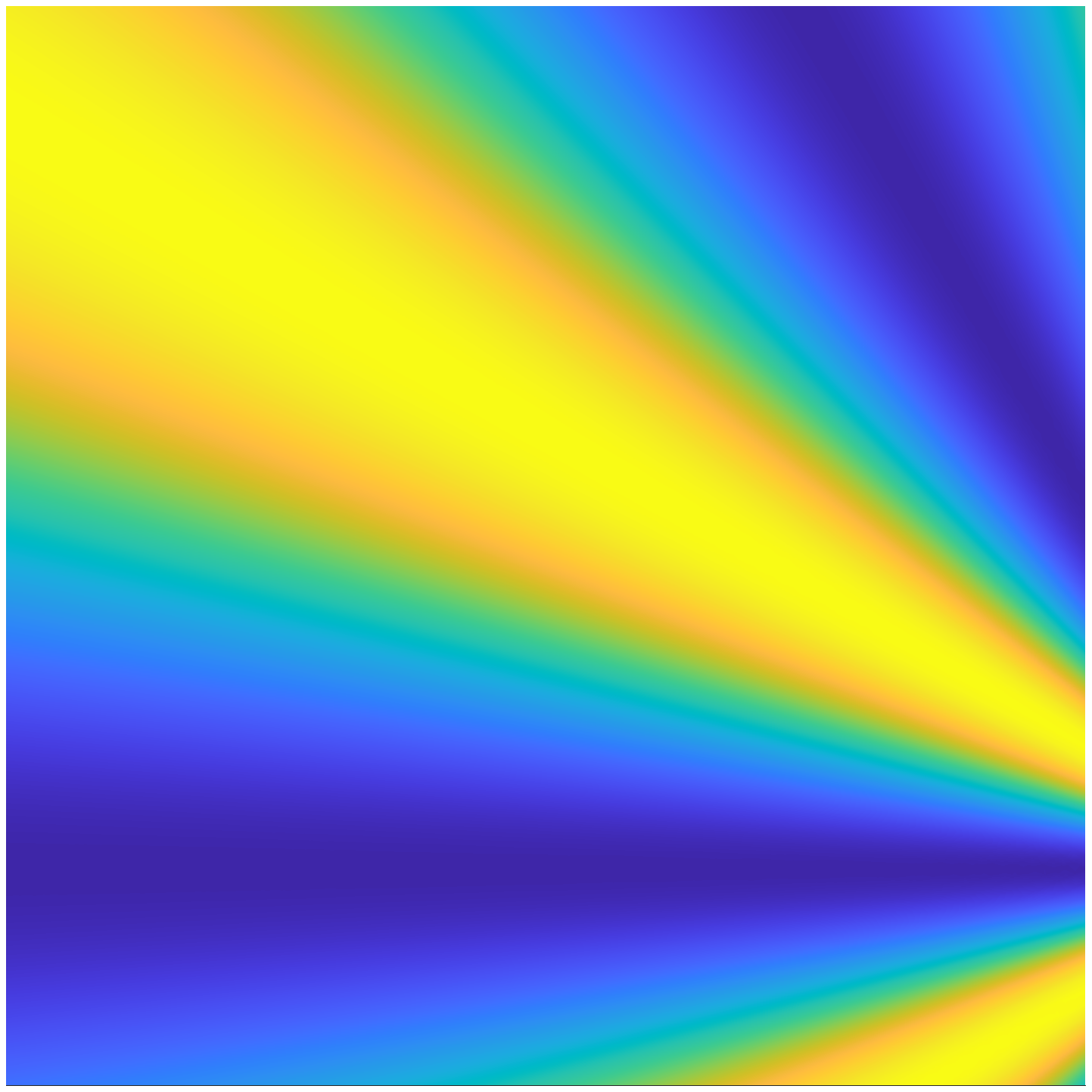}}
\hspace{0.2 cm}
\subfloat[$Q_{12}$]{\includegraphics[width=0.22\linewidth]
{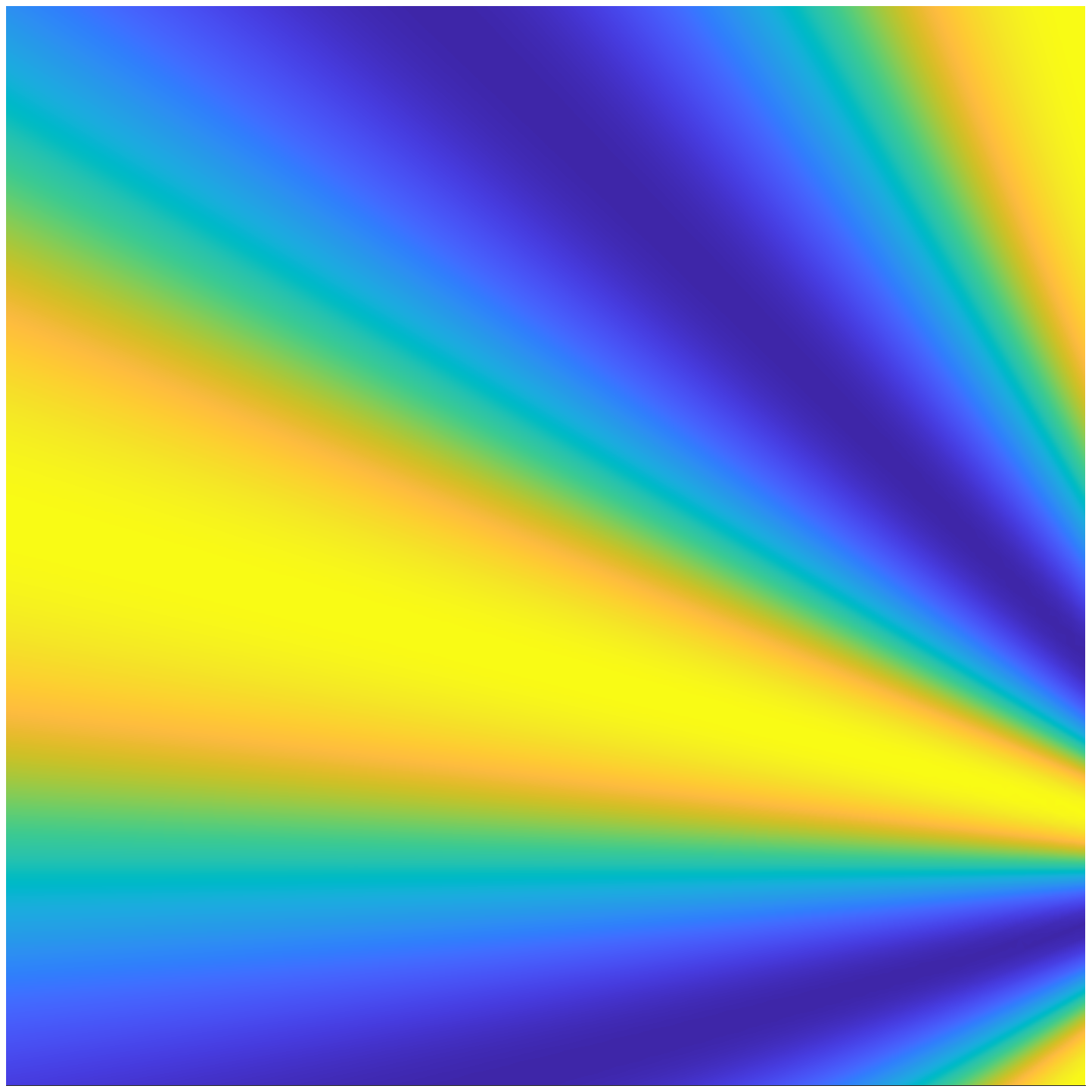}}
\hspace{0.2 cm}
\subfloat[$M_1$]{\includegraphics[width=0.22\linewidth]
{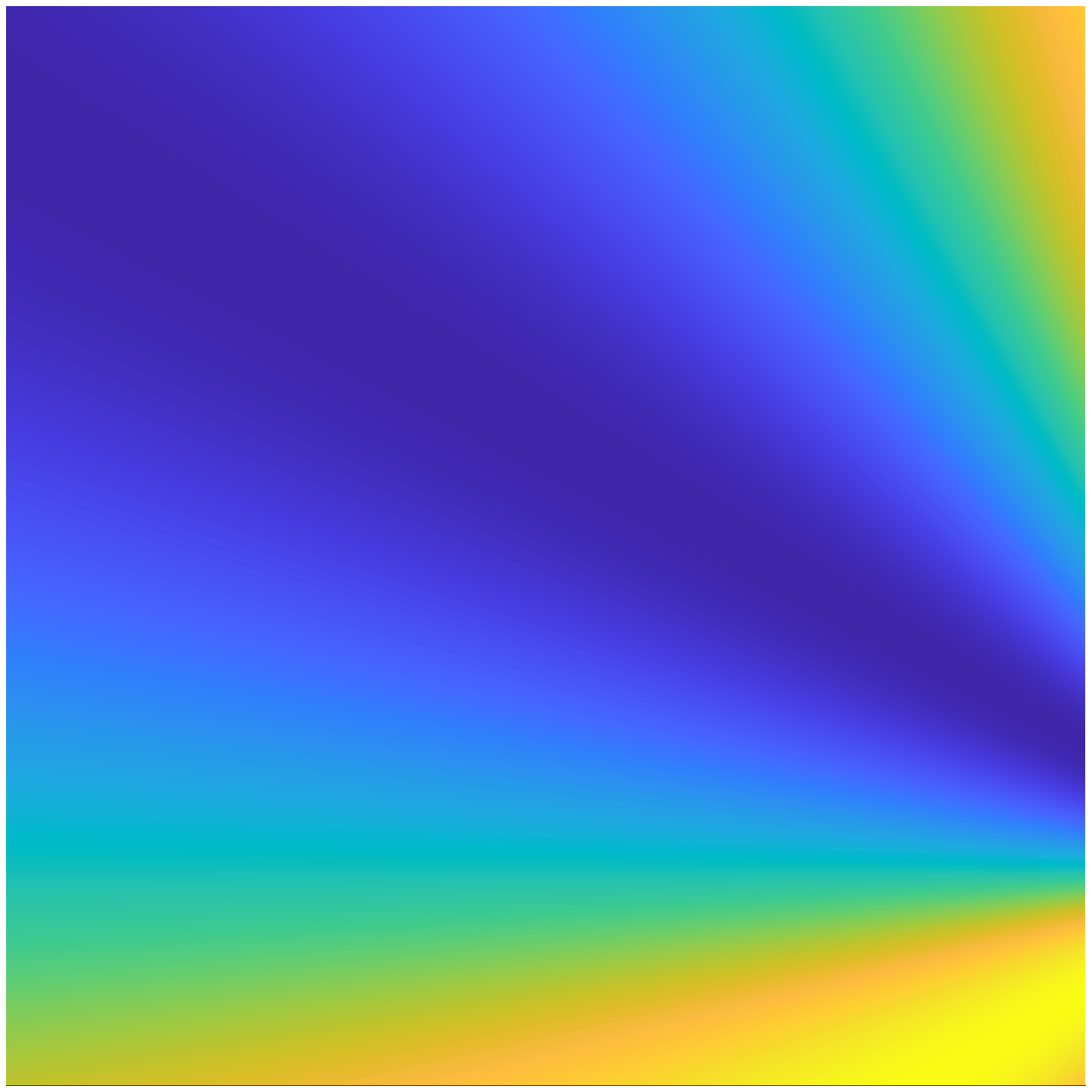}}
\hspace{0.2 cm}
\subfloat[$M_2$]{\includegraphics[width=0.22\linewidth]
{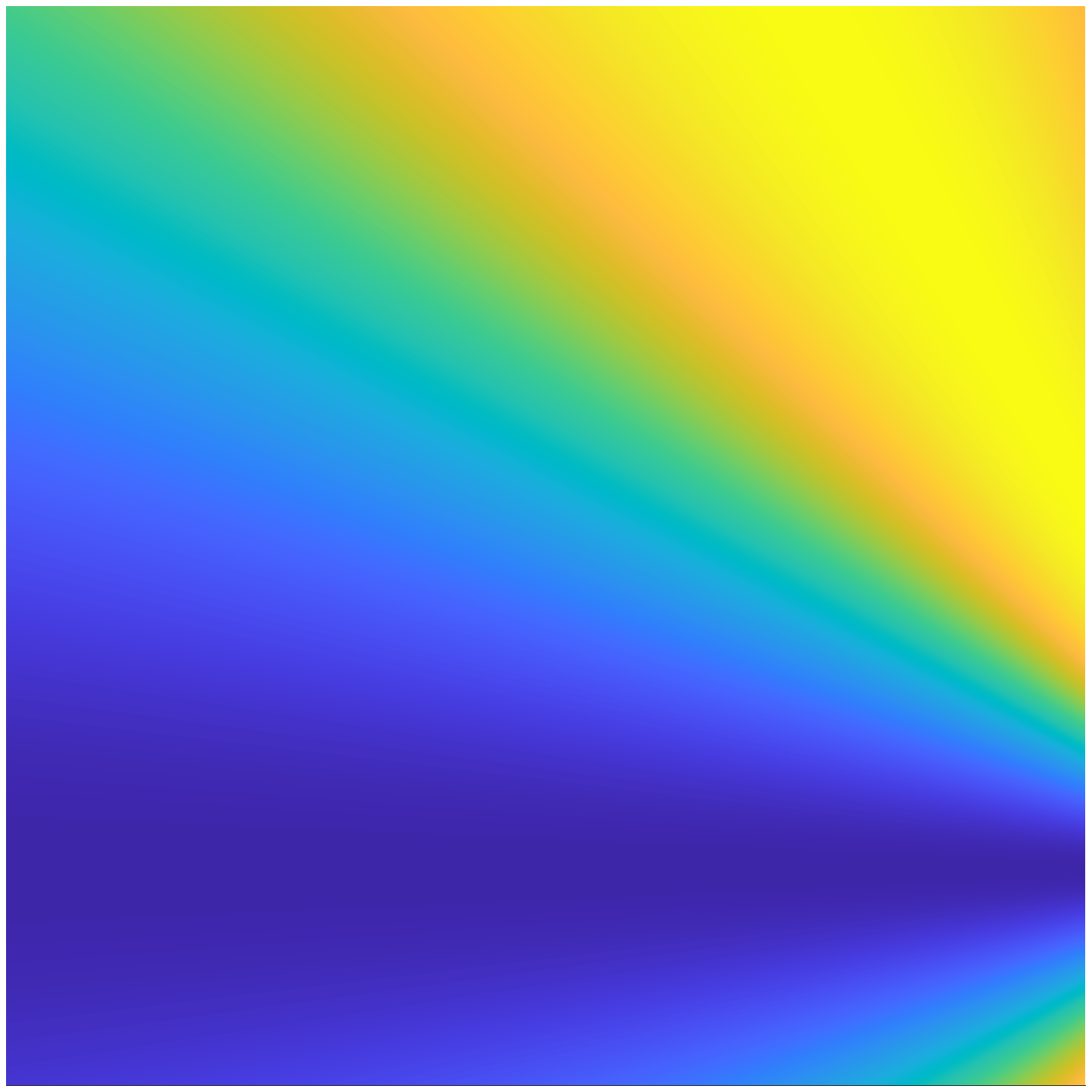}}
\caption{Test case on a geometry with holes. Exact solution $\Psi=({Q_{11},Q_{12},M_1,M_2})\tr.$}
\label{fig:level-surface-Ex3}
\end{figure}

This test case is challenging as the variations of the solution near the four holes need to be well-resolved. We consider a sequence of weakly acute triangulations with mesh sizes $h\in\{0.2357,0.1179,0.0589,0.0295,0.0147\}$. 
Table~\ref{table:three-ener-Ex3} reports the energies of the exact solution, $\Psi$, of its nodal interpolant, $\Psi^I$, and of the discrete solution $\Psi^{\rm cv}_{\calT_i}$ at convergence, on the five triangulations. As in the previous test case, the energy of the nodal interpolant, $E(\Psi^I)$, is smaller than the exact energy $E(\Psi)$ (even on the finest mesh), and, as expected, the discrete energy $E(\Psi^{\rm cv}_{\calT_i})$ is (slightly) smaller than $E(\Psi^I)$. The two energies actually coincide in all displayed digits on $\calT_5$. The relative energy error $\Delta E_i:=1-E(\Psi^{\rm cv}_{\calT_i})/E(\Psi)$ exhibits second-order convergence, with some super-convergence in the pre-asymptotic regime on the coarser meshes. 
Table~\ref{table:three-ener-Ex3} also reports the relative errors on $\Psi^{\rm cv}_{\calT_i}$ in the $H^1$-seminorm and in the $L^2$-norm, yielding first- and second-order convergence, as in the previous test cases.

The computed discrete energies $E(\Psi^j_{\calT_i})$ over the first five outer iterations, $j\in\{0,\ldots,5\}$, are reported in Table~\ref{table:energy-iter-Ex3} on the five triangulations. We also report the discrete energy at convergence. We observe that the energy decreases in all cases. Here, in contrast to the baseline test case and the test case with higher winding number, the energy of the initial state grows substantially under mesh refinement. The reason is that the blended initial state is not very smooth  to be well-resolved by the coarser meshes. 
Table~\ref{table:energy-iter-Ex3} also reports the total number of Uzawa/ADMM iterations (in units of $10^3$) and the number of outer iterations (between parentheses), as well as the final adaptive inner tolerance. Finally, Figure~\ref{figl:HistEx3} shows histograms of the pointwise logarithmic errors $\log_{10}|\disp^*-\varphi(\disr^*)|$ at each interior mesh vertex for the five triangulations. Similar conclusions to the previous test cases can be drawn. 

\ifHAL \begin{table}[htb] \else \begin{table}[width=0.9\linewidth,cols=11,pos=htb] \fi
\ifHAL \centering \fi
\caption{Test case on a geometry with holes. Exact energy $E(\Psi)$, energy for the nodal interpolation of the exact solution, $E(\Psi^I)$, computed energy $E(\Psi^{\rm cv}_{\calT_i})$ at convergence, relative energy error $\Delta E_{i}$, and relative errors in $H^1$-seminorm and $L^2$-norm on $\Psi^{\rm cv}_{\calT_i}$, together with convergence rates, for the five triangulations. The second column (kDofs) reports the total number of Dofs (4 times the number of mesh vertices) in units of $10^3$.}
\label{table:three-ener-Ex3}\label{Table:ex3-error-estimates}
\small
\ifHAL
\begin{tabular}{|c|r|ccc|cc|cc|cc|}
\hline
\else
\begin{tabular*}{\tblwidth}{@{}cr ccc cc cc cc@{}}
\toprule
\fi
$i$ & kDofs & $E(\Psi)$ & $E(\Psi^I)$ & $E(\Psi^{\rm cv}_{\calT_i})$ & $\Delta E_{i}$ & rate & $H^1$-seminorm & rate & $L^2$-norm & rate \\
\ifHAL \hline \fi
1 & 0.66   & 58.251 & 55.920 & 55.846 & $4.13\times10^{-2}$ & --   & $4.45\times10^{-1}$ & --   & $7.96\times10^{-2}$ & --   \\
2 & 2.36 & 58.251 & 57.977 & 57.962 & $4.97\times10^{-3}$ & 3.06 & $2.43\times10^{-1}$ & 0.88 & $2.07\times10^{-2}$ & 1.94 \\
3 & 8.82 & 58.251 & 58.212 & 58.211 & $6.98\times10^{-4}$ & 2.83 & $1.25\times10^{-1}$ & 0.96 & $5.21\times10^{-3}$ & 1.99 \\
4 & 34.04 & 58.251 & 58.244 & 58.243 & $1.33\times10^{-4}$ & 2.39 & $6.28\times10^{-2}$ & 0.99 & $1.30\times10^{-3}$ & 2.00 \\
5 & 133.62 & 58.251 & 58.249 & 58.249 & $3.06\times10^{-5}$ & 2.12 & $3.14\times10^{-2}$ & 1.00 & $3.27\times10^{-4}$ & 2.00 \\
\ifHAL \hline \fi
\ifHAL \end{tabular} \else \end{tabular*} \fi
\end{table}
\ifHAL \begin{table}[htb] \else \begin{table}[width=.6\linewidth,cols=6,pos=htb] \fi
\ifHAL \centering \fi
\caption{Test case on a geometry with holes. Discrete energy $E(\Psi^j_{\calT_i})$ for outer iteration index $j\in\{0,\ldots,5\}$ and mesh index $i\in\{1,\ldots,5\}$. The energy at convergence is also reported. The two bottom lines report the total number of Uzawa/ADMM iterations (in units of $10^3$) together with the number of outer iterations (in parentheses) and the final adaptive inner tolerance.}
\label{table:energy-iter-Ex3}
\ifHAL
\begin{tabular}{|c|rrrrr|}
\hline
$j\,\backslash\,i$ & \multicolumn{1}{c}{$1$} & \multicolumn{1}{c}{$2$} & \multicolumn{1}{c}{$3$} & \multicolumn{1}{c}{$4$} & \multicolumn{1}{c|}{$5$} \\
\hline
\else
\begin{tabular*}{\tblwidth}{@{}c rrrrr@{}}
\toprule
$j\,\backslash\,i$ & \multicolumn{1}{c}{$1$} & \multicolumn{1}{c}{$2$} & \multicolumn{1}{c}{$3$} & \multicolumn{1}{c}{$4$} & \multicolumn{1}{c}{$5$} \\
\midrule
\fi
0  & 137.304 & 185.307 & 244.482 & 333.019 & 491.235 \\
1  & 100.209 & 130.035 & 170.545 & 236.902 & 359.656 \\
2  &  84.093 & 105.004 & 139.378 & 197.327 & 319.442 \\
3  &  78.501 &  96.806 & 129.669 & 185.999 & 306.784 \\
4  &  75.540 &  94.442 & 126.423 & 182.826 & 302.002 \\
5  &  73.232 &  93.715 & 124.987 & 181.297 & 299.344 \\
\ifHAL \hline \fi
cv & 55.846 & 57.962 & 58.211 & 58.243 & 58.249 \\
\ifHAL \hline \fi
Iterations & 0.47\,(20)& 0.68\,(31) & 1.10\,(40) & 3.64\,(66) & 21.68\,(174) \\ 
\ifHAL \hline \fi
Tolerance & $10^{-6}$& $10^{-6}$ & $10^{-6}$ & $10^{-7}$ & $10^{-8}$ \\
\ifHAL \hline \fi
\ifHAL \end{tabular} \else \end{tabular*} \fi
\end{table}
\begin{figure}[htb]
\centering
{\includegraphics[width=8cm,height=6cm]{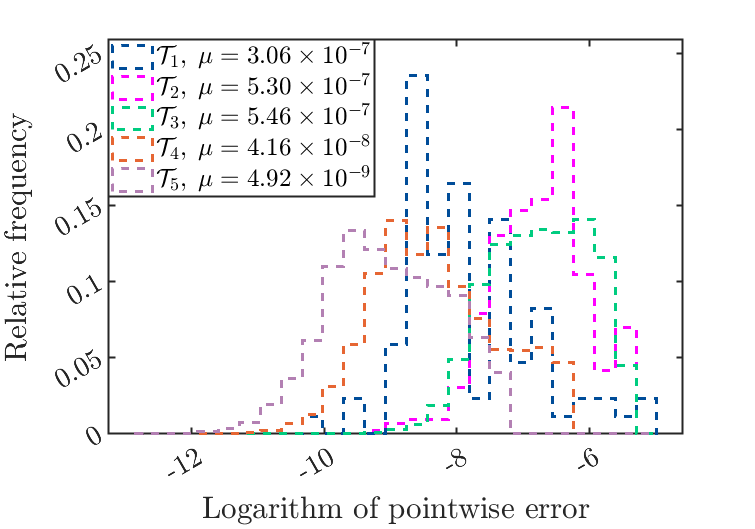}}
\caption{Test case on a geometry with holes. Histograms of absolute pointwise error (logarithmic  scale) between $\disp^*$ and $\varphi(\disr^*)$ at convergence on the five triangulations. The mean value for each distribution ($\mu$) is reported in the legend.}
\label{figl:HistEx3}
\end{figure}

\subsection{Test case with multiple solutions}

We consider the square annular domain $\Omega:=(-1,1)^2\setminus [-1/3,1/3]^2$. Its boundary consists of two disjoint connected components, the outer boundary $\partial\Omega_{\rm out}:=\partial(-1,1)^2$ and the inner boundary $\partial\Omega_{\rm in}:=\partial(-1/3,1/3)^2$. 
We construct two minimizers, $\Psi_\pm$, of the elastic (harmonic) energy as follows.
We let $H$ be the harmonic function in $\Omega$ so that $H=0$ on $\partial\Omega_{\rm out}$ and $H=1$ on $\partial\Omega_{\rm in}$. Then, for all $(x,y)\in \Omega$, setting $z=x+ i y$, we
let $\Psi_\pm(x,y) := \calG(\bn(z))$ with the unit vector $\bn(z)$ identified with the complex number $\frac{z}{|z|}\exp(\pm i\pi H(z))$. One can verify that the magnetization vector
points away from the origin on $\partial\Omega_{\rm out}$ and towards the origin on $\partial\Omega_{\rm in}$. Consequently, in any admissible state, the magnetization vector
has to rotate by an odd multiple of $\pi$ between the outer and the inner boundary.
The main purpose of this example is to address the applicability of our algorithm to approximate multiple local minimizers. 

We consider five weakly acute structured triangulations of size $h\in \{0.1179,0.0589,0.0295,0.0147,0.0074\}$. The exact solutions $\Psi_\pm$ are used, as above, only to prescribe the Dirichlet boundary data. We generate an initial guess on each mesh $\calT_i$ by perturbing either $\Psi_+$or $\Psi_-$. To this purpose, for each $i\in\{1,\ldots,5\}$, and for every interior mesh vertex with coordinates $(x,y)\in\Omega$, we set
$\Psi_{\calT_i}^{0,\pm}(x,y) := \calG(\bn^{0,\pm}_{\calT_i}(z))$ with the unit vector $\bn^{0,\pm}_{\calT_i}(z)$ identified with the complex number $\frac{z}{|z|}\exp(\pm i(\pi H_{\calT_i}(z)+aw_{\calT_i}(z)))$,
where $a$ is a parameter, $H_{\calT_i}$ is the discrete $\mathbb{P}_1$-finite element harmonic function on the mesh $\calT_i$ and $w_\calT(z):=H_\calT(z)\big(1-H_\calT(z)\big)\sin(\pi z_1)\sin(\pi z_2)$. We always use $\zeta=\rho:=1$ in the Uzawa/ADMM iterations.
Figur{e} \ref{fig:Ex4-solutions} presents the four components of the two minimizers $\Psi_+$ (second row) and $\Psi_-$ (fourth row) computed on the finest mesh $\calT_5$, as well as the four components of the two intial guesses using a positive sign (first row) and a negative sign (third row); in both cases, the perturbation parameter is $a=1$.

\begin{figure}[htb]\centering
\subfloat[$Q_{11}$]{\includegraphics[width=0.2\linewidth]{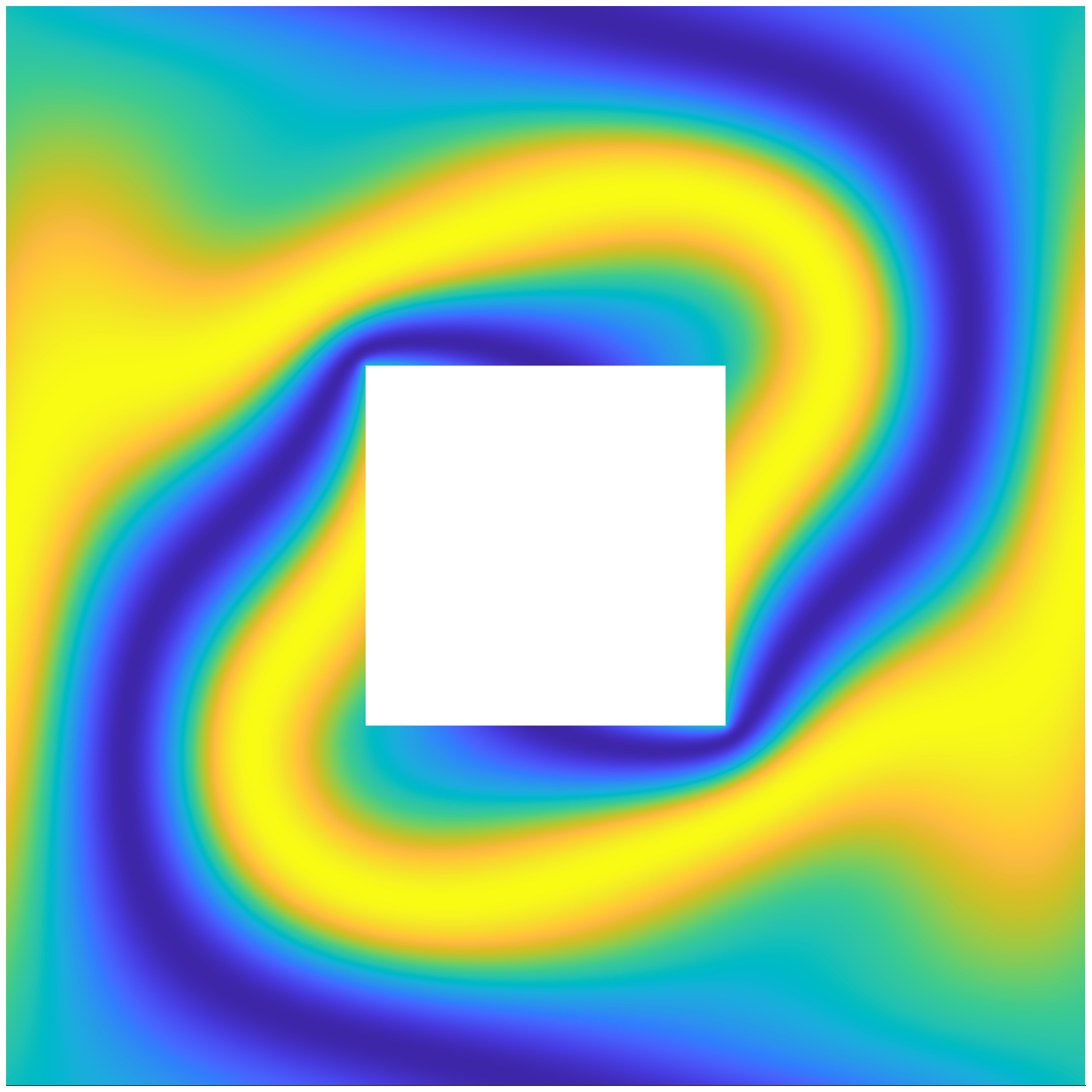}}\hspace{0.3 cm}
\subfloat[$Q_{12}$]{\includegraphics[width=0.2\linewidth]{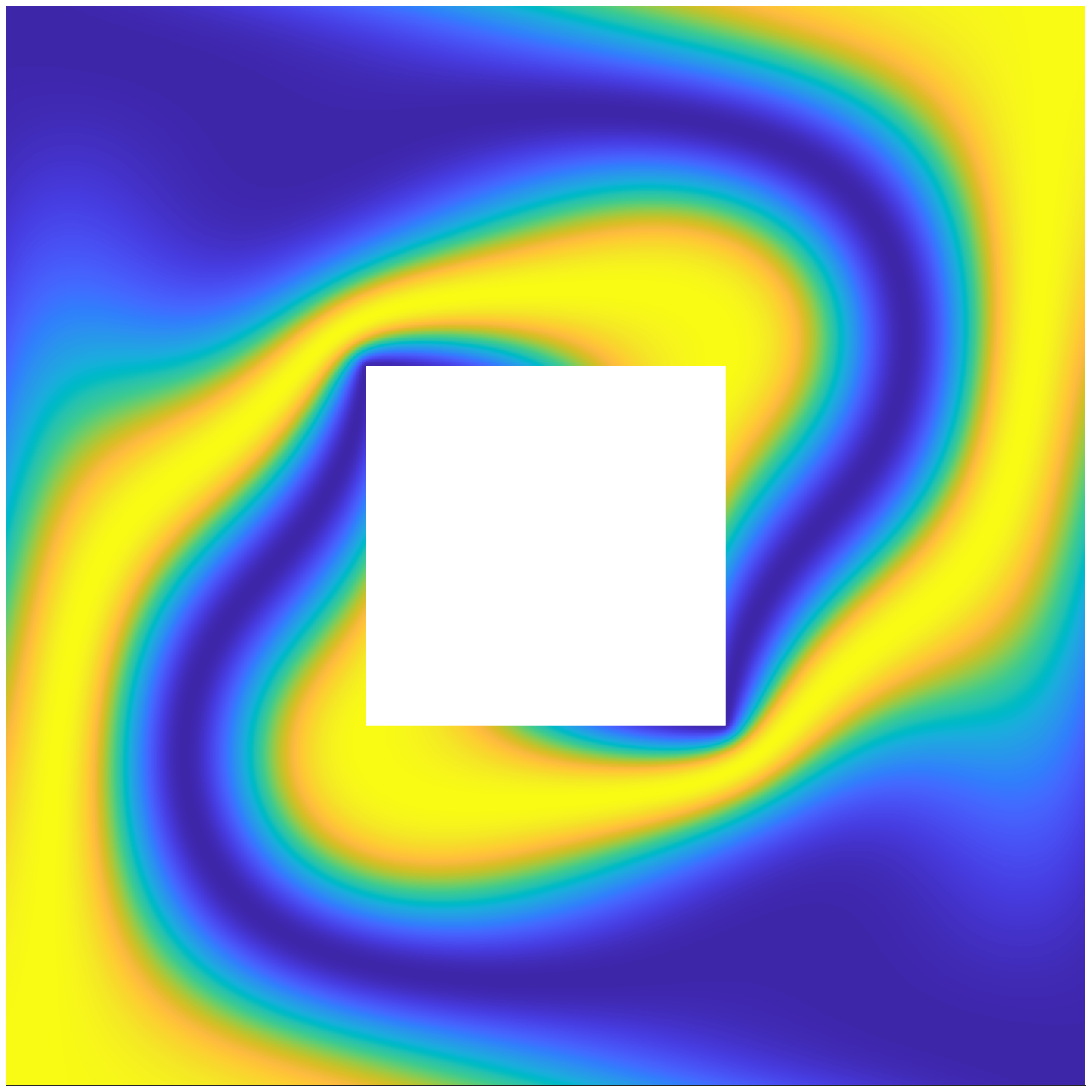}}\hspace{0.3 cm}
\subfloat[$M_1$]{\includegraphics[width=0.2\linewidth]{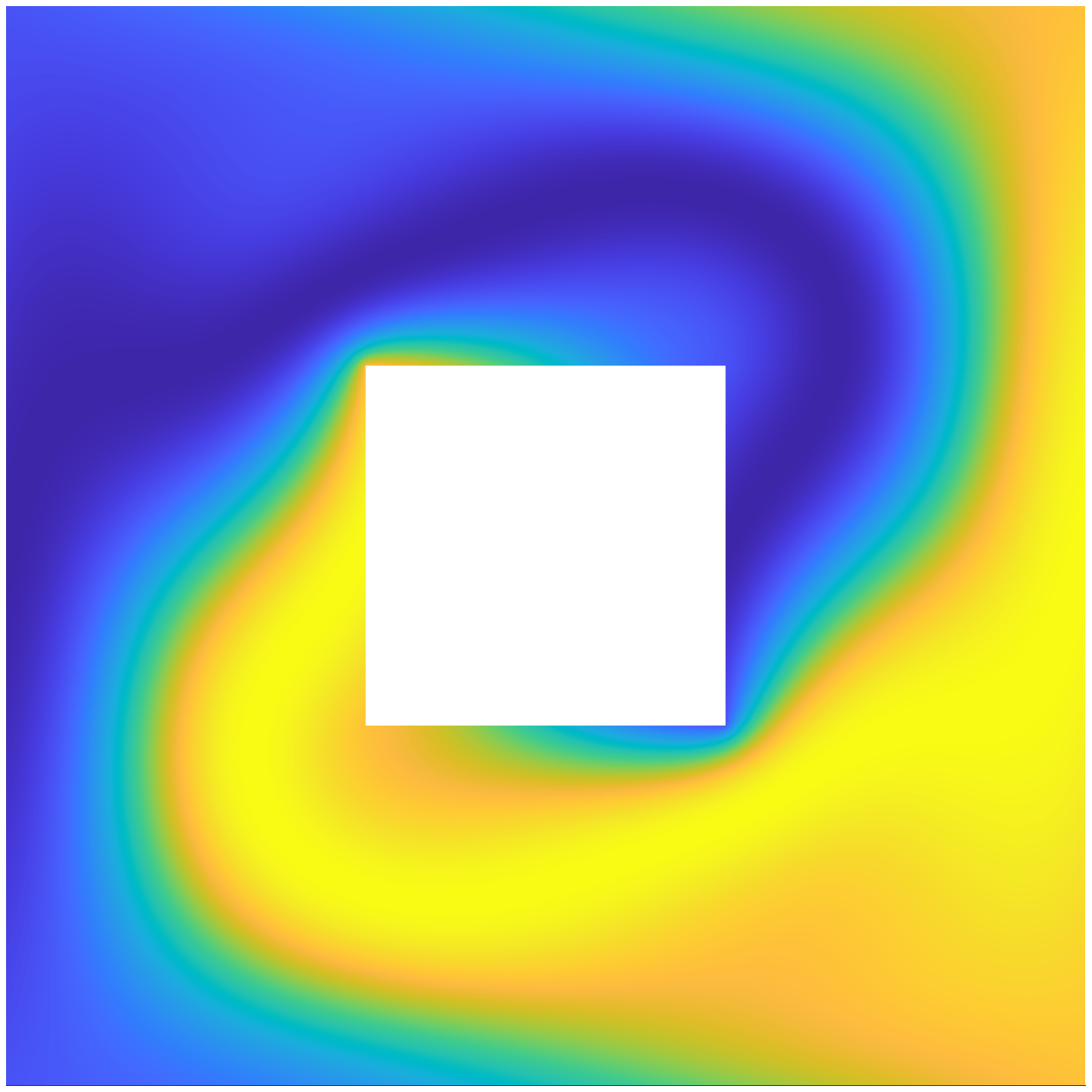}}\hspace{0.3 cm}
\subfloat[$M_2$]{\includegraphics[width=0.2\linewidth]{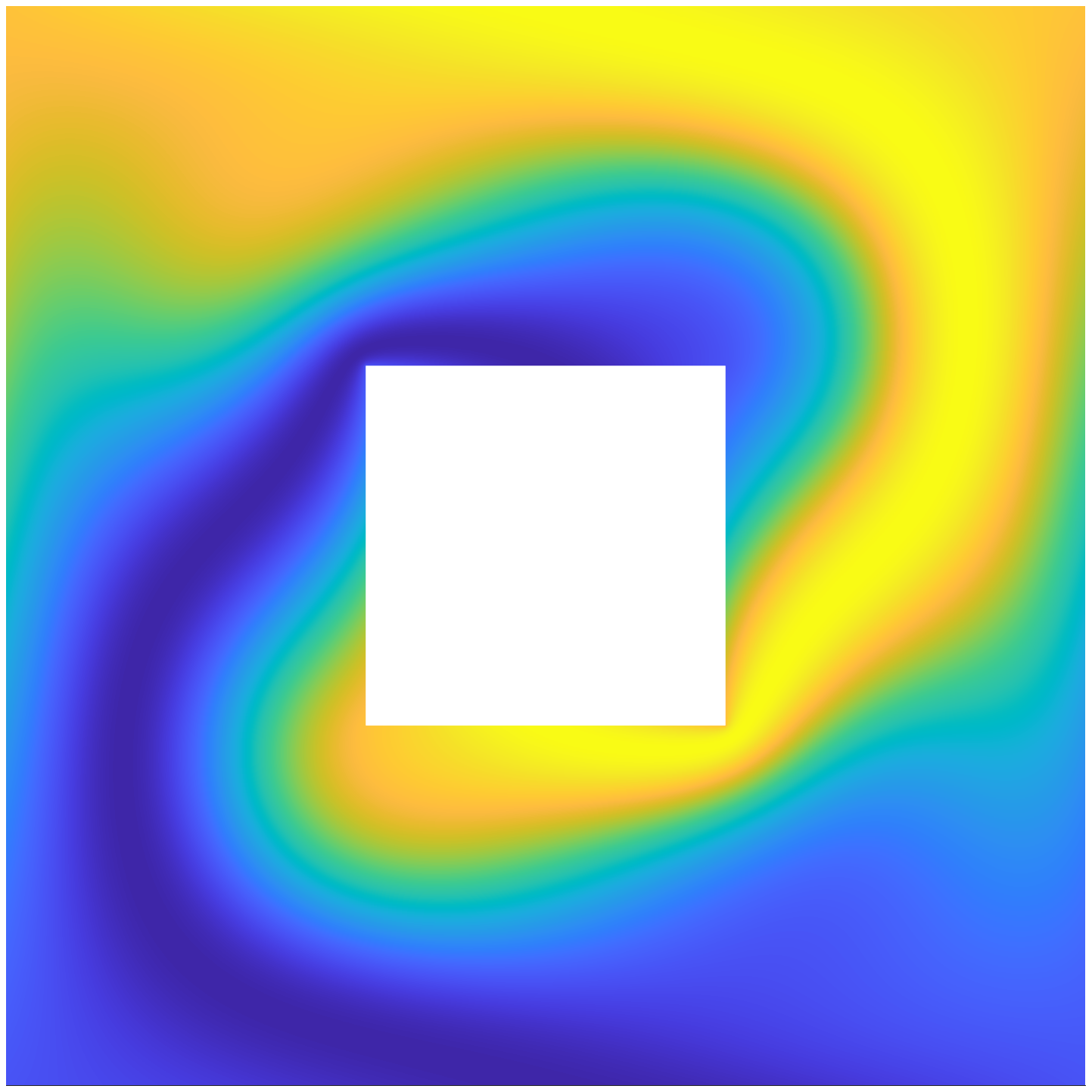}}
\\
\subfloat[$Q_{11}$]{\includegraphics[width=0.2\linewidth]{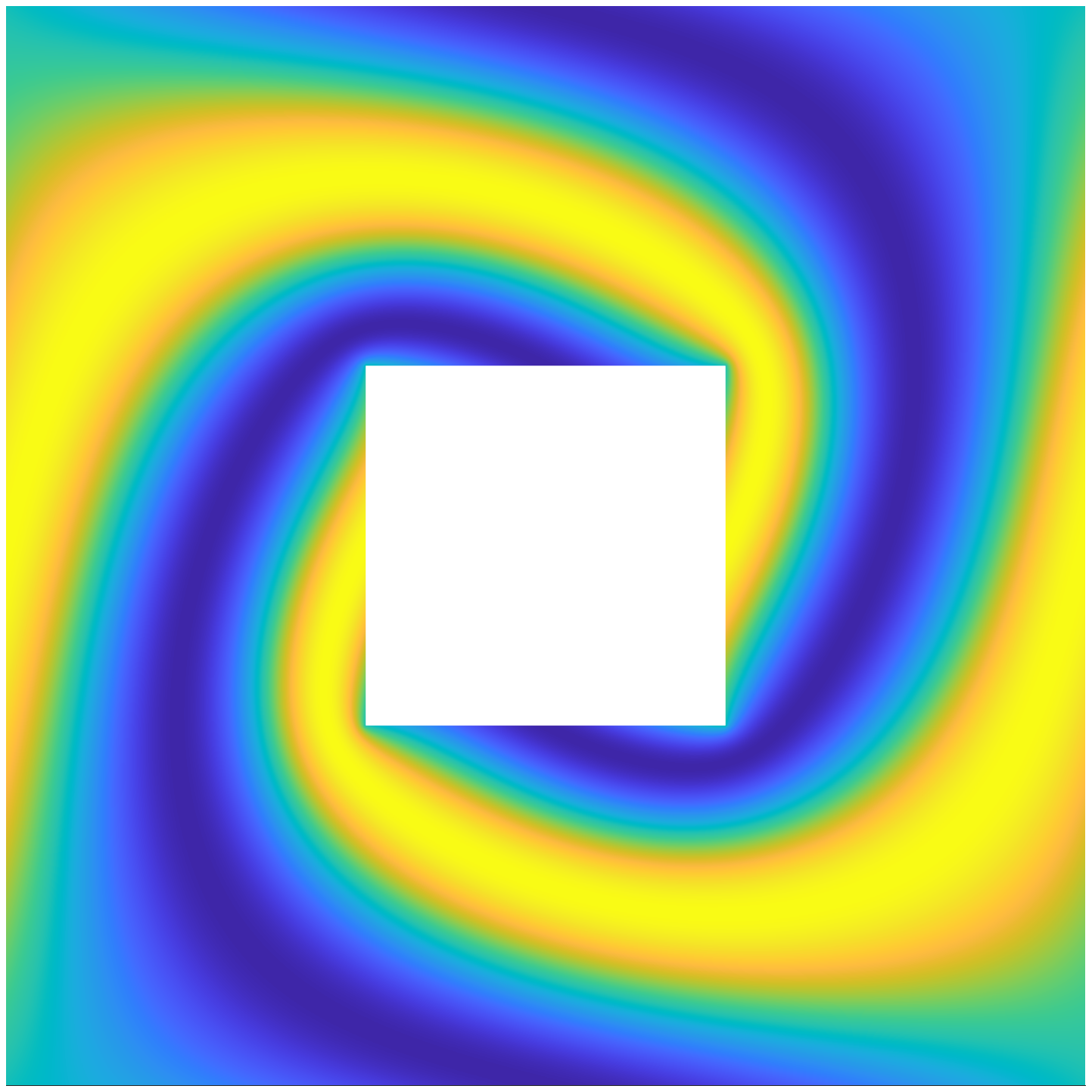}}\hspace{0.3 cm}
\subfloat[$Q_{12}$]{\includegraphics[width=0.2\linewidth]{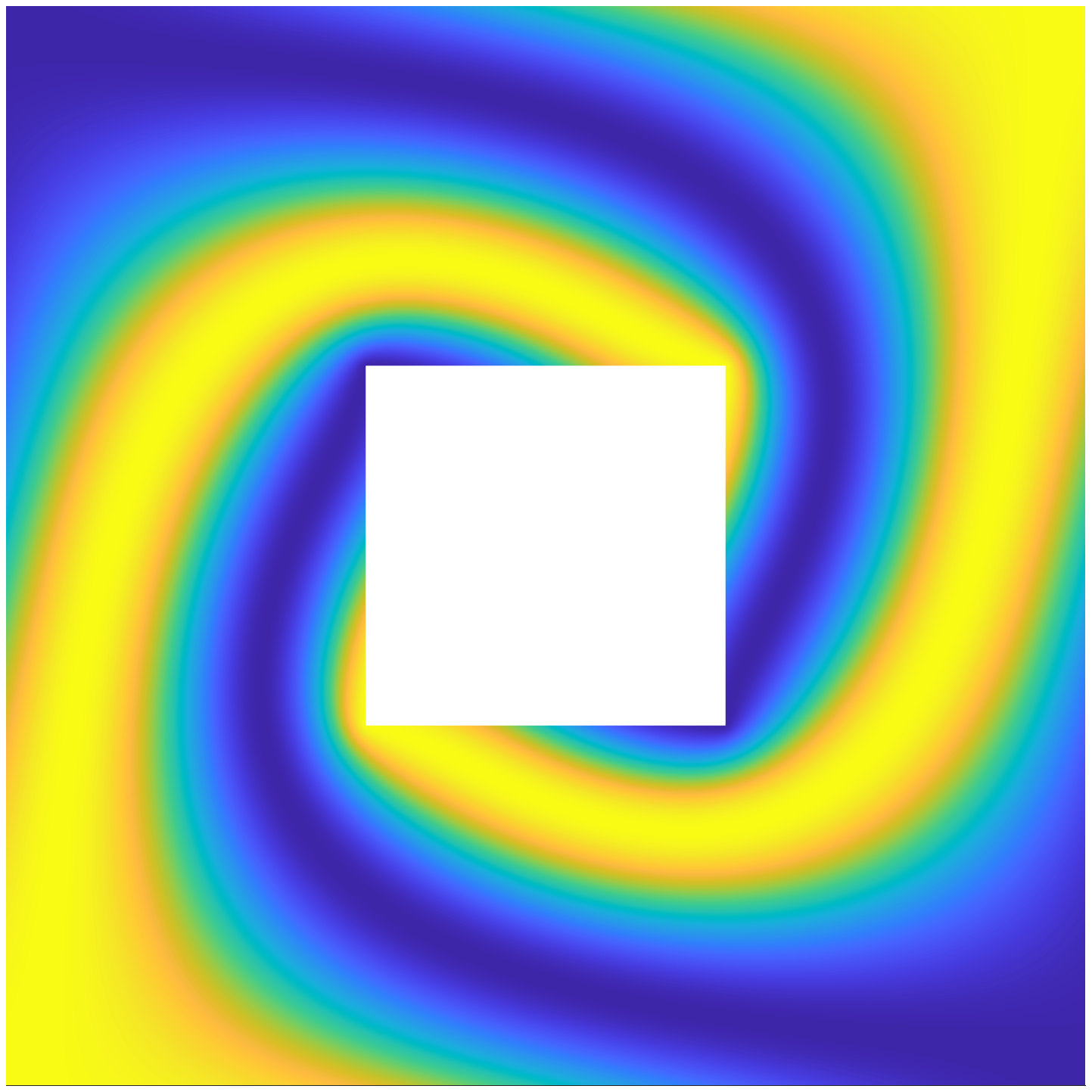}}\hspace{0.3 cm}
\subfloat[$M_1$]{\includegraphics[width=0.2\linewidth]{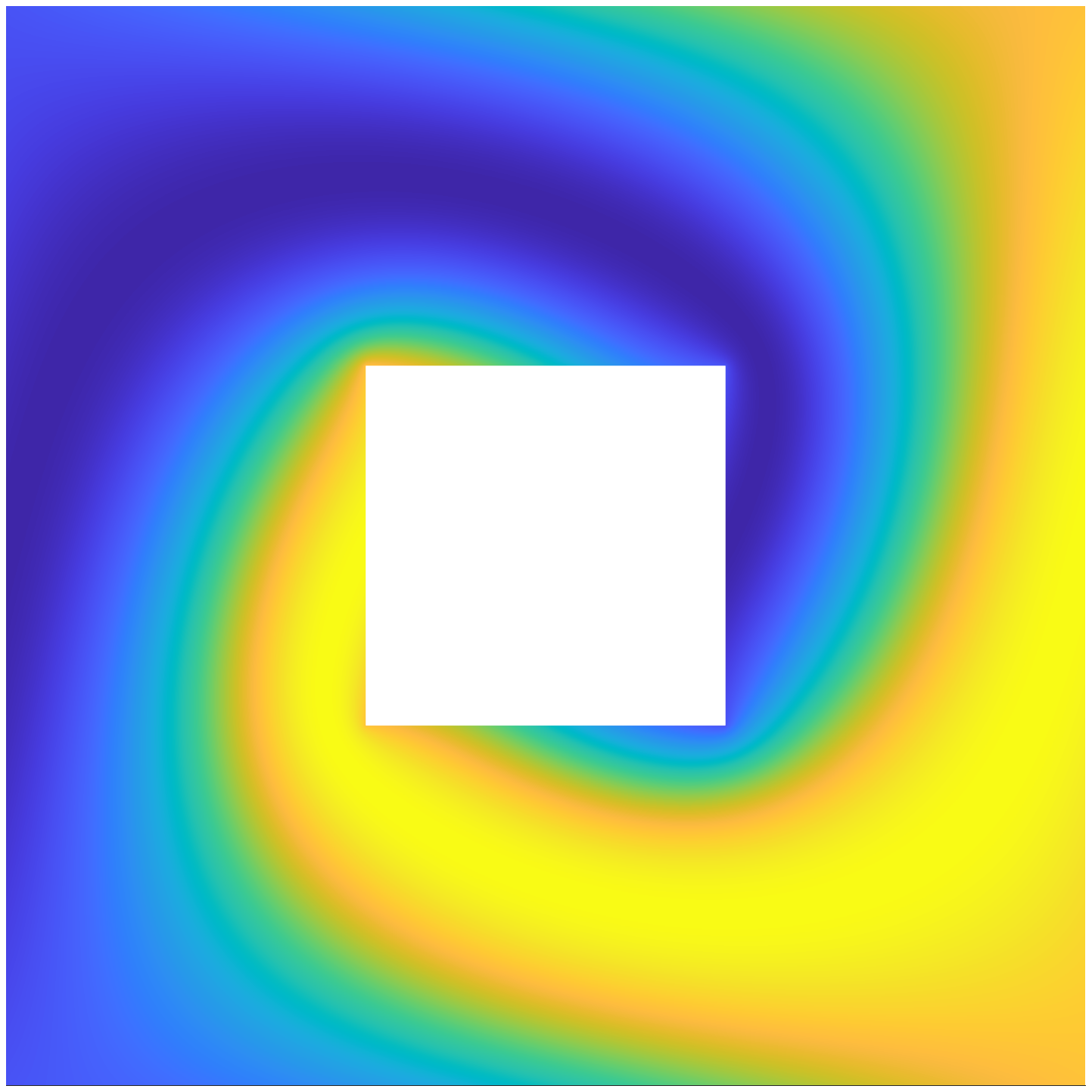}}\hspace{0.3 cm}
\subfloat[$M_2$]{\includegraphics[width=0.2\linewidth]{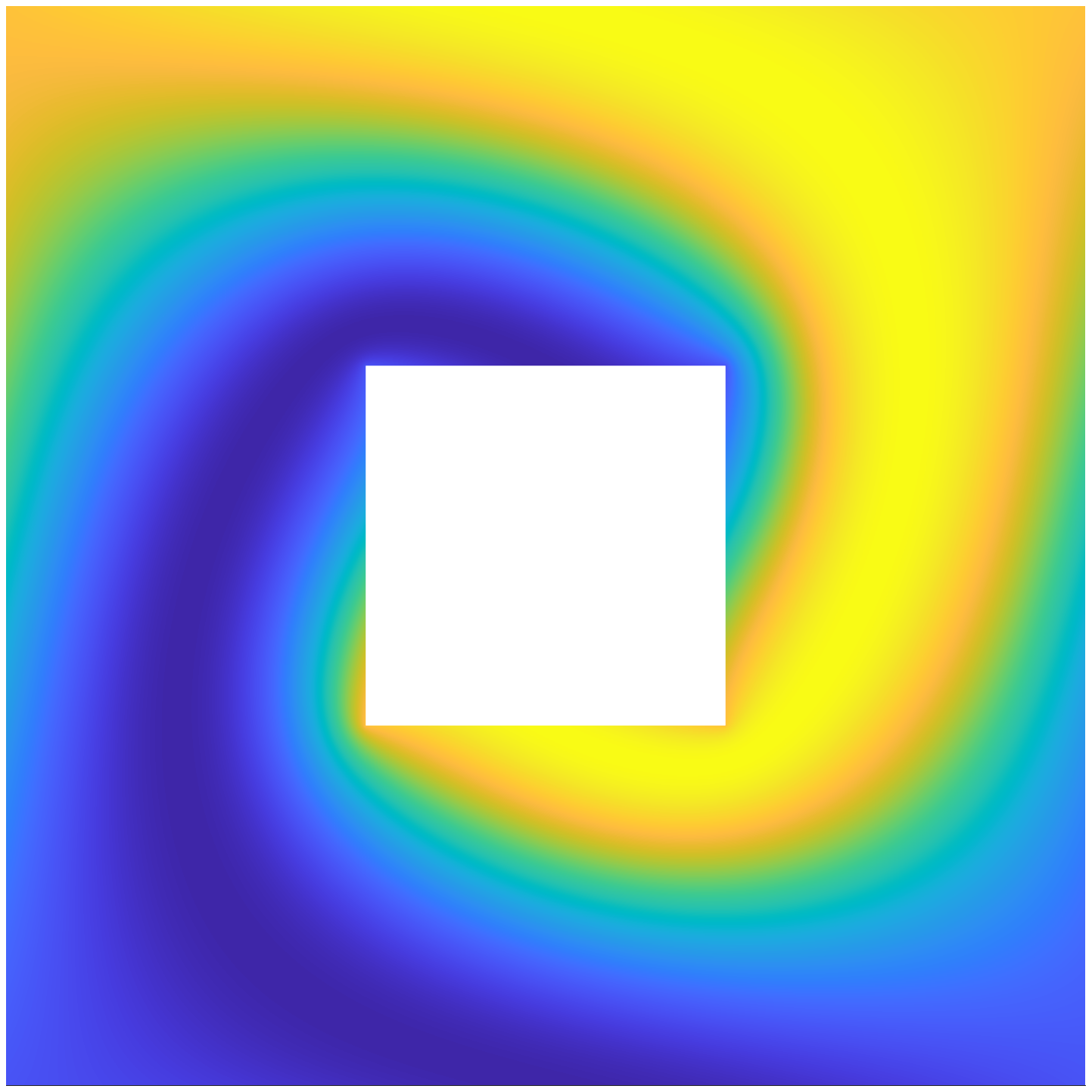}}
\\
\subfloat[$Q_{11}$]{\includegraphics[width=0.2\linewidth]{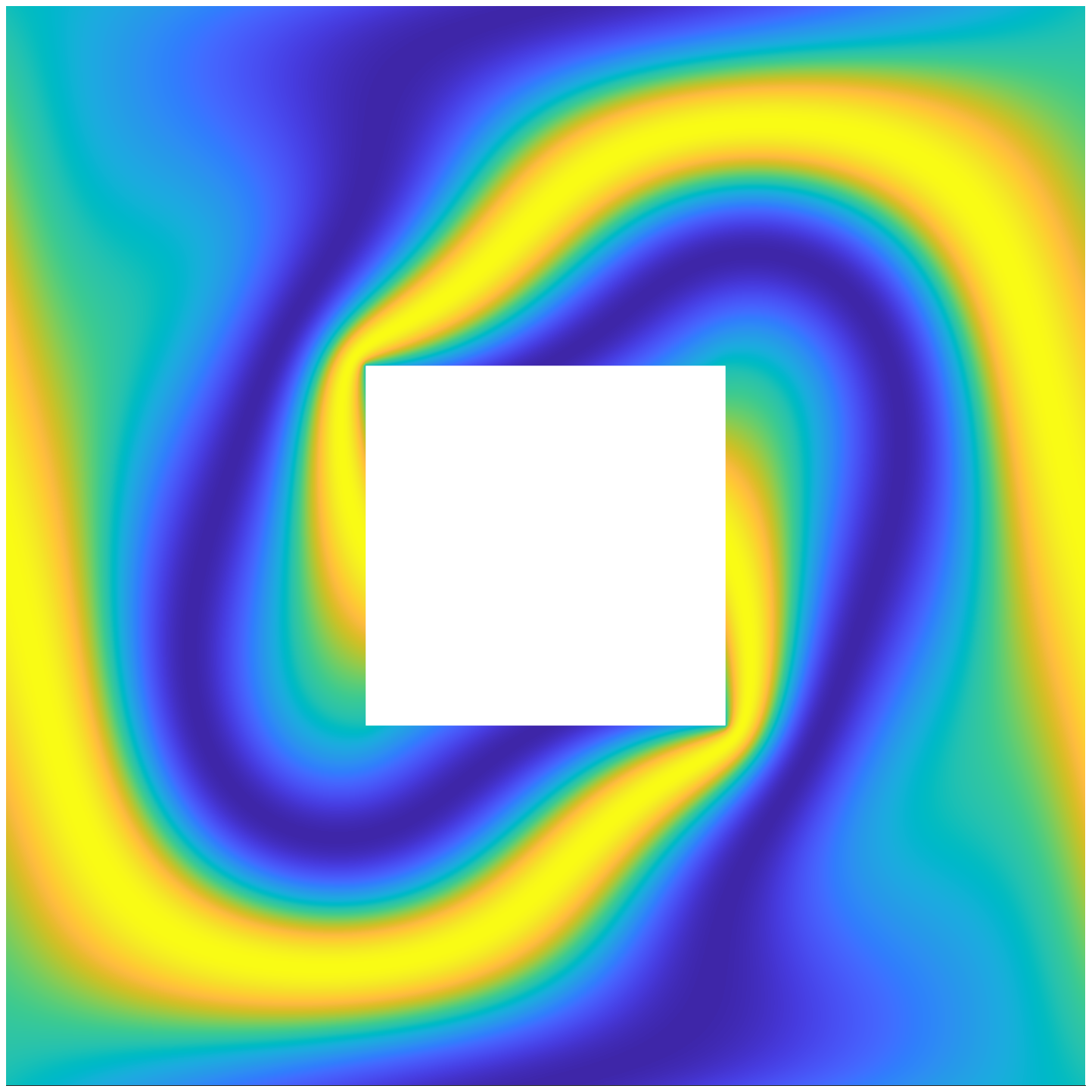}}\hspace{0.3 cm}
\subfloat[$Q_{12}$]{\includegraphics[width=0.2\linewidth]{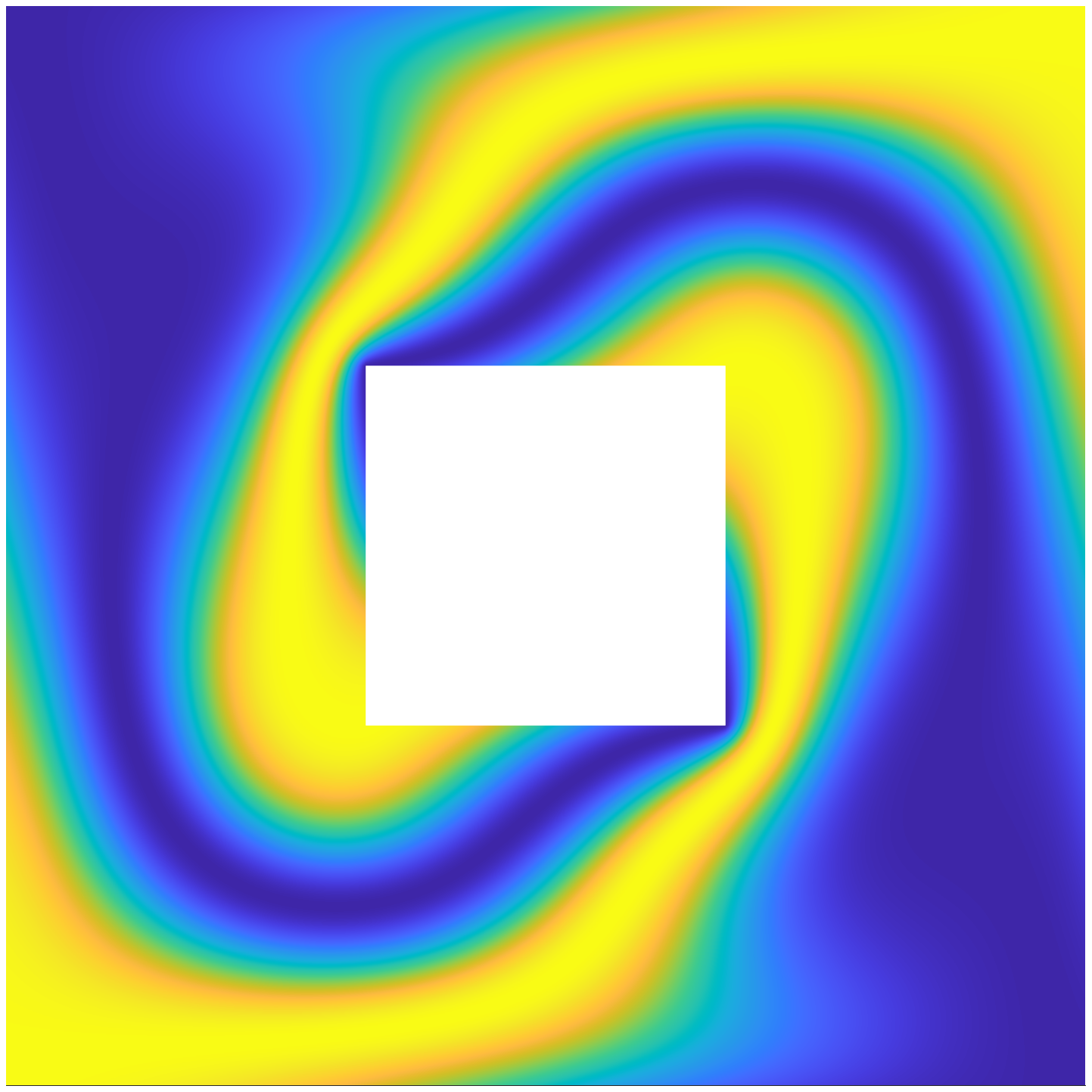}}\hspace{0.3 cm}
\subfloat[$M_1$]{\includegraphics[width=0.2\linewidth]{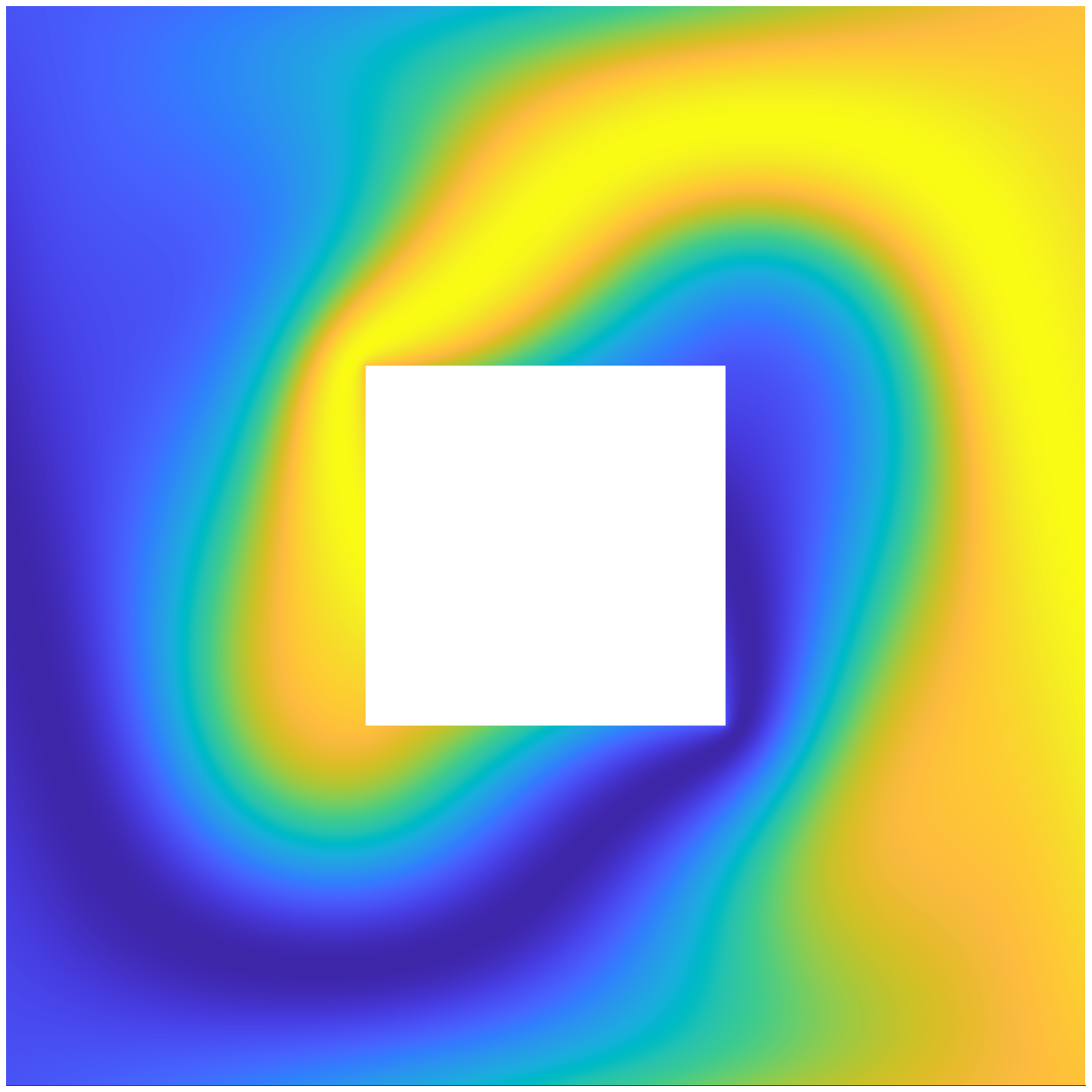}}\hspace{0.3 cm}
\subfloat[$M_2$]{\includegraphics[width=0.2\linewidth]{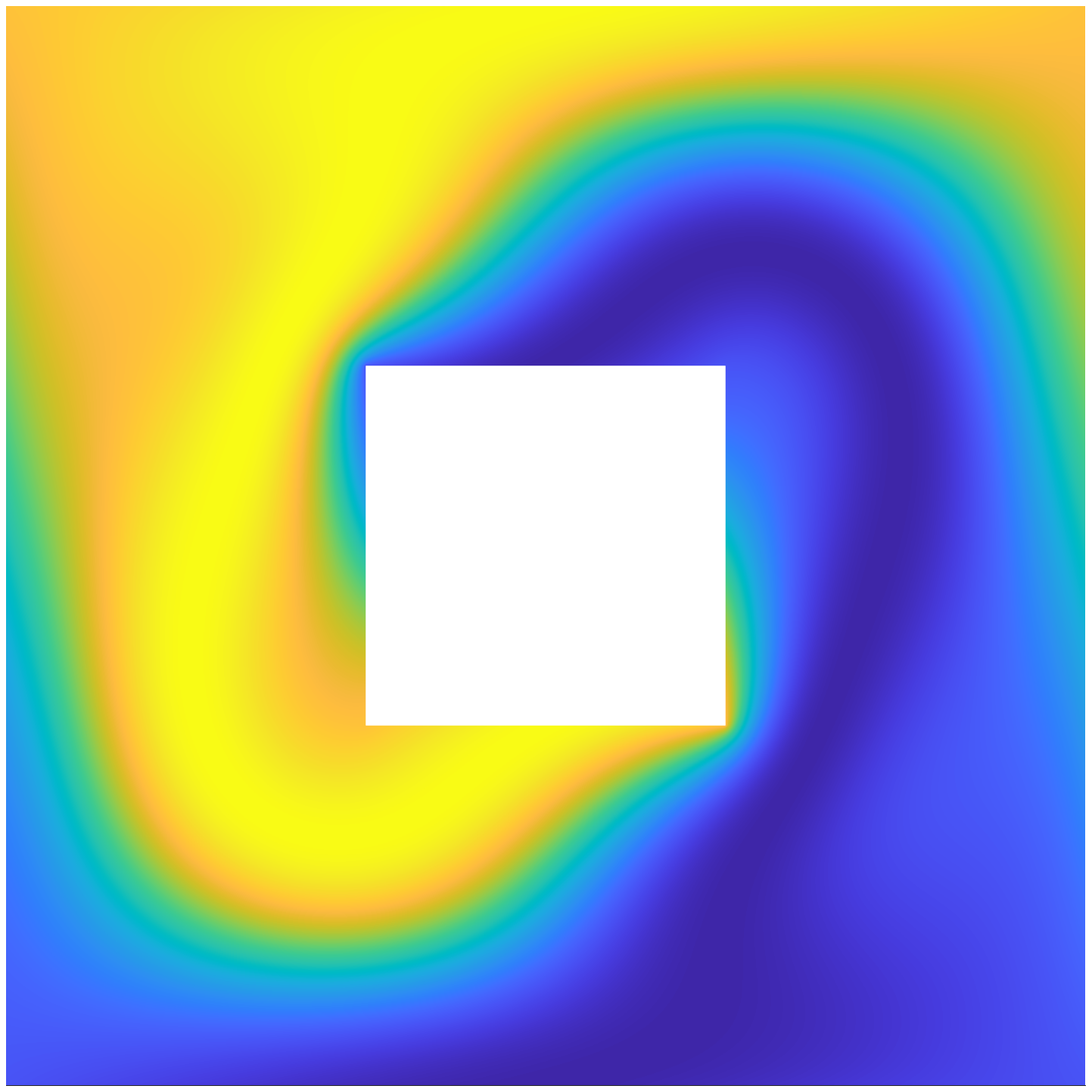}}
\\
\subfloat[$Q_{11}$]{\includegraphics[width=0.2\linewidth]{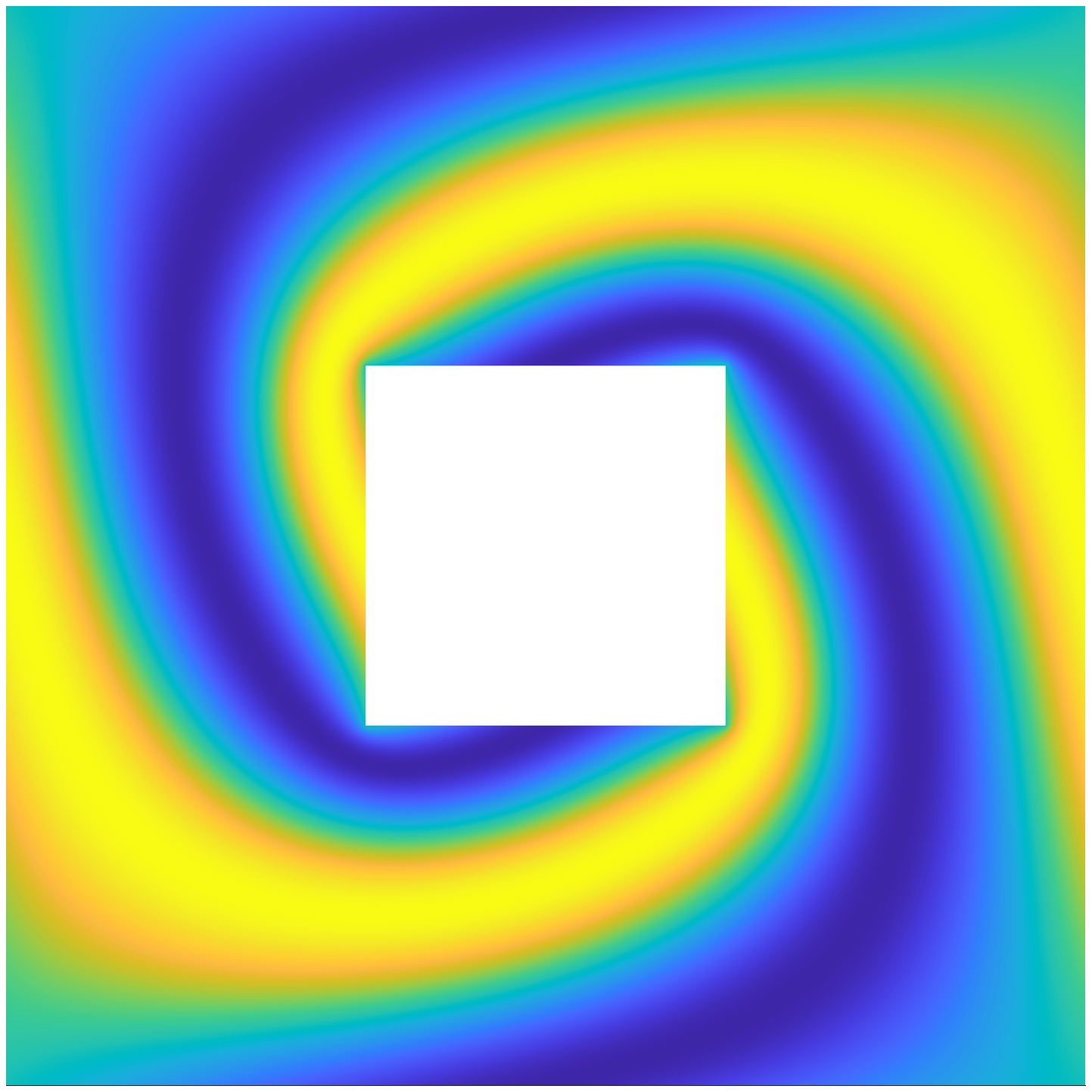}}\hspace{0.3 cm}
\subfloat[$Q_{12}$]{\includegraphics[width=0.2\linewidth]{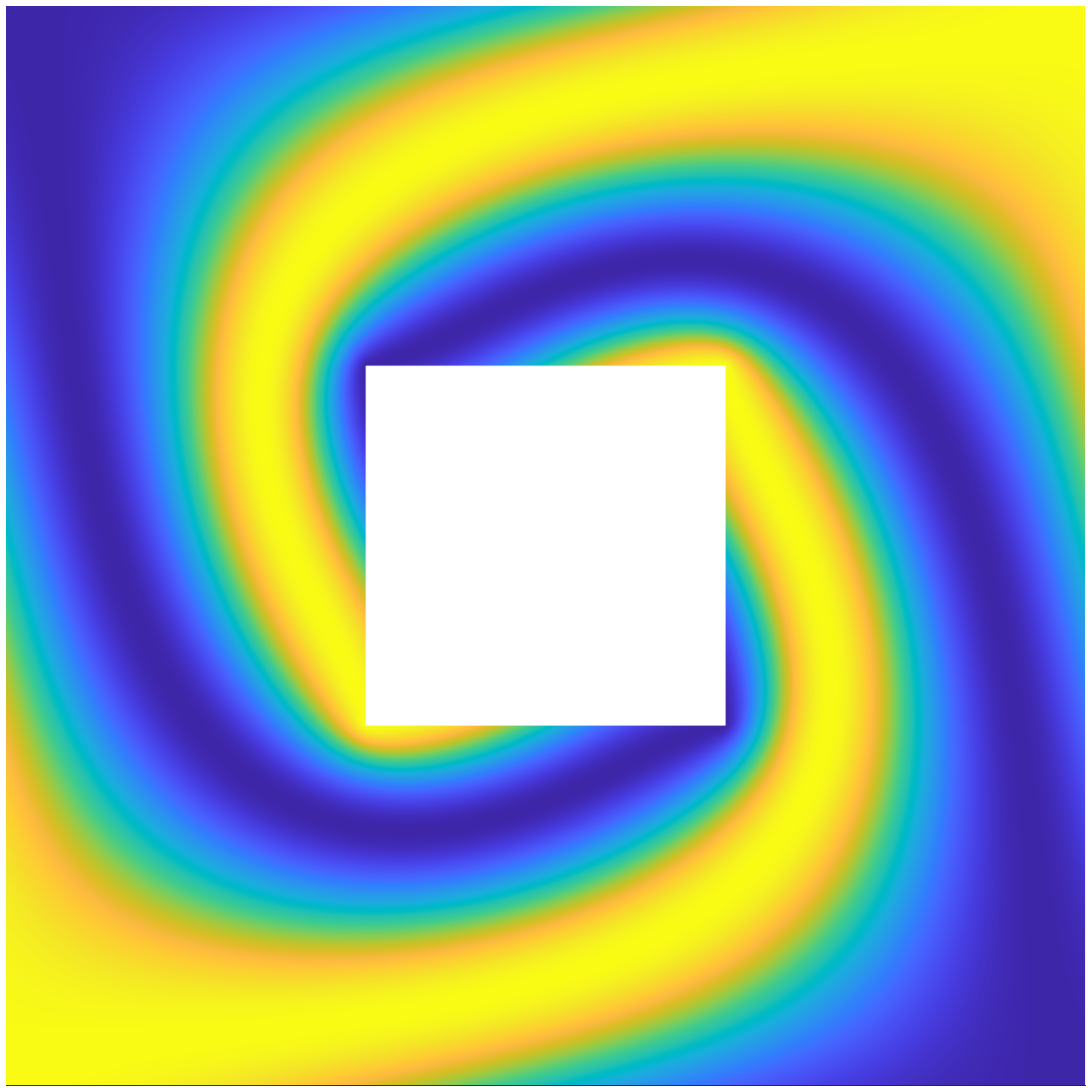}}\hspace{0.3 cm}
\subfloat[$M_1$]{\includegraphics[width=0.2\linewidth]{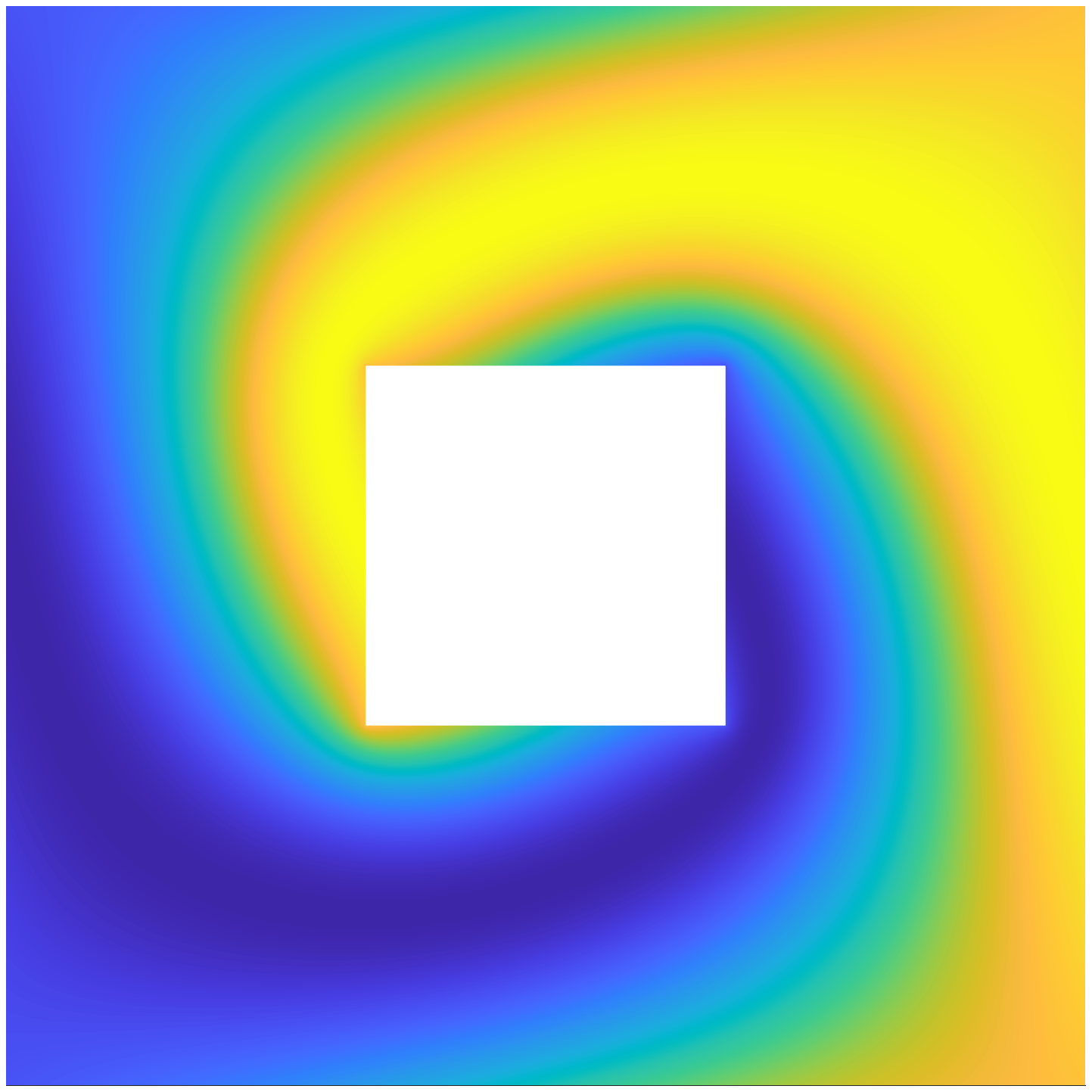}}\hspace{0.3 cm}
\subfloat[$M_2$]{\includegraphics[width=0.2\linewidth]{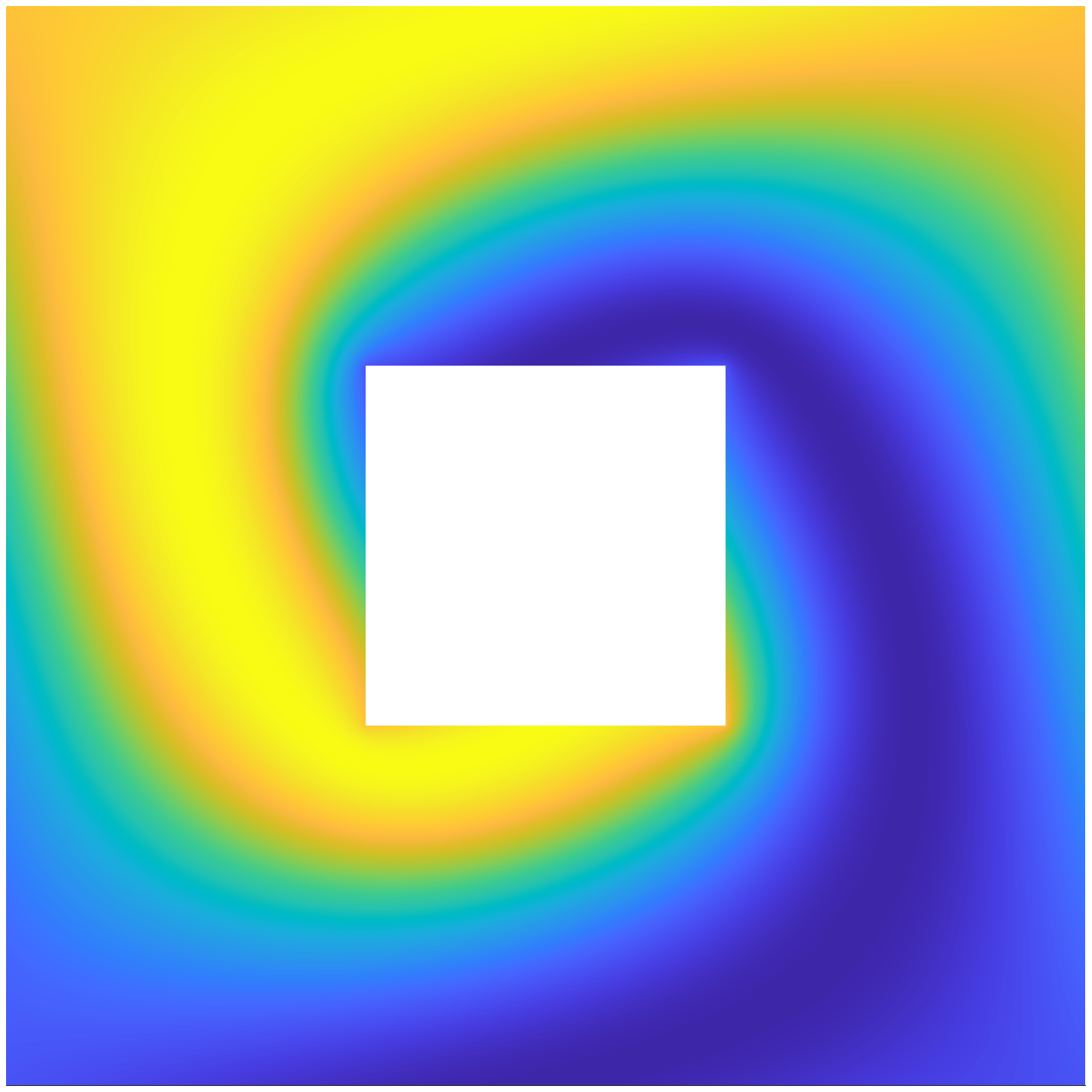}}
\caption{Test case with multiple solutions. Initial guess  $\Psi^{0,+}_{\calT_5}$ (first row) and converged solution $\Psi^{\rm cv}_{\calT_5}$ (second row); Initial  guess  $\Psi^{0,-}_{\calT_5}$ (third  row) and converged solution $\Psi^{\rm cv}_{\calT_5}$ (fourth row). The perturbation parameter for the initial guess is $a=1$.}
\label{fig:Ex4-solutions}
\end{figure}

Table~\ref{table:ener-psiplus-Ex4} reports the discrete energies $E(\Psi^j_{\calT_i})$ for the outer iteration index $j\in\{0,\ldots,5\}$ and the mesh index $i\in\{1,\ldots,5\}$. We also report the discrete energy at convergence, the total number of Uzawa/ADMM iterations (in units of $10^3$) together with the number of outer iterations (between parentheses), and the final adaptive inner tolerance. Here, we use $a:=1$ to compute the initial guess, we use $\Psi^{0,-}_{\calT_i}$ for the table on the left and $\Psi^{0,+}_{\calT_i}$ for the table on the right. We observe an energy decrease in all cases, and a clear improvement as the mesh is refined. Moreover, the choice of sign in the initial condition determines the limit reached by the outer iteration, either $\Psi_+$ or $\Psi_-$. We note that the energies at convergence of the two runs coincide in all displayed digits on each of the five meshes.

\ifHAL \begin{table}[htb] \else \begin{table}[width=\linewidth,cols=12,pos=htb] \fi 
\centering
\caption{Test case with multiple solutions. Discrete energy $E(\Psi^j_{\calT_i})$ for the outer iteration index $j\in\{0,\ldots,5\}$ and the mesh index $i\in\{1,\ldots,5\}$. The energy at convergence is also reported. The two bottom lines of each table report the total number of Uzawa/ADMM iterations (in units of $10^3$) together with the number of outer iterations (in parentheses) and the final adaptive inner tolerance. Left: initial guess is $\Psi^{0,-}_{\calT_i}$; right: initial guess is $\Psi^{0,+}_{\calT_i}$  (right); the perturbation parameter is $a=1$ in both cases.}
\label{table:ener-psiplus-Ex4}
\begin{minipage}[t]{0.49\linewidth}
\centering
\resizebox{\linewidth}{!}{%
\ifHAL
\begin{tabular}{|c|rrrrr|}
\hline
$j\,\backslash\,i$ & \multicolumn{1}{c}{$1$} & \multicolumn{1}{c}{$2$} & \multicolumn{1}{c}{$3$} & \multicolumn{1}{c}{$4$} & \multicolumn{1}{c|}{$5$} \\
\hline
\else
\begin{tabular}{@{}c rrrrr@{}}
\toprule
$j\,\backslash\,i$ & \multicolumn{1}{c}{$1$} & \multicolumn{1}{c}{$2$} & \multicolumn{1}{c}{$3$} & \multicolumn{1}{c}{$4$} & \multicolumn{1}{c}{$5$} \\
\midrule
\fi
0  & 204.055 & 217.695 & 221.046 & 221.832 & 221.989 \\
1  & 185.694 & 195.963 & 198.140 & 198.588 & 198.652 \\
2  & 176.618 & 184.579 & 185.955 & 186.192 & 186.216 \\
3  & 171.774 & 178.127 & 179.075 & 179.218 & 179.227 \\
4  & 168.974 & 174.428 & 175.254 & 175.380 & 175.389 \\
5  & 167.283 & 172.346 & 173.199 & 173.338 & 173.351 \\
\ifHAL \hline \fi
cv & 164.145 & 169.851 & 170.940 & 171.134 & 171.153 \\
\ifHAL \hline \fi
Iterations & 0.70\,(58) & 0.55\,(30) & 0.96\,(26) & 3.21\,(24) & 21.14\,(27) \\
\ifHAL \hline \fi
Tolerance & $10^{-6}$ & $10^{-6}$ & $10^{-6}$ & $10^{-7}$ & $10^{-9}$ \\
\ifHAL \hline \fi
\end{tabular}
}
\end{minipage}
\hfill
\begin{minipage}[t]{0.49\linewidth}
\centering
\resizebox{\linewidth}{!}{%
\ifHAL
\begin{tabular}{|c|rrrrr|}
\hline
$j\,\backslash\,i$ & \multicolumn{1}{c}{$1$} & \multicolumn{1}{c}{$2$} & \multicolumn{1}{c}{$3$} & \multicolumn{1}{c}{$4$} & \multicolumn{1}{c|}{$5$} \\
\hline
\else
\begin{tabular}{@{}c rrrrr@{}}
\toprule
$j\,\backslash\,i$ & \multicolumn{1}{c}{$1$} & \multicolumn{1}{c}{$2$} & \multicolumn{1}{c}{$3$} & \multicolumn{1}{c}{$4$} & \multicolumn{1}{c}{$5$} \\
\midrule
\fi
0  & 204.055 & 217.695 & 221.046 & 221.832 & 221.989 \\
1  & 185.694 & 195.963 & 198.140 & 198.588 & 198.652 \\
2  & 176.618 & 184.579 & 185.955 & 186.192 & 186.216 \\
3  & 171.774 & 178.127 & 179.075 & 179.218 & 179.227 \\
4  & 168.974 & 174.428 & 175.254 & 175.380 & 175.389 \\
5  & 167.283 & 172.346 & 173.199 & 173.338 & 173.351 \\
\ifHAL \hline \fi
cv & 164.145 & 169.851 & 170.940 & 171.134 & 171.153 \\
\ifHAL \hline \fi
Iterations & 0.70\,(58) & 0.55\,(30) & 0.96\,(26) & 3.21\,(24) & 21.14\,(27) \\
\ifHAL \hline \fi
Tolerance & $10^{-6}$ & $10^{-6}$ & $10^{-6}$ & $10^{-7}$ & $10^{-9}$  \\
\ifHAL \hline \fi
\end{tabular}
}
\end{minipage}
\end{table}

\begin{figure}[htb]\centering
\subfloat[$\calT_1$]{\includegraphics[width=8cm,height=6cm]{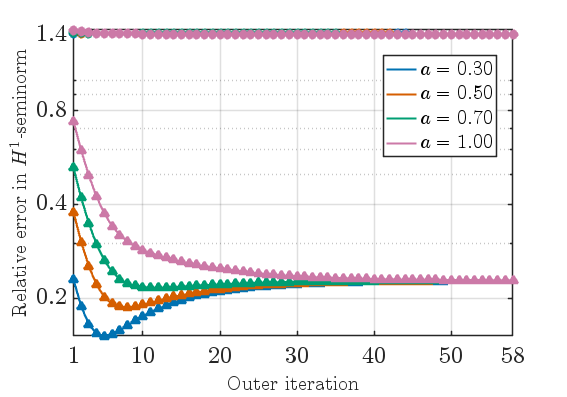}}
\subfloat[$\calT_1$]{\includegraphics[width=8cm,height=6cm]{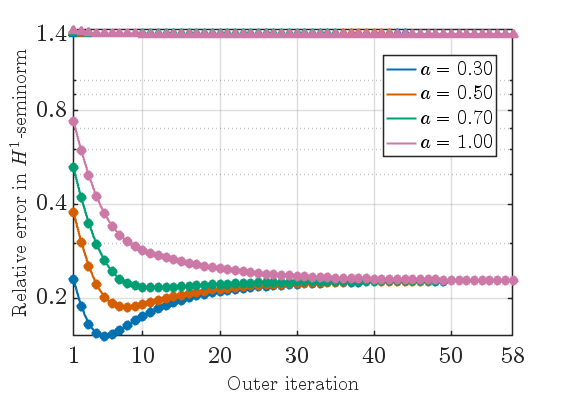}}
\\
\subfloat[$\calT_5$]{\includegraphics[width=8cm,height=6cm]{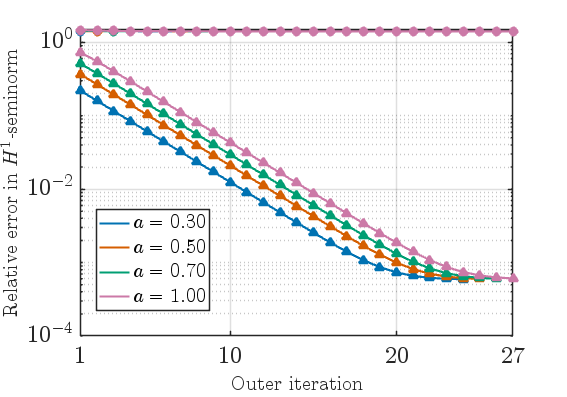}}
\subfloat[$\calT_5$]{\includegraphics[width=8cm,height=6cm]{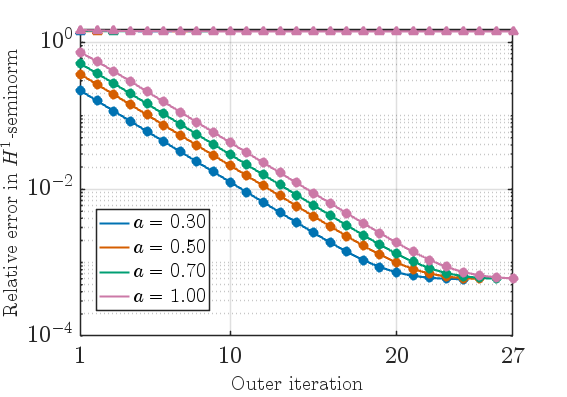}}
\caption{Test case with multiple solutions. Relative errors in the $H^1$-seminorm between $\Psi^j_{\calT_i}$  and the two reference solutions $\Psi_+$ (circles) and $\Psi_-$ (triangles) on $\calT_1$ (first row) and $\calT_5$ (second row)  as a function of the outer iteration index for an initial guess constructed using the perturbation parameter $a\in\{0.3,0.5,0.7,1\}$.  Left column: initial guess $\Psi^{0,-}_{\calT_i}$; right column: initial guess $\Psi^{0,+}_{\calT_i}$ for the  corresponding mesh index $i\in\{1,5\}$.}
\label{fig:Ex4-dist}
\end{figure}

Figure~\ref{fig:Ex4-dist} displays the relative errors in the $H^1$-seminorm between the outer iterates $\Psi^j_{\calT_i}$ and the ($\mathbb{P}_1$-finite element approximations) of the two reference solutions $\Psi_+$ (circles) and $\Psi_-$ (triangles) on $\calT_1$ and on $\calT_5$, starting from an initial guess with perturbation parameter $a\in\{0.3,0.5,0.7,1\}$. The left column corresponds to the initial guess $\Psi^{0,-}_{\calT_i}$ and the right column to $\Psi^{0,+}_{\calT_i}$. In all cases, the error with respect to the minimizer of the sign opposite to that of the initial guess remains close to $1.41$, which is the relative distance in the $H^1$-seminorm between $\Psi_+$ and $\Psi_-$. In contrast, the other error converges to a smaller value resulting from the space discretization error. Hence, the outer iteration always converges to an approximation of the minimizer selected by the sign in the initial guess. The value of $a$ only affects the initial error and, slightly, the number of outer iterations. Independently of the value of $a$, the four runs reach the same reference solution. This result indicates that the overall algorithm remains fairly robust with respect to perturbations of the initial guess in the case of multiple solutions. 

\ifHAL \FloatBarrier \fi

\section*{Acknowledgements}        
R. M. acknowledges the support of Department of Scienceand Technology, India (INSPIRE Faculty Fellowship, IFA24-MA 208).
\ifHAL
\bibliographystyle{amsplain}
\fi   
\bibliography{References_LDG}

\end{document}